\documentclass[12pt]{amsart}
\usepackage{amsmath,amsfonts,amssymb,amsthm}
\usepackage[shortalphabetic,abbrev,nobysame]{amsrefs}
\usepackage{mathrsfs}
\usepackage[all]{xy}
 \usepackage{relsize}
\usepackage{hyperref}
\usepackage{accents}
\usepackage{color}
\usepackage{soul}
\allowdisplaybreaks

\DeclareMathOperator{\Res}{Res}
\DeclareMathOperator{\Hom}{Hom}

\DeclareMathOperator{\End}{End}

\begin{document}

\newtheorem{thm}{Theorem}[section]
\newtheorem{prop}[thm]{Proposition}
\newtheorem{coro}[thm]{Corollary}
\newtheorem{conj}[thm]{Conjecture}
\newtheorem{example}[thm]{Example}
\newtheorem{lem}[thm]{Lemma}
\newtheorem{rem}[thm]{Remark}
\newtheorem{hy}[thm]{Hypothesis}
\newtheorem*{acks}{Acknowledgements}
\theoremstyle{definition}
\newtheorem{de}[thm]{Definition}
\newtheorem{ex}[thm]{Example}

\newtheorem{convention}[thm]{Convention}

\newtheorem{bfproof}[thm]{{\bf Proof}}
\xymatrixcolsep{5pc}

\newcommand{\C}{{\mathbb{C}}}
\newcommand{\Z}{{\mathbb{Z}}}
\newcommand{\N}{{\mathbb{N}}}
\newcommand{\Q}{{\mathbb{Q}}}
\newcommand{\te}[1]{\textnormal{{#1}}}
\newcommand{\set}[2]{{
    \left.\left\{
        {#1}
    \,\right|\,
        {#2}
    \right\}
}}
\newcommand{\sett}[2]{{
    \left\{
        {#1}
    \,\left|\,
        {#2}
    \right\}\right.
}}

\newcommand{\choice}[2]{{
\left[
\begin{array}{c}
{#1}\\{#2}
\end{array}
\right]
}}
\def \<{{\langle}}
\def \>{{\rangle}}

\def\({\left(}
\def\){\right)}

\newcommand{\overit}[2]{{
    \mathop{{#1}}\limits^{{#2}}
}}
\newcommand{\belowit}[2]{{
    \mathop{{#1}}\limits_{{#2}}
}}

\newcommand{\wt}[1]{\widetilde{#1}}

\newcommand{\wh}[1]{\widehat{#1}}

\newcommand{\no}[1]{{
    \mathopen{\overset{\circ}{
    \mathsmaller{\mathsmaller{\circ}}}
    }{#1}\mathclose{\overset{\circ}{\mathsmaller{\mathsmaller{\circ}}}}
}}

\newcommand{\nob}[1]{{
    \mathopen{\overset{\bullet}{
    \mathsmaller{\mathsmaller{\bullet}}}
    }{#1}\mathclose{\overset{\bullet}{\mathsmaller{\mathsmaller{\bullet}}}}
}}

\newlength{\dhatheight}
\newcommand{\dwidehat}[1]{%
    \settoheight{\dhatheight}{\ensuremath{\widehat{#1}}}%
    \addtolength{\dhatheight}{-0.45ex}%
    \widehat{\vphantom{\rule{1pt}{\dhatheight}}%
    \smash{\widehat{#1}}}}
\newcommand{\dhat}[1]{%
    \settoheight{\dhatheight}{\ensuremath{\hat{#1}}}%
    \addtolength{\dhatheight}{-0.35ex}%
    \hat{\vphantom{\rule{1pt}{\dhatheight}}%
    \smash{\hat{#1}}}}

\newcommand{\dwh}[1]{\dwidehat{#1}}

\newcommand{\ck}[1]{\check{#1}}

\newcommand{\dis}{\displaystyle}

\newcommand{\pd}[1]{\frac{\partial}{\partial {#1}}}

\newcommand{\pdiff}[2]{\frac{\partial^{#2}}{\partial #1^{#2}}}


\newcommand{\g}{{\frak g}}
\newcommand{\fg}{\g}
\newcommand{\ff}{{\frak f}}
\newcommand{\f}{\ff}
\newcommand{\gc}{{\bar{\g'}}}
\newcommand{\h}{{\frak h}}
\newcommand{\cent}{{\frak c}}
\newcommand{\notc}{{\not c}}
\newcommand{\Loop}{{\mathcal L}}
\newcommand{\G}{{\mathcal G}}
\newcommand{\D}{\mathcal D}
\newcommand{\Y}{\mathcal Y}
\newcommand{\T}{\mathcal T}
\newcommand{\Free}{\mathcal F}
\newcommand{\Cfk}{\mathcal C}
\newcommand{\nil}{\mathfrak n}
\newcommand{\al}{\alpha}
\newcommand{\alck}{\al^\vee}
\newcommand{\be}{\beta}
\newcommand{\beck}{\be^\vee}
\newcommand{\ssl}{{\mathfrak{sl}}}
\newcommand{\id}{\te{id}}
\newcommand{\rtu}{{\xi}}
\newcommand{\period}{{N}}
\newcommand{\half}{{\frac{1}{2}}}
\newcommand{\quar}{{\frac{1}{4}}}
\newcommand{\oct}{{\frac{1}{8}}}
\newcommand{\hex}{{\frac{1}{16}}}
\newcommand{\reciprocal}[1]{{\frac{1}{#1}}}
\newcommand{\inverse}{^{-1}}
\newcommand{\inv}{\inverse}
\newcommand{\SumInZm}[2]{\sum\limits_{{#1}\in\Z_{#2}}}
\newcommand{\uce}{{\mathfrak{uce}}}
\newcommand{\Rcat}{\mathcal R}


\newcommand{\orb}[1]{|\mathcal{O}({#1})|}
\newcommand{\up}{_{(p)}}
\newcommand{\uq}{_{(q)}}
\newcommand{\upq}{_{(p+q)}}
\newcommand{\uz}{_{(0)}}
\newcommand{\uk}{_{(k)}}
\newcommand{\nsum}{\SumInZm{n}{\period}}
\newcommand{\ksum}{\SumInZm{k}{\period}}
\newcommand{\overN}{\reciprocal{\period}}
\newcommand{\df}{\delta\left( \frac{\xi^{k}w}{z} \right)}
\newcommand{\dfl}{\delta\left( \frac{\xi^{\ell}w}{z} \right)}
\newcommand{\ddf}{\left(D\delta\right)\left( \frac{\xi^{k}w}{z} \right)}

\newcommand{\ldfn}[1]{{\left( \frac{1+\xi^{#1}w/z}{1-{\xi^{#1}w}/{z}} \right)}}
\newcommand{\rdfn}[1]{{\left( \frac{{\xi^{#1}w}/{z}+1}{{\xi^{#1}w}/{z}-1} \right)}}
\newcommand{\ldf}{{\ldfn{k}}}
\newcommand{\rdf}{{\rdfn{k}}}
\newcommand{\ldfl}{{\ldfn{\ell}}}
\newcommand{\rdfl}{{\rdfn{\ell}}}

\newcommand{\kprod}{{\prod\limits_{k\in\Z_N}}}
\newcommand{\lprod}{{\prod\limits_{\ell\in\Z_N}}}
\newcommand{\E}{{\mathcal{E}}}
\newcommand{\F}{{\mathcal{F}}}

\newcommand{\SY}{{\mathcal{S}}}
\newcommand{\Etopo}{{\mathcal{E}_{\te{topo}}}}

\newcommand{\Ye}{{\mathcal{Y}_\E}}

\newcommand{\rh}{{{\bf h}}}
\newcommand{\rp}{{{\bf p}}}
\newcommand{\rrho}{{{\pmb \varrho}}}
\newcommand{\ral}{{{\pmb \al}}}

\newcommand{\comp}{{\mathfrak{comp}}}
\newcommand{\ctimes}{{\widehat{\boxtimes}}}
\newcommand{\ptimes}{{\widehat{\otimes}}}
\newcommand{\ptimeslt}{{
{}_{\te{t}}\ptimes
}}
\newcommand{\ptimesrt}{{\ot_{\te{t}} }}
\newcommand{\ttp}[1]{{
    {}_{{#1}}\ptimes
}}
\newcommand{\bigptimes}{{\widehat{\bigotimes}}}
\newcommand{\bigptimeslt}{{
{}_{\te{t}}\bigptimes
}}
\newcommand{\bigptimesrt}{{\bigptimes_{\te{t}} }}
\newcommand{\bigttp}[1]{{
    {}_{{#1}}\bigptimes
}}

\newcommand{\ot}{\otimes}
\newcommand{\Ot}{\bigotimes}

\newcommand{\affva}[1]{V_{\wh\g}\(#1,0\)}
\newcommand{\saffva}[1]{L_{\wh\g}\(#1,0\)}
\newcommand{\saffmod}[1]{L_{\wh\g}\(#1\)}

\newcommand{\otcopies}[2]{\belowit{\underbrace{{#1}\ot \cdots \ot {#1}}}{{#2}\te{-times}}}

\newcommand{\twotcopies}[3]{\belowit{\underbrace{{#1}\wh\ot_{#2} \cdots \wh\ot_{#2} {#1}}}{{#3}\te{-times}}}


\newcommand{\tar}{{\mathcal{DY}}_0\(\mathfrak{gl}_{\ell+1}\)}
\newcommand{\U}{{\mathcal{U}}}
\newcommand{\V}{{\mathcal{V}}}
\newcommand{\htar}{\mathcal{DY}_\hbar\(A\)}
\newcommand{\hhtar}{\widetilde{\mathcal{DY}}_\hbar\(A\)}
\newcommand{\htarz}{\mathcal{DY}_0\(\mathfrak{gl}_{\ell+1}\)}
\newcommand{\hhtarz}{\widetilde{\mathcal{DY}}_0\(A\)}
\newcommand{\qhei}{\U_\hbar\left(\hat{\h}\right)}
\newcommand{\n}{{\mathfrak{n}}}
\newcommand{\vac}{{{\bf 1}}}
\newcommand{\vtar}{{{
    \mathcal{V}_{\hbar,\tau}\left(\ell,0\right)
}}}

\newcommand{\qtar}{
    \U_q\(\wh\g_\mu\)}
\newcommand{\rk}{{\bf k}}

\newcommand{\hctvs}[1]{Hausdorff complete linear topological vector space}
\newcommand{\hcta}[1]{Hausdorff complete linear topological algebra}
\newcommand{\ons}[1]{open neighborhood system}
\newcommand{\B}{\mathcal{B}}
\newcommand{\rx}{{\bf x}}
\newcommand{\re}{{\bf e}}
\newcommand{\bk}{{\bf k}}
\newcommand{\rphi}{{\boldsymbol{ \phi}}}

\newcommand{\der}{\mathcal D}


\makeatletter
\renewcommand{\BibLabel}{%
    \Hy@raisedlink{\hyper@anchorstart{cite.\CurrentBib}\hyper@anchorend}%
    [\thebib]%
}
\@addtoreset{equation}{section}
\def\theequation{\thesection.\arabic{equation}}
\makeatother \makeatletter


\title[Quantum affine VA]{Double Yangians and quantum vertex algebras, I}

\author{Fei Kong$^1$}
\address{Key Laboratory of Computing and Stochastic Mathematics (Ministry of Education), School of Mathematics and Statistics, Hunan Normal University, Changsha, China 410081} \email{kongmath@hunnu.edu.cn}
\thanks{$^1$Partially supported by NSF of China (No. 12371027, No. 12471029).}

\author{Haisheng Li}
\address{Department of Mathematical Sciences, Rutgers University, Camden, NJ 08102}
\email{hli@camden.rutgers.edu}


\begin{abstract}
For any symmetrizable generalized Cartan matrix $A$, we introduce an algebra $\wh{\mathcal{DY}}(A)$,
which is essentially the centrally extended double Yangian when $A$ is of finite type, and we give
a new field (current) presentation of $\wh{\mathcal{DY}}(A)$.
Among the main results, for any $\ell\in \C$ we construct a universal vacuum $\wh{\mathcal{DY}}(A)$-module
$\mathcal{V}_A(\ell)$ of level $\ell$,
prove that there exists a natural $\hbar$-adic weak quantum vertex algebra structure on $\mathcal{V}_A(\ell)$,
and give an isomorphism between the category of restricted $\wh{\mathcal{DY}}(A)$-modules of level $\ell$
and the category of $\mathcal{V}_A(\ell)$-modules.
\end{abstract}
\maketitle

\section{Introduction}
It is well known that vertex algebras are naturally associated to affine Lie algebras, giving an important class of examples.
In the general vertex algebra theory, an important subject arguably is establishing and exploring natural associations
of quantum vertex algebras (in some sense) to quantum affine algebras and double Yangians.

As a fundamental work, Etingof and Kazhdan (see \cite{EK-qva}) introduced a theory of
quantum vertex operator algebras and constructed a family of examples as formal
deformations of the (universal) affine vertex algebras of type $A$. Since then,
this theory has been extensively studied and developed in various directions.

In \cite{Li-nonlocal,Li-qva2}, notions of nonlocal vertex algebra and (weak) quantum vertex algebra
were introduced and systematically studied, where nonlocal vertex algebras are vertex algebra analogues of
noncommutative associative algebras while (weak) quantum vertex algebras are generalizations of vertex superalgebras
(including vertex algebras).
As slight generalizations of quantum vertex operator algebras in the sense of Etingof-Kazhdan,
 $\hbar$-adic (weak) quantum vertex algebras, which are formal deformations of (weak) quantum vertex algebras,
  were studied in \cite{Li-h-adic}.
 In particular, a conceptual construction of $\hbar$-adic weak quantum vertex algebras
and their modules was established by using ``$\SY$-local'' sets of fields (vertex operators).
 As an application, $\hbar$-adic quantum vertex algebras were
 associated to the centrally extended double Yangian of $\mathfrak{sl}_2$.

As for quantum affine algebras, a theory of what were called $\phi$-coordinated (quasi) modules
 for (weak) quantum vertex algebras was developed in \cite{Li-phi-coor} and
 a conceptual association of quantum vertex algebras to quantum affine algebras
 in terms of $\phi$-coordinated (quasi) modules was obtained.
 (Quantum vertex algebras were {\em explicitly} associated to
 some other algebras, e.g., to the deformed Virasoro algebra $\mathcal{V}_{p,q}$ with $q=-1$ in \cite{Li-phi-jmp}
 and to the Ding-Iohara algebra with $c=0$ in \cite{Li-Tan-Wang}.)

On the other hand, Butorac, Jing and Ko\v{z}i\'{c} (see \cite{BJK-qva-BCD}) extended Etingof-Kazhdan's construction
to types $B$, $C$ and $D$ (with rational $R$-matrices),
by using the $R$-matrix type realization of double Yangians in \cite{JYL-R-mat-DY}.
It was proved in \cite{K-qva-phi-mod-BCD} that modules for these quantum vertex operator algebras are in one-to-one
correspondence with restricted modules for the corresponding double Yangians.
Recently,  by utilizing the $R$-matrix presentation of quantum affine algebras
(see \cite{DF-qaff-RTT-Dr,JLM-qaff-RTT-Dr-BD,JLM-qaff-RTT-Dr-C}),
Ko\v{z}i\'{c}  (see \cite{Kozic-qva-tri-A, K-qva-phi-mod-BCD}) constructed the quantum vertex operator algebras associated with trigonometric
$R$-matrices of types $A$, $B$, $C$ and $D$,
and established a one-to-one correspondence between $\phi$-coordinated modules and
restricted modules for quantum affine algebras.

In a previous paper \cite{JKLT-22} joint with Jing and Tan, by using the smash product construction of
nonlocal vertex algebras (see \cite{Li-smash}), we constructed a family of $\hbar$-adic quantum vertex algebras
by explicitly deforming lattice vertex algebras. 
It is remarkable that Drinfeld's quantum affinization process can be extended to general symmetrizable quantum Kac-Moody algebras (\cite{GKV,J-KM,Naka-quiver,CJKT}).
In \cite{Kong-23}, mainly by deforming affine vertex algebras
the first named author of this paper constructed a family of $\hbar$-adic quantum affine vertex algebras
and their $\phi$-coordinated quasi modules,
based on Drinfeld presentations.


In this current paper, we continue studying centrally extended double Yangians
in the context of $\hbar$-adic quantum vertex algebras and their modules, using the Drinfeld-like generator-and-defining-relation approach.
We first introduce an algebra $\wh{\mathcal{DY}}(A)$ for any symmetrizable generalized Cartan matrix $A$,
which essentially coincides with the centrally extended double Yangian when $A$ is of finite type, and
we then give a new field (current) presentation of $\wh{\mathcal{DY}}(A)$.
Among the main results, for any $\ell\in \C$, we construct a universal vacuum $\wh{\mathcal{DY}}(A)$-module
$\mathcal{V}_A(\ell)$ of level $\ell$,
prove that there exists a natural $\hbar$-adic weak quantum vertex algebra structure on $\mathcal{V}_A(\ell)$, and
give an isomorphism between the category of restricted $\wh{\mathcal{DY}}(A)$-modules of level $\ell$
and the category of $\mathcal{V}_A(\ell)$-modules.

In the following, we discuss some technical details and mention some relevant results.
Note that  for any (possibly infinite-dimensional) Lie algebra $\g$ equipped with a non-degenerate invariant bilinear form,
 a general affine Lie algebra $\widehat{\g}$ can be defined explicitly in the sense that $\widehat{\g}$
 is linearly spanned by the standard central element ${\bf k}$ and the coefficients of the generating functions
 $a(z)$ for $a\in \g$, and commutators $[a(x),b(z)]$ are ``linearly closed.''
Unlike affine Lie algebras, centrally extended double Yangians (and as well as quantum affine algebras) are defined
by Chevalley generators and relations, including Serre relations for the high rank case.
As for $\mathfrak{sl}_2$, the centrally extended double Yangian  is defined
in terms of four generating functions $E(z), F(z), H^{+}(z), H^{-}(z)$.
In \cite{Li-h-adic}, these original currents were used as fields to generate an $\hbar$-adic quantum vertex algebra,
especially, with $H^{+}(z)$ (annihilation) and $H^{-}(z)$ (creation) considered as two fields.
Consequently, the resulted $\hbar$-adic quantum vertex algebra is a formal deformation of the affine vertex algebra
associated to a $4$-dimensional Lie algebra. Naturally, one hopes to have an $\hbar$-adic quantum vertex algebra
which is a formal deformation of the affine vertex algebra associated to $\mathfrak{sl}_2$.

With this as one of the motivations, for the algebra $\wh{\mathcal{DY}}(A)$ we introduce a new set of generating fields
denoted by $x_{i,\Y}^{\pm}(z), h_{i,\Y}(z)$ for $i\in I$,
where in particular the modified Cartan annihilation and creation currents are combined.
Then we give a characterization of $\wh{\mathcal{DY}}(A)$ in terms of this new set of fields.
In the new presentation, the same vertex-operator Serre relations are proved to hold, and
 ``$\SY$-commutators'' $[x_{i,\Y}^{+}(z_1),x_{i,\Y}^{-}(z_2)]_{\SY}$ replace the usual Lie bracket.
For proving the equivalence, we give two characterizations of the vertex-operator Serre relation.
On the other hand, to deal with the $\mathcal{W}$-algebra type nonlinear relations in
the $\SY$-commutators, we prove and use a vertex-operator iterate formula.
Just as with the association of vertex algebras to affine Lie algebras, we need to construct
``universal vacuum modules'' for $\wh{\mathcal{DY}}(A)$ and show they have a natural
$\hbar$-adic quantum vertex algebra structure.
While the construction of universal vacuum modules $\mathcal{V}_A(\ell)$ is mostly tautological,
to get topologically free $\C[[\hbar]]$-modules in several stages,
we make use of a new gadget---notions of strong submodule and ideal.

In a sequel,  we plan to determine the structure of the universal vacuum $\wh{\mathcal{DY}}(A)$-module $\mathcal{V}_A(\ell)$
and hopefully show that ${\mathcal{V}}_A(\ell)/\hbar {\mathcal{V}}_A(\ell)$ is a universal vacuum $\wh{\g}$-module of level $\ell$
with $\g=\g(A)$ for $A$ of a finite type.
 This implies that the $\hbar$-adic weak quantum vertex algebra ${\mathcal{V}}_A(\ell)$
 is a non-degenerate quantum vertex operator algebra in the sense of Etingof and Kazhdan.

This paper is organized as follows. Section 2 is preliminary, in which we recall some basics on
nonlocal vertex algebras and (weak) quantum vertex algebras.
In Section 3, we recall some basic results on $\hbar$-adic (weak) quantum vertex algebras.
 In Section 4, we establish several technical results including a vertex operator iterate formula and
 equivalent characterizations of the vertex-operator Serre relation.
 In Section \ref{sec:DY}, we study algebra $\wh{\mathcal{DY}}(A)$ and $\hbar$-adic weak quantum vertex algebra $\V_A(\ell)$,
and present the natural connection of  $\wh{\mathcal{DY}}(A)$-modules with
$\V_A(\ell)$-modules.

In this paper, with $\C$ (complex numbers) as our scalar field
we use the formal variable notations and conventions as established in \cite{FLM} and \cite{FHL}.
In particular,  $\C(x)$ and $\C(x,y)$ denote the fields of rational functions.
We also use $\partial_{x}$ for the formal differential operator $\frac{\partial}{\partial x}$.
In addition, we use $\Z_+$ and $\N$ for the sets of positive integers and nonnegative integers, respectively.

\section{Vertex algebras and quantum vertex algebras}\label{sec:cla-va}

In this section, we review some basic notions and results,
to introduce some frequently used terminologies and notations.

We begin by recalling the definition of a vertex algebra.
A {\em vertex algebra} is a vector space $V$ equipped with a linear map
\begin{align}
  Y(\cdot,x):\ & V\to  (\te{End}V)[[x,x^{-1}]]\nonumber\\
  &v\mapsto Y(v,x)=\sum_{n\in\Z}v_nx^{-n-1}\  \  (\te{where }v_n\in \te{End} V)
\end{align}
and equipped with a vector $\vac\in V$,
called the \emph{vacuum vector}, such that
\begin{align}
&Y(u,x)v\in V((x))\quad \text{for }u,v\in V,\\
  &Y(\vac,x)=1\ \  \ (\text{the identity operator on }V),\label{eq:vacuum-left}\\
  &Y(v,x)\vac\in V[[x]]\quad\te{and}\quad \lim_{x\to 0}Y(v,x)\vac=v\quad\te{for }v\in V,\label{eq:vacuum-right}
\end{align}
and such that the following \emph{Jacobi identity} holds for $u,v,w\in V$:
\begin{align}\label{eq:Jacobi}
x_0\inv\delta\!\(\frac{x_1-x_2}{x_0}\)&Y(u,x_1)Y(v,x_2)w
-x_0\inv\delta\!\(\frac{x_2-x_1}{-x_0}\)\!Y(v,x_2)Y(u,x_1)w\nonumber\\
&=x_1\inv\delta\!\(\frac{x_2+x_0}{x_1}\)\!Y(Y(u,x_0)v,x_2)w.
\end{align}

A consequence of the Jacobi identity is Borcherds' {\em commutator formula}
\begin{align}
[Y(u,x_1),Y(v,x_2)]=\sum_{n\ge 0}Y(u_nv,x_2)\frac{1}{n!}\left(\frac{\partial}{\partial x_2}\right)^n
x_1^{-1}\delta\!\left(\frac{x_2}{x_1}\right)\!.
\end{align}
This formula further implies the {\em weak commutativity,} namely {\em locality}:
For any $u,v\in V$, there exists $k\in \N$ such that
\begin{align}
(x_1-x_2)^kY(u,x_1)Y(v,x_2)=(x_1-x_2)^kY(v,x_2)Y(u,x_1).
\end{align}

The following analogue of the notion of (noncommutative) associative algebra
 (see \cite{BK},  \cite{Li-qva2}) plays a key role in the study of quantum vertex algebras:

\begin{de}
A \emph{nonlocal vertex algebra} is defined by using all the axioms that define a vertex algebra
except the Jacobi identity axiom which is replaced by the {\em weak associativity:}
For any $u,v,w\in V$, there exists $l\in \N$ such that
\begin{align}\label{nonlva-weak-asso}
(x_0+x_2)^lY(u,x_0+x_2)Y(v,x_2)w=(x_0+x_2)^lY(Y(u,x_0)v,x_2)w.
\end{align}
\end{de}

Let $V$ be a nonlocal vertex algebra. Denote by $\D$ the linear operator on $V$ given by
\begin{align}\label{D-def}
\D (v)=v_{-2}{\bf 1}=\left(\frac{d}{dz}Y(v,z){\bf 1}\right)|_{z=0}\quad \text{ for }v\in V.
\end{align}
The following are the basic properties:
\begin{align}
&[\D,Y(v,x)]=Y(\D (v),x)=\frac{d}{dx}Y(v,x),\label{D-bracket-der}\\
&e^{z\D}Y(v,x)e^{-z\D}=Y(e^{z\D}v,x)=Y(v,x+z),\label{D-exp}\\
&Y(v,x){\bf 1}=e^{x\D}v\quad \quad \text{ for }v\in V.\label{D-creation}
\end{align}

 An important class of nonlocal vertex algebras consists of what are called weak quantum vertex algebras (see \cite{Li-nonlocal}).

\begin{de}\label{weak-qva-def}
A {\em weak quantum vertex algebra} is defined by using all the axioms that define a vertex algebra $V$ except the Jacobi identity
which is replaced with the property that
for any $u,v\in V$, there exist (finitely many)
$$u^{(i)}, v^{(i)}\in V,\ f_i(x)\in \C((x))\ \ \text{ for }1\le i\le r$$
such that the following {\em $\SY$-Jacobi identity} holds:
\begin{align}\label{S-Jacobi-C}
&x_0\inv\delta\!\(\frac{x_1-x_2}{x_0}\)\!Y(u,x_1)Y(v,x_2)\\
&\quad -x_0\inv\delta\!\(\frac{x_2-x_1}{-x_0}\)\sum_{i=1}^rf_i(x_2-x_1)Y(v^{(i)},x_2)Y(u^{(i)},x_1)\nonumber\\
=\ & x_1\inv\delta\!\(\frac{x_2+x_0}{x_1}\)\!Y(Y(u,x_0)v,x_2).\nonumber
\end{align}
\end{de}

The $\SY$-Jacobi identity axiom above is equivalent to the combination of the weak associativity
and the {\em $\SY$-locality} (see \cite{EK-qva}): For any $u,v\in V$, there exist
$$u^{(i)}, v^{(i)}\in V,\ f_i(x)\in \C((x))\ \ \text{ for }1\le i\le r$$
and a nonnegative integer $k$ such that
\begin{align}\label{S-locality-1}
(x_1-x_2)^kY(u,x_1)Y(v,x_2)=(x_1-x_2)^k\sum_{i=1}^rf_i(x_2-x_1)Y(v^{(i)},x_2)Y(u^{(i)},x_1).
\end{align}

\begin{rem}
{\em Let $V$ be a weak quantum vertex algebra.
From definition, there is a linear map $\SY(x):\ V\ot V\rightarrow V\ot V\ot \C((x))$ such that for $u,v\in V$,
(\ref{S-Jacobi-C}) holds with
\begin{align}
\SY(x)(v\ot u)=\sum_{i=1}^r v^{(i)}\ot u^{(i)}\ot f_i(x)\in V\ot V\ot \C((x)).
\end{align}
Such a linear map $\SY(x)$ is called an {\em $\SY$-locality operator} of $V$.}
\end{rem}

Let $V$ be a general nonlocal vertex algebra. Following \cite{EK-qva}, let
\begin{align}
Y(x):\ V\ot V\rightarrow V((x))
\end{align}
denote the canonical linear map associated to the vertex-operator map $Y(\cdot,x)$ with
$$Y(x)(u\otimes v)=Y(u,x)v\ \ \te{ for }u,v\in V.$$

A {\em rational quantum Yang-Baxter operator} on a vector space $U$ is a linear map
$$\SY(x):\  U\ot U\rightarrow U\ot U\ot \C((x)),$$
satisfying the following \emph{rational quantum Yang-Baxter equation}:
\begin{align}\label{eq:qyb}
  \SY^{12}(z)\SY^{13}(z+w)\SY^{23}(w)=\SY^{23}(w)\SY^{13}(z+w)\SY^{12}(z)
\end{align}
on $U\ot U \ot U$. It is said to be {\em unitary} if
\begin{align}\label{eq:qyb-unitary}
  \SY^{21}(x)\SY(-x)=1,
\end{align}
where $\SY^{21}(x)=\sigma \SY(x)\sigma$ with $\sigma$ denoting the flip operator on $U\ot U$.

 \begin{de}
 A {\em quantum vertex algebra} is a weak quantum vertex algebra $V$ equipped with
 a unitary rational quantum Yang-Baxter operator $\SY(x)$ on $V$,
 such that $\SY(x)$ is an $\SY$-locality operator of $V$ and such that
 \begin{align}
 &\hspace{1.5cm} \SY(x)(v\ot {\bf 1})=v\ot {\bf 1}\quad \text{ for }v\in V,\\
&  [\mathcal{D}\ot 1,\SY(x)]=-\frac{d}{dx}\SY(x)\quad \text{(the \emph{shift condition})}, \label{eq:qyb-shift}
\end{align}
\begin{align}\label{eq:qyb-hex-id}
    \SY(z)(Y(w)\ot 1)=(Y(w)\ot 1)\SY^{23}(z)\SY^{13}(z+w)\quad \text{ (the \emph{hexagon identity})}
  \end{align}
 on $V\ot V\ot V$.
 \end{de}

The following notion is due to Etingof and Kazhdan (see \cite{EK-qva}):

\begin{de}\label{non-degenerate}
A nonlocal vertex algebra $V$ is said to be {\em non-degenerate} if for every positive integer $n$,
the linear map $Z_n: \ V^{\otimes n}\otimes \C((x_1))\cdots ((x_n))\rightarrow V((x_1))\cdots ((x_n))$ is injective, where
\begin{align}
Z_n(v^{(1)}\otimes \cdots \otimes v^{(n)}\otimes f)=f(x_1,\dots,x_n)Y(v^{(1)},x_1)\cdots Y(v^{(n)},x_n){\bf 1}.
\end{align}
\end{de}

The following result, which was formulated in \cite{Li-nonlocal}, is essentially due to \cite{EK-qva}:

\begin{prop}\label{nondeg-wqva}
Let $V$ be a weak quantum vertex algebra. Assume that $V$ is non-degenerate.
Then there exists a unique $\SY$-locality operator $\SY(x)$.
Furthermore, $\SY(x)$ is a unitary rational quantum Yang-Baxter operator and $(V,\SY(x))$ is a quantum vertex algebra.
\end{prop}

\begin{rem}
{\em Note that an $\SY$-locality operator is part of a quantum vertex algebra structure by definition.
 In view of Proposition \ref{nondeg-wqva}, the term ``non-degenerate quantum vertex algebra'' (without specifying
 an $\SY$-locality operator) is unambiguous.}
\end{rem}

\begin{de}
Let $V$ be a nonlocal vertex algebra. A {\em $V$-module} is a vector space $W$ equipped with a linear map
\begin{align*}
  Y_W(\cdot,x):\ & V\to  (\te{End}W)[[x,x^{-1}]]; \quad  v\mapsto Y_W(v,x),
\end{align*}
satisfying the conditions that $Y_W({\bf 1},x)=1_W$ (the identity operator on $W$),
$$Y_W(v,x)w\in W((x))\quad \text{ for }v\in V,\ w\in W,$$
and that for $u,v\in V,\ w\in W$, there exists $l\in \N$ such that
\begin{align}\label{module-weak-asso}
(x_0+x_2)^lY_W(u,x_0+x_2)Y_W(v,x_2)w=(x_0+x_2)^lY_W(Y(u,x_0)v,x_2)w.
\end{align}
\end{de}



Let $W$ be a general vector space (over $\C$). Set
\begin{align}
\E(W)=\Hom (W,W((x)))\subset (\End W)[[x,x^{-1}]].
\end{align}
A typical element of $\E(W)$ is the identity operator on $W$, denoted by $1_W$.

\begin{de}
A finite sequence $a^{[1]}(x),\dots, a^{[r]}(x)$ in $\E(W)$ is said to be \emph{compatible}
if  there exists a nonnegative integer $k$ such that
\begin{align}
 \( \prod_{1\le i<j\le r}(x_i-x_j)^k\) a^{[1]}(x_1)\cdots a^{[r]}(x_r)
  \in \Hom (W,W((x_1,\dots,x_r))).
\end{align}
A subset $U$ of $\E(W)$ is said to be \emph{compatible}
if every finite sequence in $U$ is compatible.
\end{de}

Let $(a(x),b(x))$ be a compatible ordered pair in $\E(W)$.
Define $a(x)_nb(x)\in \E(W)$ for $n\in \Z$ in terms of generating function
$$Y_{\E}(a(x),z)b(x)=\sum_{n\in \Z}a(x)_nb(x) z^{-n-1}$$
by
\begin{align}
Y_{\E}(a(x),z)b(x)=z^{-k}\left((x_1-x)^ka(x_1)b(x)\right)|_{x_1=x+z},
\end{align}
where $k$ is any nonnegative integer such that
$$(x_1-x)^ka(x_1)b(x)\in \Hom (W,W((x,x_1))).$$

A compatible subspace $U$ of $\E(W)$ is said to be {\em $Y_\E$-closed} if
$$a(x)_mb(x)\in U\ \ \text{ for all }a(x),b(x)\in U,\ m\in \Z.$$

The following result was obtained in  \cite{Li-nonlocal}:

\begin{thm}
Let $V$ be a $Y_{\E}$-closed compatible subspace of $\E(W)$ with $1_W\in V$.
Then $(V,Y_\E,1_W)$ is a nonlocal vertex algebra and $W$ is a $V$-module with $Y_W(a(x),z)=a(z)$ for $a(x)\in V$.
On the other hand,  for any compatible subset $U$ of $\E(W)$, there exists a (unique) smallest
$Y_\E$-closed compatible subspace $\<U\>$ such that $U\cup \{ 1_W\}\subset \<U\>$, where
\begin{eqnarray}
\<U\>={\rm span}\{ u^{(1)}(x)_{m_1}\cdots u^{(r)}(x)_{m_r}1_W\ |\ r\ge 0,\ u^{(i)}(x)\in U,\ m_i\in \Z\}.
\end{eqnarray}
\end{thm}

A subset $U$ of $\E(W)$ is said to be {\em $\SY$-local} if for any $a(x),b(x)\in U$, there exist
$$c_i(x),\ d_i(x)\in U,\ f_i(x)\in \C((x))\  \te{ for }1\le i\le r$$
and a nonnegative integer $k$ such that
\begin{align}
(x_1-x_2)^ka(x_1)b(x_2)=(x_1-x_2)^k\sum_{i=1}^r f_i(x_2-x_1)c_i(x_2)d_i(x_1).
\end{align}

With this notion, we have (see \cite{Li-nonlocal}):

\begin{prop}
Let $U$ be an $\SY$-local subset of $\E(W)$. Then $U$ is compatible and $\<U\>$ is a weak quantum vertex algebra.
\end{prop}

\section{$\hbar$-adic quantum vertex algebras and their modules}

 In this section, we recall some basic notions and results on $\hbar$-adic nonlocal vertex algebras
and $\hbar$-adic (weak) quantum vertex algebras.

\subsection{Topologically free $\C[[\hbar]]$-modules}
Let $\hbar$ be a formal variable. A $\C[[\hbar]]$-module $V$
is said to be \emph{torsion-free} if $\hbar v\ne 0$ for any nonzero vector $v\in V$,
and said to be \emph{separated} if $\cap_{n\ge1}\hbar^n V=0$.
For a $\C[[\hbar]]$-module $V$, using subsets $v+\hbar^nV$ for $v\in V$, $n\ge 1$
as the basis of open sets one obtains a topology on $V$, which is called the
\emph{$\hbar$-adic topology}.
A $\C[[\hbar]]$-module $V$ is said to be \emph{$\hbar$-adically complete}
if every Cauchy sequence in $V$ with respect to the $\hbar$-adic topology has a limit in $V$.
A $\C[[\hbar]]$-module $V$ is \emph{topologically free} if $V=V_0[[\hbar]]$
for some $\C$-subspace $V_0$ of $V$.

Let $U$ and $V$ be topologically free $\C[[\hbar]]$-modules.
The $\hbar$-adically completed tensor product $U\wh\ot V$ of $U$ and $V$
is defined to be the $\hbar$-adic completion of $U\otimes V$.
If $U=U_0[[\hbar]]$ and $V=V_0[[\hbar]]$, then $U\wh\ot V=(U_0\ot V_0)[[\hbar]]$.

The following basic facts can be either found (cf. \cite{Ka}) or proved straightforwardly:

\begin{lem}\label{tpfree-basics}
(1)  A $\C[[\hbar]]$-module is topologically free
if and only if it is torsion-free, separated, and $\hbar$-adically complete.
(2) If $W$ is a topologically free $\C[[\hbar]]$-module, then $W=U[[\hbar]]$
for any $\C$-subspace $U$ of $W$ such that $W=U\oplus \hbar W$ over $\C$.
(3) If $U$ is a $\C[[\hbar]]$-submodule of a torsion-free and separated $\C[[\hbar]]$-module
and if $U$ itself is $\hbar$-adically complete, then $U$ is topologically free.
\end{lem}

\begin{lem}\label{def-K[[h]]'}
Let $W$ be a topologically free $\C[[\hbar]]$-module and let $K$ be a $\C$-subspace of $W$.
Define $K[[\hbar]]'$ to be the image of $K[[\hbar]]$ in $W$ under the canonical $\C[[\hbar]]$-module morphism from $K[[\hbar]]$ to $W$.
Then $K[[\hbar]]'$ is topologically free. On the other hand, a $\C[[\hbar]]$-submodule $K$ of $W$ itself is $\hbar$-adically complete
if and only if $K=K[[\hbar]]'$.
\end{lem}


The following result can be found in \cite{Ka}:

\begin{lem}\label{free-top}
Let $W_1,W_2$ be topologically free $\C[[\hbar]]$-modules and let
$\psi:W_1\rightarrow W_2$ be a continuous $\C[[\hbar]]$-module map.
 If the derived $\C$-linear map $\bar{\psi}: W_1/\hbar W_1\rightarrow W_2/\hbar W_2$ is one-to-one (resp. onto),
 then $\psi$ is one-to-one (resp. onto).
\end{lem}

We also have the following simple facts (see \cite{Li-h-adic}; the proof of Proposition 3.7):

\begin{lem}\label{[U]}
Let $U$ be a $\C[[\hbar]]$-submodule of a $\C[[\hbar]]$-module $W$. Set
\begin{align}
[U]=\{ w\in W\ |\ \hbar^n w\in U\ \text{ for some }n\in \N\}.
\end{align}
Then $\left[ [U]\right]=[U]$. On the other hand, if $[U]=U$, then $U\cap \hbar^n W =\hbar^n U$ for $n\in \N$.
\end{lem}

\begin{lem}\label{simple-fact-2}
Let $U$ be a $\C[[\hbar]]$-submodule of a topologically free $\C[[\hbar]]$-module $W$ such that $[U]=U$. Then
$\left[\bar{U}\right]=\bar{U}$, $\overline{\left[\bar{U}\right]}=\left[\bar{U}\right]$, and $\bar{U}$ is
topologically free.
\end{lem}

\begin{lem}\label{simple-fact-3}
Let $U$ and $V$ be topologically free $\C[[\hbar]]$-modules and let $\psi: U\rightarrow V$
be a $\C[[\hbar]]$-module map. Then $[\ker (\psi)]=\ker (\psi)$, $\overline{\ker (\psi)}=\ker (\psi)$, and
$\ker (\psi)$ is topologically free.
\end{lem}

\begin{lem}\label{basic-facts-top-algebra}
Let $A$ be an associative algebra over $\C[[\hbar]]$, which is topologically free as a $\C[[\hbar]]$-module.

1) For any left (right) ideal $J$ of $A$ as a $\C[[\hbar]]$-algebra, $[J]$ is a  left (right) ideal.

2) For any  left (right) ideal $J$ of $A$ such that $[J]=J$, $\overline{J}$ is a  left (right) ideal of $A$
such that $[\overline{J}]=\overline{J}$ and $\overline{J}$ is topologically free.
 \end{lem}

\begin{lem}\label{strong-submodule}
Let $W$ be a topologically free $\C[[\hbar]]$-module and let $U$ be a $\C[[\hbar]]$-submodule such that
$[U]=U$ and $U$ is $\hbar$-adically complete. Then there exist a $\C$-subspace $W_0$ of $W$ and a $\C$-subspace $U_0$ of $W_0$ such that
$W=W_0[[\hbar]]$ and $U=U_0[[\hbar]]$. Furthermore, $W/U=(W_0/U_0)[[\hbar]]$. In particular, $W/U$ is topologically free.
\end{lem}

\begin{proof} As a $\C[[\hbar]]$-submodule of $W$, $U$ is torsion-free and separated. In addition, $U$ is
$\hbar$-adically complete by assumption.
 Then $U$ is topologically free, i.e., $U=U_0[[\hbar]]$ for some $\C$-subspace $U_0$ of $U$.
 Noticing that $U\cap \hbar W=\hbar U$ by Lemma \ref{[U]},  we have $U_0\cap \hbar W=U_0\cap U\cap \hbar W=U_0\cap \hbar U=0$.
 Then there exists a $\C$-subspace $W_0'$ of $W$ such that  $W=W_0'\oplus U_0\oplus \hbar W$ over $\C$.
 Set $W_0=W_0'+U_0$.
 Then $U_0\subset W_0$ and $W=W_0[[\hbar]]$. Next, we prove $W/U=(W_0/U_0)[[\hbar]]$.
The natural $\C$-linear map $\theta: W_0\rightarrow W_0/U_0$ gives
a surjective $\C[[\hbar]]$-morphism $\tilde{\theta}: W=W_0[[\hbar]]\rightarrow (W_0/U_0)[[\hbar]]$.
It is clear that $U=U_0[[\hbar]]\subset \ker \tilde{\theta}$, so
we have a surjective $\C[[\hbar]]$-morphism $\bar{\theta}: W/U\rightarrow (W_0/U_0)[[\hbar]]$.
Suppose
$$v+U=(v_0+\hbar v_1+\hbar^2v_2+\cdots)+U\in \ker\bar{\theta}$$
with $v_n\in W_0$ for $n\ge 0$.
Then $\theta(v_0)+\hbar \theta (v_1)+\hbar^2\theta(v_2)+\cdots =0$ in $(W_0/U_0)[[\hbar]]$, which implies
 $\theta (v_n)=0$, i.e., $v_n\in U_0$ for all $n\ge 0$. Thus $v\in U$. This proves that $\bar{\theta}$ is also injective,
 and hence $\bar{\theta}$ is a $\C[[\hbar]]$-isomorphism.
 \end{proof}

Let $M$ be a $\C[[\hbar]]$-module. For a sequence $\{ w(n)\}_{n\in \N}$ in $M$,
we write $\lim_{n\rightarrow \infty}w(n)=0$ to signify that for every $r\in \Z_{+}$,
there exists $k\in \Z_{+}$ such that $w(n)\in \hbar^rM$ for all $n\ge k$. Define
\begin{align}
M_{\hbar}((x))=\left\{ \left.\psi(x)=\sum_{n\in \Z}\psi(n)x^{n}\in M[[x,x^{-1}]]\ \right|\ \lim_{n\rightarrow \infty}\psi(-n)=0\right\}.
\end{align}
If $M=M_0[[\hbar]]$ is a topologically free $\C[[\hbar]]$-module, then $M_{\hbar}((x))=(M_0((x)))[[\hbar]]$, which is topologically free.

Let $W=W_0[[\hbar]]$ be a topologically free $\C[[\hbar]]$-module.
Then $\End W=(\End W_0)[[\hbar]]$, which is a topologically free $\C[[\hbar]]$-module.
Define
\begin{align}
 \E_\hbar(W)=\Hom (W, W_{\hbar}((x)))\subset (\te{End} W)[[x,x^{-1}]].
\end{align}
Note that $\E_\hbar(W)=\E(W_0)[[\hbar]]$ is a topologically free $\C[[\hbar]]$-module.
On the other hand, in terms of the natural $\C[[\hbar]]$-module maps
\begin{align}
\pi_n:\ (\te{End} W)[[x,x^{-1}]]\longrightarrow (\te{End} (W/\hbar^nW))[[x,x^{-1}]]
\end{align}
for $n\in \Z_{+}$, we have
\begin{eqnarray}
\quad\quad\quad
 \E_\hbar(W)=\{\psi(x)\in (\te{End} W)[[x,x^{-1}]]\ |\ \pi_n(\psi(x))\in \E(W/\hbar^nW)\ \ \text{ for }n\ge 1\}.
\end{eqnarray}

It is clear that we have a natural inverse system
\begin{eqnarray*}
    0\longleftarrow W/\hbar W\longleftarrow W/\hbar^2 W\longleftarrow \cdots\longleftarrow,
\end{eqnarray*}
which induces an inverse system
\begin{eqnarray*}
    0\longleftarrow \E(W/\hbar W)\longleftarrow \E(W/\hbar^2 W)\longleftarrow \cdots\longleftarrow.
\end{eqnarray*}
Then
\begin{align}
  \E_\hbar(W)=\varprojlim_{n\ge1}\E(W/\hbar^nW).
\end{align}


\subsection{$\hbar$-adic nonlocal vertex algebras and modules}
The following notion was introduced in \cite{Li-h-adic}:

\begin{de}\label{de:h-adic-nonlocal-va}
An {\em $\hbar$-adic nonlocal vertex algebra} is a topologically free $\C[[\hbar]]$-module $V$,
equipped with a $\C[[\hbar]]$-module map
$Y(\cdot,x):V\to (\End V)[[x,x^{-1}]]$ and a distinguished vector $\vac\in V$,
satisfying the conditions that
$$Y(u,x)v\in V_{\hbar}((x))\quad \text{ for }u,v\in V,$$
$$Y({\bf 1},x)=1_V, \quad Y(v,x){\bf 1}\in V[[x]]\ \ \text{ and }\ \ (Y(v,x){\bf 1})|_{x=0}=v\ \text{ for }v\in V,$$
and that for any $u,v,w\in V$ and for any $n\in \Z_{+}$, there exists $l\in \N$ such that
\begin{align}
(x_0+x_2)^lY(u,x_0+x_2)Y(v,x_2)w\equiv (x_2+x_0)^lY(Y(u,x_0)v,x_2)w
\end{align}
modulo $\hbar^n V[[x_0^{\pm 1},x_2^{\pm 1}]]$.
\end{de}

We have (see \cite{Li-h-adic}):

\begin{prop}
 Let $V$ be a topologically free $\C[[\hbar]]$-module equipped with a $\C[[\hbar]]$-module map
$Y(\cdot,x):V\to (\End V)[[x,x^{-1}]]$ and a vector $\vac\in V$. Then $(V,Y,{\bf 1})$ is an $\hbar$-adic nonlocal vertex algebra
if and only if for every positive integer $n$, $(V/\hbar^nV, \overline{Y}_n,\vac)$
is a nonlocal vertex algebra over $\C$, where $\overline{Y}_n(\cdot,x):V/\hbar^nV\to(\End(V/\hbar^nV))[[x,x^{-1}]]$
is the canonical map induced from $Y(\cdot,x)$.
\end{prop}

For any $\hbar$-adic nonlocal vertex algebra $V=V_0[[\hbar]]$, define a $\C[[\hbar]]$-module map
$$Y(x_1,x_2):\ V\wh\ot V\wh\ot \C((x))[[\hbar]]\rightarrow (\End V)[[x_1^{\pm 1},x_2^{\pm}]]$$
by
\begin{align}
Y(x_1,x_2)(u\otimes v\otimes f(x))=f(x_1-x_2)Y(u,x_1)Y(v,x_2)
\end{align}
for $u,v\in V_0,\ f(x)\in \C((x))$.

\begin{de}\label{def-hwqva}
{\em An $\hbar$-adic weak quantum vertex algebra} is an $\hbar$-adic nonlocal vertex algebra $V$
which satisfies the ($\hbar$-adic) {\em $\SY$-locality:} For any $u,v\in V$,
there exists $F(u,v,x)\in V\wh\ot V\wh\ot \C((x))[[\hbar]]$,
satisfying the condition that for any $n\in \Z_{+}$, there is $k\in \N$ such that
\begin{align}
(x_1-x_2)^kY(u,x_1)Y(v,x_2)\equiv (x_1-x_2)^kY(x_2,x_1)(F(u,v,x))
\end{align}
modulo $\hbar^n(\End V)[[x_1^{\pm 1},x_2^{\pm 1}]]$.
\end{de}

\begin{de}
Let $M$ be a $\C[[\hbar]]$-module. For $A(x_1,x_2),B(x_1,x_2)\in M[[x_1^{\pm 1},x_2^{\pm 1}]]$,
we write $A(x_1,x_2)\sim B(x_1,x_2)$ if for any $n\in \Z_{+}$, there exists $k\in\N$ such that
\begin{align}
  (x_1-x_2)^kA(x_1,x_2)\equiv (x_1-x_2)^kB(x_1,x_2)\   \mod \hbar^n M[[x_1^{\pm 1},x_2^{\pm 1}]].
\end{align}
\end{de}

Let $V$ be an $\hbar$-adic weak quantum vertex algebra. From definition, there exists a $\C[[\hbar]]$-module map
$F(x): V\wh\ot V\rightarrow V\wh\ot V\wh\ot \C((x))[[\hbar]]$ such that
\begin{align}
Y(u,x_1)Y(v,x_2)\sim Y(x_2, x_1)(F(x)(v\otimes u))\quad \text{ for }u,v\in V.
\end{align}
As before, such a $\C[[\hbar]]$-module map $F(x)$ is called an {\em $\SY$-locality operator} of $V$.

We have (see \cite[(2.25), Prop. 2.19]{Li-h-adic}):

 \begin{prop}\label{lem:q-Jacobi}
 Let $V$ be an $\hbar$-adic weak quantum vertex algebra and let $\SY(x)$ be an $\SY$-locality operator of $V$. Then
 \begin{align}
&x_0\inv\delta\!\(\frac{x_1-x_2}{x_0}\)\!  Y(u,x_1)Y(v,x_2)w\\
&\quad -x_0\inv\delta\!\(\frac{x_2-x_1}{-x_0}\)\! Y(x_2,x_1)(\SY(x)(v\ot u))w\nonumber\\
  =\ & x_1\inv\delta\!\(\frac{x_2+x_0}{x_1}\)\!Y(Y(u,x_0)v,x_2)w\nonumber
\end{align}
 for $u,v,w\in V$, which is called the {\em $\SY$-Jacobi identity for the ordered triple $(u,v,w)$.}
 \end{prop}

 As an immediate consequence of the $\SY$-Jacobi identity, we have
\begin{align}
  &Y(u,x_1)Y(v,x_2)-Y(x_2,x_1)(\mathcal{S}(x)(v\ot u))\\
  =\ & \Res_{x_0}x_1\inv\delta\!\(\frac{x_2+x_0}{x_1}\)\!Y(Y(u,x_0)v,x_2)\nonumber
\end{align}
(the {\em $\mathcal{S}$-commutator formula}) for $u,v\in V$.

For an $\hbar$-adic nonlocal vertex algebra $V$, define a $\C[[\hbar]]$-linear operator $\D$ on $V$
just as for a nonlocal vertex algebra (see (\ref{D-def})). Then the properties (\ref{D-bracket-der})--(\ref{D-creation}) hold.



Let $W$ be a topologically free $\C[[\hbar]]$-module.
A {\em rational quantum Yang-Baxter operator} on $W$ is a $\C[[\hbar]]$-module map
\begin{align}
  \SY(x):\ W\wh\ot W\to W\wh\ot W\wh\ot \C((x))[[\hbar]],
\end{align}
satisfying the rational quantum Yang-Baxter equation.

The following notion, formulated in \cite{Li-h-adic}, slightly generalizes Etingof-Kazhdan's notion of
quantum vertex operator algebra (see \cite{EK-qva}; cf. Remark \ref{hqva-EKqvoa} below):

\begin{de}\label{h-qva}
An \emph{$\hbar$-adic quantum vertex algebra}
is an $\hbar$-adic nonlocal vertex algebra $V$ equipped with a unitary rational quantum Yang-Baxter operator
$\SY(x)$ on $V$, such that $\SY(x)$ is an $\SY$-locality operator and such that the shift condition (\ref{eq:qyb-shift})
and  the hexagon identity (\ref{eq:qyb-hex-id}) on $V\wh\ot V\wh \ot V$ hold.
 \end{de}

\begin{rem}\label{hqva-EKqvoa}
{\em Note that a quantum vertex operator algebra in the sense of \cite{EK-qva}
is an $\hbar$-adic quantum vertex algebra $V$ such that $V/\hbar V$ is a vertex algebra.
For a general $\hbar$-adic quantum vertex algebra $V$,
$V/\hbar V$ is a quantum vertex algebra.}
\end{rem}

The following is a reformulation of Proposition 1.11 of \cite{EK-qva}:

\begin{prop}\label{nondegenerate-wqva}
Let $V$ be an $\hbar$-adic weak quantum vertex algebra which is non-degenerate in the sense that
 the nonlocal vertex algebra $V/\hbar V$ is non-degenerate.
Then there exists a unique $\SY$-locality operator
$\SY(x)$. Furthermore, $\SY(x)$ is a unitary rational quantum Yang-Baxter operator and $(V,\SY(x))$
is an $\hbar$-adic quantum vertex algebra.
\end{prop}

Let $V$ be an $\hbar$-adic nonlocal vertex algebra and let $U$ be a subset. Set $U^{(1)}=\C[[\hbar]] (U\cup\{{\bf 1}\})$
and then inductively define $U^{(n+1)}$ for $n\ge 1$ by
\begin{align}
U^{(n+1)}={\rm span}_{\C[[\hbar]]}\{ a_mb\ |\ a,b\in U^{(n)},\ m\in \Z\}.
\end{align}
We have ${\bf 1}\in U^{(n)}$ and $U^{(n)}\subset U^{(n+1)}$ for all $n\in \Z_{+}$.
Set $U^{(\infty)}=\cup_{n\ge 1}U^{(n)}$. Then $U^{(\infty)}$ is $Y$-closed and contains $U\cup \{{\bf 1}\}$.
Furthermore, $\overline{[U^{(\infty)}]}$ is an $\hbar$-adic nonlocal vertex subalgebra,
which is called the {\em $\hbar$-adic nonlocal vertex subalgebra generated by $U$}.

The following are some straightforward facts:

\begin{lem}\label{U-infinity-facts}
Let $U$ be a subset of an $\hbar$-adic nonlocal vertex algebra $V$. Set
\begin{align}
U^{\infty}={\rm span}_{\C[[\hbar]]}\{ u^{(1)}_{m_1}\cdots u^{(r)}_{m_r}{\bf 1}\ |\ r\ge 0,\ u^{(i)}\in U,\ m_i\in \Z\}.
\end{align}
Then
\begin{align}
U^{\infty}\subset U^{(\infty)}, \quad U^{(\infty)}\subset \overline{U^{\infty}}, \quad \overline{U^{(\infty)}}=\overline{U^{\infty}},
\quad \overline{[U^{(\infty)}]}=\overline{[U^{\infty}]}.
\end{align}
In particular, if $V= \overline{[U^{\infty}]}$, then $V= \overline{[U^{(\infty)}]}=\<U\>$.
\end{lem}

A subset $U$ of an $\hbar$-adic nonlocal vertex algebra $V$ is called an {\em $\SY$-local subset} if for any $u,v\in U$, there exists
$A(x)\in (\C U\ot \C U\ot \C((x)))[[\hbar]]$
such that $$Y(u,x_1)Y(v,x_2)\sim Y(x_2,x_1)(A(x)).$$

We have (see \cite[Prop. 3.12]{Li-h-adic}):

\begin{lem}\label{weak-quantum-subalgebra}
Let $V$ be an $\hbar$-adic nonlocal vertex algebra. Suppose $V=U^{(\infty)}[[\hbar]]'$ for some
$\SY$-local subset $U$. Then $V$ is an $\hbar$-adic weak quantum vertex algebra.
\end{lem}

 \begin{de}
 Let be an $\hbar$-adic nonlocal vertex algebra. A {\em $V$-module} is
 a topologically free $\C[[\hbar]]$-module $W$ equipped with a $\C[[\hbar]]$-module map
 $$Y_W(\cdot,x):\ V\rightarrow  (\End W)[[x,x^{-1}]],$$
satisfying the following conditions:
$$Y_W(v,x)\in \E_{\hbar}(W)\quad \text{ for }v\in V,$$
 $Y_W({\bf 1},x)=1_W$, and
the {\em ($\hbar$-adic) weak associativity} holds: For any $u,v\in V,\ w\in W$ and for every positive integer $n$,
there exists $l\in \N$ such that
\begin{align}
(x_0+x_2)^lY_W(u,x_0+x_2)Y_W(v,x_2)w\equiv (x_2+x_0)^lY_W(Y(u,x_0)v,x_2)w
\end{align}
modulo $\hbar^n W[[x_0^{\pm 1},x_2^{\pm 1}]]$.
\end{de}

 From definition we immediately have:

\begin{prop}
Let $V$ be an $\hbar$-adic nonlocal vertex algebra and let $W$ be a topologically free $\C[[\hbar]]$-module
equipped with a $\C[[\hbar]]$-module map
 $$Y_W(\cdot,x):\ V\rightarrow  (\End W)[[x,x^{-1}]].$$
Then $W$ is a $V$-module if and only if $W/\hbar^nW$ is a $(V/\hbar^nV)$-module for every $n\in \Z_+$.
 \end{prop}

 For the rest of this subsection, assume that $V$ is an $\hbar$-adic nonlocal vertex algebra and
$(W,Y_W)$ is a $V$-module. Just as for $V$, define a $\C[[\hbar]]$-module map
\begin{align}
Y_W(x_1,x_2):\ V\widehat{\otimes}V\widehat{\otimes} \C((x))[[\hbar]]\rightarrow (\End W)[[x_1^{\pm 1},x_2^{\pm 1}]]
\end{align}
by
\begin{align*}
Y_W(x_1,x_2)(u\otimes v\otimes f(x))=f(x_1-x_2)Y_W(u,x_1)Y_W(v,x_2)
\end{align*}
for $u,v\in V_0,\ f(x)\in \C((x))$, where $V=V_0[[\hbar]]$.

The following are some standard results (see \cite{Li-h-adic}, Prop. 2.25; cf. \cite{Li-nonlocal}):

\begin{lem}\label{sim-commutator}
Let
$$u,v \in V,\ A(x)\in V\widehat{\otimes}V\widehat{\otimes} \C((x))[[\hbar]].$$
Then
$$Y_W(u,x_1)Y_W(v,x_2)\sim Y_W(x_2,x_1)(A(x))$$
if and only if
 \begin{align}
&x_0^{-1}\delta\!\(\frac{x_1-x_2}{x_0}\)\! Y_W(u,x_1)Y_W(v,x_2)
-x_0^{-1}\delta\!\(\frac{x_2-x_1}{-x_0}\)\! Y_W(x_2,x_1)(A(x))\nonumber\\
 &\hspace{2cm} = x_1^{-1}\delta\!\(\frac{x_2+x_0}{x_1}\)\!Y_W(Y(u,x_0)v,x_2).\nonumber
\end{align}
\end{lem}

\begin{prop}\label{prop-2.25}
Let
$$u,v, c^{(0)},c^{(1)},\dots \in V,\ A(x)\in V\widehat{\otimes}V\widehat{\otimes} \C((x))[[\hbar]]$$
with $\lim_{n\rightarrow \infty}c^{(n)}=0$. If
\begin{align*}
Y(u,x_1)Y(v,x_2)-Y(x_2,x_1)(A(x))
=\sum_{n\ge 0}Y(c^{(n)},x_2)\frac{1}{n!}\!\left(\frac{\partial}{\partial x_2}\right)^nx_1^{-1}\delta\!\left(\frac{x_2}{x_1}\right)
\end{align*}
on $V$, then
\begin{align*}
Y_W(u,x_1)Y_W(v,x_2)-Y_W(x_2,x_1)(A(x))
=\sum_{n\ge 0}Y_W(c^{(n)},x_2)\frac{1}{n!}\!\left(\frac{\partial}{\partial x_2}\right)^nx_1^{-1}\delta\!\left(\frac{x_2}{x_1}\right)\!.
\end{align*}
Furthermore,  if $(W,Y_W)$ is faithful, the converse is also true with $u_nv=c^{(n)}$ for $n\ge 0$.
\end{prop}

\begin{prop}\label{va-S-locality-module}
Let
$$u,v \in V,\ A(x)\in V\widehat{\otimes}V\widehat{\otimes} \C((x))[[\hbar]].$$
If
$$Y(u,x_1)Y(v,x_2)\sim Y(x_2,x_1)(A(x))$$
then
$$Y_W(u,x_1)Y_W(v,x_2)\sim Y_W(x_2,x_1)(A(x)).$$
The converse is also true if $(W,Y_W)$  is a faithful module.
\end{prop}

We also have (see \cite[Prop. 3.13]{Li-h-adic}):

\begin{lem}\label{U-vacuum-like}
Let  $e\in W$ and let $U$ be an $\SY$-local subset of $V$ such that
$$Y_W(u,x)e\in W[[x]]\quad \text{ for }u\in U.$$
Then $Y_W(v,x)e\in W[[x]]$ for all $v\in \<U\>.$  Furthermore, the map
$\theta: \ \<U\>\rightarrow W,$ defined by $\theta(v)=v_{-1}e$ for $v\in \<U\>$, is a $\<U\>$-module homomorphism.
\end{lem}

We present the following two nonstandard results for studying double Yangians:

\begin{lem}\label{Y-V-W-a-b-relation}
Let $a,b, u, v\in V, \ R(x), f(x), g(x)\in \C((x))[[\hbar]].$
 If
\begin{align}\label{Y-aba=gba-CD}
Y(a,x)Y(b,z)-R(z-x)Y(b,z)Y(a,x)=f(x-z)Y(u,z)+g(z-x)Y(v,z)
\end{align}
on $V$, then
\begin{align}\label{prop-Y-aba=gba-CD-module}
&\ \quad Y_W(a,x)Y_W(b,z)-R(z-x)Y_W(b,z)Y_W(a,x)\\
&=f(x-z)Y_W(u,z)+g(z-x)Y_W(v,z)\nonumber
\end{align}
on $W$. Furthermore, the converse is also true if $(W,Y_W)$ is faithful.
\end{lem}

\begin{proof} Write $f(x)=\sum_{m\in \Z}f(m)x^m$ with $f(m)\in \C[[\hbar]]$ and define
\begin{align}
\tilde{f}(z)=\Res_xf(x)e^{zx}=\sum_{n\ge 0}\frac{1}{n!}f(-n-1)z^n\in \C[[\hbar]][[z]].
\end{align}
Note that
\begin{align*}
f(x-z)-f(-z+x)=\tilde{f}\!\left(\frac{\partial}{\partial z}\right)x^{-1}\delta\!\left(\frac{z}{x}\right)\!.
\end{align*}
Using this we rewrite (\ref{Y-aba=gba-CD}) as
\begin{align}
&Y(a,x)Y(b,z)\\
&\ \  -\left(R(z-x)Y(b,z)Y(a,x)+f(-z+x)Y(u,z)+g(z-x)Y(v,z)\right)\nonumber\\
 =\ &(f(x-z)-f(-z+x))Y(u,z)\nonumber\\
=\ &Y(u,z)\tilde{f}\!\left(\frac{\partial}{\partial z}\right)x^{-1}\delta\!\left(\frac{z}{x}\right)\!.\nonumber
\end{align}
Similarly,  we rewrite (\ref{prop-Y-aba=gba-CD-module}) as
\begin{align*}
&Y_W(a,x)Y_W(b,z)\\
&\ \  -\left(R(z-x)Y_W(b,z)Y_W(a,x)+f(-z+x)Y_W(u,z)+g(z-x)Y_W(v,z)\right)\\
=\ &Y_W(u,z)\tilde{f}\!\left(\frac{\partial}{\partial z}\right)x^{-1}\delta\!\left(\frac{z}{x}\right)\!.\nonumber
\end{align*}
Note that  $\lim_{n\rightarrow \infty}f(-n)=0$ as $f(x)\in \C((x))[[\hbar]]$.
Then it follows  immediately from Propositions \ref{prop-2.25}.
\end{proof}

\begin{lem}\label{pq-va-module-relation}
Let
$$u,v\in V,\ B(x)\in V\wh\ot V\wh\ot \C((x))[[\hbar]], \ p(x,z), q(x,z)\in \C[x,z]$$
 with $p(x,0)\ne 0$. If
\begin{align}\label{pq-uv-va}
p(x_1-x_2,\hbar)Y(u,x_1)Y(v,x_2)=q(x_2-x_1,\hbar)Y(x_2,x_1)(B(x))
\end{align}
on $V$, then
\begin{align}\label{pq-uv-module}
p(x_1-x_2,\hbar)Y_W(u,x_1)Y_W(v,x_2)=q(x_2-x_1,\hbar)Y_W(x_2,x_1)(B(x))
\end{align}
on $W$. The converse is also true if $(W,Y_W)$ is a faithful $V$-module.
\end{lem}

\begin{proof} Rewrite  (\ref{pq-uv-va}) as
$$p(x_1-x_2,\hbar)Y(u,x_1)Y(v,x_2)=p(x_1-x_2,\hbar)\left( \frac{q(x_2-x_1,\hbar)}{p(-x_2+x_1,\hbar)}\right)Y(x_2,x_1)(B(x)),$$
where as a convention,
$$ \frac{q(x_2-x_1,\hbar)}{p(-x_2+x_1,\hbar)}:=e^{-x_1\frac{\partial}{\partial x_2}}\iota_{x_2,\hbar}
\left( \frac{q(x_2,\hbar)}{p(-x_2,\hbar)}\right)\!.$$
The same argument in \cite[Lemma 4.25]{Li-h-adic} proves
\begin{align}
Y(u,x_1)Y(v,x_2)\sim \frac{q(x_2-x_1,\hbar)}{p(-x_2+x_1,\hbar)}Y(x_2,x_1)(B(x)).\label{similar-YuYv-A}
\end{align}
Then by Lemma \ref{sim-commutator} we have
\begin{align}\label{S-Jacobi-uv-B}
&x_0^{-1}\delta\!\(\frac{x_1-x_2}{x_0}\)\!  Y(u,x_1)Y(v,x_2)\\
 &\quad -x_0^{-1}\delta\!\(\frac{x_2-x_1}{-x_0}\)\frac{q(x_2-x_1,\hbar)}{p(-x_2+x_1,\hbar)} Y(x_2,x_1)(B(x))\nonumber\\
 =\ &x_1^{-1}\delta\!\(\frac{x_2+x_0}{x_1}\)\!Y(Y(u,x_0)v,x_2).\nonumber
\end{align}
Using this and (\ref{pq-uv-va}) we get
\begin{align*}
\Res_{x_0}x_1^{-1}\delta\!\(\frac{x_2+x_0}{x_1}\)p(x_0,\hbar)Y(Y(u,x_0)v,x_2)=0.
\end{align*}
Then
\begin{align}\label{pY-uv-regular}
&\Res_{x_0}x_1^{-1}\delta\!\(\frac{x_2+x_0}{x_1}\)p(x_0,\hbar)Y(u,x_0)v\\
=\ & e^{-x_2\D}\Res_{x_0}x_1^{-1}\delta\!\(\frac{x_2+x_0}{x_1}\)p(x_0,\hbar)Y(Y(u,x_0)v,x_2){\bf 1}\nonumber\\
=\ &0.\nonumber
\end{align}
With (\ref{similar-YuYv-A}),  by Proposition \ref{va-S-locality-module} we have
$$Y_W(u,x_1)Y_W(v,x_2)\sim \frac{q(x_2-x_1,\hbar)}{p(-x_2+x_1,\hbar)}Y_W(x_2,x_1)(B(x)),$$
which by Lemma \ref{sim-commutator} is equivalent to
 \begin{align}\label{pq-Sjacobi-module}
&x_0^{-1}\delta\!\(\frac{x_1-x_2}{x_0}\)\! Y_W(u,x_1)Y_W(v,x_2)\\
&\quad -x_0^{-1}\delta\!\(\frac{x_2-x_1}{-x_0}\)\! \frac{q(x_2-x_1,\hbar)}{p(-x_2+x_1,\hbar)} Y_W(x_2,x_1)(B(x))\nonumber\\
  =\ & x_1^{-1}\delta\!\(\frac{x_2+x_0}{x_1}\)\!Y_W(Y(u,x_0)v,x_2).\nonumber
\end{align}
Applying $\Res_{x_0}p(x_1-x_2,\hbar)$ and then using (\ref{pY-uv-regular}) we obtain (\ref{pq-uv-module}) as
\begin{align*}
&p(x_1-x_2,\hbar)Y_W(u,x_1)Y_W(v,x_2)-Y_W(x_2,x_1)(B(x))\\
=\ &\Res_{x_0}x_1^{-1}\delta\!\(\frac{x_2+x_0}{x_1}\)\!p(x_0,\hbar)Y_W(Y(u,x_0)v,x_2)\\
=\ &0.
\end{align*}

Conversely, assume that $(W,Y_W)$ is a faithful $V$-module and  (\ref{pq-uv-module}) holds.
By Proposition \ref{va-S-locality-module},  (\ref{similar-YuYv-A}) holds, and then by Lemma \ref{sim-commutator},
(\ref{S-Jacobi-uv-B}) holds.  On the other hand,  with (\ref{pq-uv-module}) we have (\ref{pq-Sjacobi-module}).
As $(W,Y_W)$ is faithful, we conclude that (\ref{pY-uv-regular}) holds.
Then we obtain (\ref{pq-uv-va}).
\end{proof}

\begin{de}
Let $V$ be an $\hbar$-adic nonlocal vertex algebra. An {\em $\hbar$-adic nonlocal vertex subalgebra} of $V$
is a $\C[[\hbar]]$-submodule $U$ such that $U$ itself is $\hbar$-adically complete,
\begin{eqnarray}\label{subalgebra-closure}
{\bf 1}\in U,\ \ u_mv\in U\quad \text{ for }u,v\in U,\ m\in \Z,
\end{eqnarray}
and $\lim_{m\rightarrow \infty}u_mv=0$ with respect to the $\hbar$-adic topology of $U$.
Furthermore, $U$ is called a {\em strong $\hbar$-adic nonlocal vertex subalgebra} if $U=[U]$ in $V$.
\end{de}

\begin{rem}
{\em From definition, a $\C[[\hbar]]$-submodule $U$ is an $\hbar$-adic nonlocal vertex subalgebra if and only if
there is an $\hbar$-adic nonlocal vertex algebra structure on $U$ such that the natural embedding of $U$ into $V$ is a homomorphism.
On the other hand, a $\C[[\hbar]]$-submodule $U$ is a strong $\hbar$-adic nonlocal vertex subalgebra if and only if
$[U]=U$, $U$ is $\hbar$-adically complete, and (\ref{subalgebra-closure}) holds.}
\end{rem}

\begin{de}
Let $V$ be an $\hbar$-adic nonlocal vertex algebra. A {\em strong ideal} of $V$ is a $\C[[\hbar]]$-submodule $J$ such that
$[J]=J$, $J$ is $\hbar$-adically complete, and
\begin{align}
v_m J\subset J,\ \ a_m V\subset J\quad \text{ for }v\in V,\ m\in \Z,\ a\in J.
\end{align}
An $\hbar$-adic nonlocal vertex algebra $V$ is said to be {\em relatively simple}
if it is nonzero and if $V$ is its only nonzero strong ideal.
\end{de}


The following is straightforward:

\begin{lem}\label{kernal-strong-ideal}
Let $\psi: V\rightarrow K$ be a homomorphism of $\hbar$-adic nonlocal vertex algebras.
Then $\ker (\psi)$ is a strong ideal of $V$.
\end{lem}

Using Lemma \ref{strong-submodule} we immediately have:

\begin{lem}\label{strong-ideal}
Let $J$ be a strong ideal of an $\hbar$-adic nonlocal vertex algebra $V$.
Then the $\C[[\hbar]]$-module $V/J$ is topologically free and it is an $\hbar$-adic nonlocal vertex algebra.
\end{lem}

\begin{lem}\label{strong-ideal-by-T}
Let $T$ be a subset of an $\hbar$-adic nonlocal vertex algebra.
Denote by $(T)$ the intersection of all strong ideals of $V$ containing $T$.
Then $(T)$ is the smallest strong ideal containing $T$.
\end{lem}

\begin{prop}\label{simple-hva}
Let $V$ be an $\hbar$-adic nonlocal vertex algebra such that $V/\hbar V$ is a simple nonlocal vertex algebra over $\C$.
Then $V$ is relatively simple.
\end{prop}

\begin{proof} Let $J$ be any nonzero strong ideal of $V$. In view of Lemma \ref{strong-submodule},
 there exist a $\C$-subspace $V_0$ of $V$ and a $\C$-subspace $J_0$ of $V_0$ such that
$V=V_0[[\hbar]]$ and $J=J_0[[\hbar]]$. Consider the natural $\C[[\hbar]]$-morphism
 $\theta: J\rightarrow V\rightarrow V/\hbar V.$
 We have
 $$\ker \theta=J\cap \hbar V=\hbar J.$$
Then $\theta$ reduces to an injective $\C[[\hbar]]$-morphism $\bar{\theta}: J/\hbar J\rightarrow V/\hbar V$.
 It is clear that $J/\hbar J$ is an ideal of $V/\hbar V$.
Note that since $J\ne 0$, we have $J\ne \hbar J$, i.e., $J/\hbar J\ne 0$.
As $V/\hbar V$ is a simple nonlocal vertex algebra (over $\C$),
we conclude that $J/\hbar J=V/\hbar V$, which implies $J_0=V_0$. Therefore, we have $J=J_0[[\hbar]]=V_0[[\hbar]]=V$.
\end{proof}

Let $W$ be a module for an $\hbar$-adic nonlocal vertex algebra $V$.
A  {\em strong $V$-submodule} of $W$ is a $\C[[\hbar]]$-submodule
$U$ such that $[U]=U$, $U$ is $\hbar$-adically complete, and
\begin{align}
a_m u\subset U\quad \text{ for }a\in V,\ m\in \Z,\ u\in U.
\end{align}
A $V$-module $W$ is said to be {\em relatively simple (irreducible)} if $W\ne 0$ and
 if $W$ is its only nonzero strong $V$-submodule.

Using Lemma \ref{strong-submodule}, we immediately have:

\begin{lem}\label{hva-quotient-module}
Let $W$ be a module for an $\hbar$-adic nonlocal vertex algebra $V$ and let $U$ be a strong $V$-submodule of $W$.
Then $W/U$ is topologically free and it is a $V$-module.
\end{lem}

On the other hand, from the proof of Lemma \ref{simple-hva} we immediately have:

\begin{prop}\label{simple-hva-module}
Let $V$ be an $\hbar$-adic nonlocal vertex algebra and let $W$ be a $V$-module such that
$W/\hbar W$ is an irreducible $V/\hbar V$-module. Then $W$ is relatively simple.
\end{prop}

\subsection{General constructions}
Let $W=W_0[[\hbar]]$ be a topologically free $\C[[\hbar]]$-module. Recall
$$\E_\hbar(W)=\E(W_0)[[\hbar]]=(\te{Hom}(W_0,W_0((x))))[[\hbar]]$$
and for each positive integer $n$, we have a natural linear map
\begin{align}
\pi_n:\   (\End W)[[x,x^{-1}]]\longrightarrow (\End (W/\hbar^n W))[[x,x^{-1}]].
\end{align}
Let $a(x)\in (\End W)[[x,x^{-1}]]$. Then $a(x)\in \E_{\hbar}(W)$ if and only if
 $\pi_n(a(x))\in \E(W/\hbar^n W)$ for every positive integer $n$.

\begin{de}
A finite sequence in $\E_\hbar(W)$ is said to be \emph{$\hbar$-adically compatible}
if it is compatible in $ \E(W/\hbar^n W)$ for every positive integer $n$.
A subset $U$ of $\E_\hbar(W)$ is said to be \emph{$\hbar$-adically compatible}
if every finite sequence in $U$ is $\hbar$-adically compatible, i.e.,
for every positive integer $n$, $\pi_n(U)$ is a compatible subset of $ \E(W/\hbar^n W)$.
\end{de}

Let $(a(x),b(x))$ be an $\hbar$-adically compatible pair in $\E_\hbar(W)$.
For every positive integer $n$, as $(\pi_n(a(x)),\pi_n(b(x))$ is a compatible pair in $\E(W/\hbar^nW)$,
we have $\pi_n(a(x))_m\pi_n(b(x))\in \E(W/\hbar^nW)$ for $m\in \Z$.
For $m\in \Z$, define $a(x)_mb(x)\in \E_\hbar(W)$ by
\begin{eqnarray}\label{eq:def-y-E-op}
a(x)_mb(x)=\varprojlim_{n\ge 1}\pi_n(a(x))_m\pi_n(b(x)).
\end{eqnarray}
Form a generating function
\begin{eqnarray}
Y_{\E}(a(x),z)b(x)=\sum_{m\in \Z}a(x)_mb(x) z^{-m-1}.
\end{eqnarray}
An $\hbar$-adically compatible $\C[[\hbar]]$-submodule $K$ of $\E_\hbar(W)$ is said to be {\em $Y_{\E}$-closed} if
$$a(x)_mb(x)\in K\quad \te{for }a(x),b(x)\in K,\ m\in \Z.$$

The following result was obtained  in \cite[Theorem 4.16]{Li-h-adic}:

\begin{thm}\label{thm:h-adic-va-abs-construct}
Let $W$ be a topologically free $\C[[\hbar]]$-module. Assume that
$V$ is an $\hbar$-adically compatible $\C[[\hbar]]$-submodule of $\E_\hbar(W)$
such that $V$ contains $1_W$, $V$ is $Y_\E$-closed, $[V]=V$, and $V$ is $\hbar$-adically complete.
Then $(V,Y_\E,1_W)$ carries the structure of an $\hbar$-adic nonlocal vertex algebra
and $W$ is a faithful $V$-module with $Y_W(a(x),z)=a(z)$ for $a(x)\in V$.
On the other hand, let $U$ be any $\hbar$-adically compatible subset of $\E_\hbar(W)$. Then
there exists a (unique) smallest $\hbar$-adically compatible $\C[[\hbar]]$-submodule $\<U\>$ of $\E_\hbar(W)$
such that $\<U\>$ contains $U\cup \{1_W\}$, $\<U\>$ is $Y_\E$-closed, $[\<U\>]=\<U\>$, and $\<U\>$ is $\hbar$-adically complete.
\end{thm}

By a straightforward argument we have:

\begin{lem}\label{module-rep}
Let $(W,Y_W)$ be a module for an $\hbar$-adic nonlocal vertex algebra $V$. For $u,v\in V$,
if $(Y_W(u,x),Y_W(v,x))$ is $\hbar$-adically compatible, then
\begin{align}
Y_W(Y(u,z)v,x)=Y_{\E}(Y_W(u,x),z)Y_W(v,x).
\end{align}
\end{lem}

Define a $\C[[\hbar]]$-module map
\begin{eqnarray}
 Z_W(x_1,x_2):\  \ \E_\hbar(W)\widehat\ot\E_\hbar(W)\widehat\ot\C((z))[[\hbar]]\to (\End W)[[x_1^{\pm 1},x_2^{\pm1}]]\nonumber
\end{eqnarray}
by
 \begin{eqnarray}
 Z_W(x_1,x_2)( a(x)\ot b(x)\ot f(z))= f(x_1-x_2)a(x_1)b(x_2).
\end{eqnarray}

 \begin{de}
 A subset $U$ of $\E_\hbar(W)$ is said to be \emph{$\SY$-local}
 if for any $a(x),b(x)\in U$, there exists $A(z)\in(\C U\ot \C U\ot \C((z)))[[\hbar]]$ such that
\begin{eqnarray}
  a(x_1)b(x_2)\sim Z_W(x_2,x_1)(A(z)).
\end{eqnarray}
\end{de}

 We have (see \cite[Lemmas 4.23]{Li-h-adic}):

\begin{lem}\label{W-hwqva}
Let $U$ be an $\mathcal{S}$-local subset of $\E_{\hbar}(W)$.
Then $U$  is $\hbar$-adically compatible.
\end{lem}

The following is a slight modification of \cite[Theorem 4.24]{Li-h-adic}  (cf.  \cite[Thm 2.9]{Li-qva2}):

\begin{thm}\label{hwqva-construction}
Let $V$ be a topologically free $\C[[\hbar]]$-module, $U$ a subset, ${\bf 1}$ a vector of $V$, and $Y_0(\cdot,x)$ a map
from $U$ to $\E_{\hbar}(V)$. Assume all the following conditions hold:
\begin{align}
Y_0(u,x){\bf 1}\in V[[x]]\ \text{ and }\ \lim_{x\rightarrow 0}Y_0(u,x){\bf 1}=u\ \text{ for }u\in U,
\end{align}
$U(x):=\!\{ u(x):=Y_0(u,x)=\sum_{n\in \Z}u_nx^{-n-1} | u\in U\}$ is $\SY$-local, and
$V=\overline{[U^{\infty}]}$, where
\begin{align}
U^{\infty}={\rm span}_{\C[[\hbar]]}\{ u^{(1)}_{n_1}\cdots u^{(r)}_{n_r}{\bf 1}\ |\ r\ge 0,\ u^{(i)}\in U,\ n_i\in \Z\}.
\end{align}
In addition, assume that there exists a $\C[[\hbar]]$-module map $\psi_x: V\rightarrow \<U(x)\>$ such that $\psi_x({\bf 1})=1_V$ and
\begin{align}
\psi_x(u_nv)=u(x)_n\psi_x(v)\quad \text{ for }u\in U,\ n\in \Z,\ v\in V.
\end{align}
Then $Y_0(\cdot,x)$ can be extended uniquely to a $\C[[\hbar]]$-module map $Y(\cdot,x): V\rightarrow \E_{\hbar}(V)$ such that
$(V,Y,{\bf 1})$ carries the structure of an $\hbar$-adic nonlocal vertex algebra and
$\psi_x$ is an isomorphism of $\hbar$-adic nonlocal vertex algebras.
Furthermore, if $V=U^{\infty}[[\hbar]]'$, then $V$ is an $\hbar$-adic weak quantum vertex algebra.
\end{thm}

\begin{proof} First, with $U(x)$ an $\SY$-local subset of $\E_{\hbar}(V)$ we have an $\hbar$-adic nonlocal vertex algebra $\<U(x)\>$ with
 $V$ as a module. Second, by Lemma \ref{U-vacuum-like} we get $a(x){\bf 1}\in V[[x]]$ for $a(x)\in \<U(x)\>$ and
 the map $\phi: \<U(x)\>\rightarrow V$, defined by $\phi(a(x))=(a(x){\bf 1})|_{x=0}$ for $a(x)\in \<U(x)\>$
  is a $\<U(x)\>$-module morphism with $\phi(1_V)={\bf 1}$. Notice that
$$\psi_x(u)=\psi_x(u_{-1}{\bf 1})=u(x)_{-1}1_V=u(x)\ \text{ and }\ \phi(u(x))=\lim_{x\rightarrow 0}u(x){\bf 1}=u$$
for $u\in U$. It follows that
$$\phi\circ \psi_x=1\ \text{ on }U^{\infty}\ \text{ and }\ \psi_x\circ \phi=1\ \text{ on }U(x)^{\infty}.$$
Then $\phi\circ \psi_x=1\ \text{ on }\overline{[U^{\infty}]}=V.$
On the other hand, we have $\psi_x\circ \phi=1$ on $\overline{U(x)^{\infty}}$. As $U(x)^{(\infty)}\subset \overline{U(x)^{\infty}}$
by Lemma \ref{U-infinity-facts}, we have
$\psi_x\circ \phi=1$ on $U(x)^{(\infty)}$, which implies $\psi_x\circ \phi=1$ on $\overline{[U(x)^{(\infty)}]}=\<U(x)\>$.
Thus, $\psi_x$ and $\phi$ are inverses each other.
The rest follows from the proof of \cite[Theorem 4.24]{Li-h-adic}.
\end{proof}

The following is a straightforward $\hbar$-adic analogue of (\cite[Theorem 2.11]{Li-qva2}):

\begin{thm}\label{hwqva-module}
Let $V$ be an $\hbar$-adic weak quantum vertex algebra, let $U$ be a subset of $V$ such that $V=\<U\>$,
let $W$ be a topologically free $\C[[\hbar]]$-module, and
let $Y_W^{0}(\cdot,x): U\rightarrow \E_{\hbar}(W)$ be a map.
Suppose that $Y_W^{0}(U)$ is $\SY$-local and that there
  exists a $\C[[\hbar]]$-module map $\psi_x: V\rightarrow \<Y_W^0(U)\>$ such that $\psi_x({\bf 1})=1_W$ and
\begin{align}
\psi_x(u_nv)=u(x)_n\psi_x(v)\quad \text{ for }u\in U,\ n\in \Z,\ v\in V,
\end{align}
where $u(x):=Y_W^0(u,x)$.
Then $Y_W^0(\cdot,x)$ can be extended uniquely to a $\C[[\hbar]]$-module map $Y_W(\cdot,x): V\rightarrow \E_{\hbar}(W)$
such that $(W,Y_W)$ carries the structure of a $V$-module.
\end{thm}

To conclude this subsection, we present some technical results which we need.
Let $W$ be a topologically free $\C[[\hbar]]$-module.
The following is a straightforward $\hbar$-adic version of a standard result:

\begin{lem}\label{iterate-module}
Let $a(x),b(x)\in \E_{\hbar}(W)$, $A(z)\in \E_{\hbar}(W)\wh\ot \E_{\hbar}(W)\wh\ot \C((z)))[[\hbar]]$ such that
\begin{eqnarray}
  a(x_1)b(x_2)\sim Z_W(x_2,x_1)(A(z)).
\end{eqnarray}
Then $(a(x),b(x))$ is $\hbar$-adically compatible and
\begin{align}
&x_0^{-1}\delta\!\left(\frac{x_1-x}{x_0}\right)a(x_1)b(x)-x_0^{-1}\delta\!\left(\frac{x-x_1}{-x_0}\right)\!Z_W(x,x_1)(A(z))\\
&\quad\quad =x_1^{-1}\delta\!\left(\frac{x+x_0}{x_1}\right)\!Y_\E(a(x),x_0)b(x).\nonumber
\end{align}
\end{lem}

\begin{lem}\label{h-adic-compatibility}
Assume
$$a(x),b(x)\in \E_{\hbar}(W), \  A(z)\in \E_{\hbar}(W)\wh\otimes \E_{\hbar}(W)\wh\otimes \C((z))[[\hbar]],\ \  p(x,y)\in \C(x,y)$$
 with $p(x,y)\ne 0$ and $\iota_{x,y}(p(x,y)^{\pm 1})\in \C((x))[[y]]$ (analytic at $y=0$)
such that
\begin{align}
p(x_1-x_2,\hbar)a(x_1)b(x_2)=Z_W(x_2,x_1)(A(z)),
\end{align}
where $p(x_1-x_2,\hbar):=(\iota_{x,y}p(x,y))|_{x=x_1-x_2,y=\hbar}$.
Then
\begin{align}
a(x_1)b(x_2)\sim p(-x_2+x_1,\hbar)^{-1}Z_W(x_2,x_1)(A(z)).
\end{align}
\end{lem}

\begin{proof} Let $n$ be any positive integer.
Write $\iota_{x,y}(p(x,y)^{-1})=\sum_{i\ge 0}p_i^{-}(x)y^i$ with $p_i^{-}(x)\in \C((x))$.
Let $k\in \N$ such that
$x^kp_i^{-}(x)\in \C[[x]]$ for all $0\le i\le n-1$. Then we obtain
\begin{align*}
&(x_1-x_2)^ka(x_1)b(x_2)\\
=\ & \sum_{i\ge 0}(x_1-x_2)^kp_i^{-}(x_1-x_2)\hbar ^i \(p(x_1-x_2,\hbar)a(x_1)b(x_2)\)\\
\equiv \ & \sum_{0\le i<n}(x_1-x_2)^kp_i^{-}(x_1-x_2)\hbar ^i \(p(x_1-x_2,\hbar)a(x_1)b(x_2)\)\ \mod \hbar^n\\
= \ & \sum_{0\le i<n}\((-x_2+x_1)^kp_i^{-}(-x_2+x_1)\)\hbar ^i \(p(x_1-x_2,\hbar)a(x_1)b(x_2)\)\\
= \ & \sum_{0\le i<n}(-x_2+x_1)^kp_i^{-}(-x_2+x_1)\hbar ^i \(Z_W(x_2,x_1)A(z)\)\\
\equiv \ & \sum_{i\ge 0}(-x_2+x_1)^kp_i^{-}(-x_2+x_1)\hbar ^i \(Z_W(x_2,x_1)A(z)\)\ \mod \hbar^n\\
=\ &(x_1-x_2)^kp(-x_2+x_1,\hbar)^{-1}Z_W(x_2,x_1)(A(z)),
\end{align*}
as desired.
\end{proof}

Using Lemma \ref{h-adic-compatibility} we immediately have:

\begin{lem}\label{new-added}
Let $U$ be a subset of $\E_{\hbar}(W)$. Assume that for $a(x),b(x)\in U$, there exist
$$A(z)\in (\C U\otimes \C U\otimes \C((z)))[[\hbar]],\ \  p(x,y)\in \C(x,y)$$
 with $p(x,y)\ne 0$ and $\iota_{x,y}(p(x,y)^{\pm 1})\in \C((x))[[y]]$ such that
\begin{align}
p(x_1-x_2,\hbar)a(x_1)b(x_2)=Z_W(x_2,x_1)(A(z)).
\end{align}
Then $U$ is $\SY$-local.
\end{lem}

Using Lemmas \ref{Y-V-W-a-b-relation} and \ref{pq-va-module-relation}, we immediately have:

\begin{lem}\label{lem-a-b-C-D}
Let $U$ be an $\hbar$-adically compatible subset of $\E_{\hbar}(W)$, and let
$$a(z),b(z)\in U, \ R(x), f(x), g(x)\in \C((x))[[\hbar]],\ C(z),D(z)\in \E_{\hbar}(W)$$
such that
\begin{align}\label{aba=gba-CD}
a(z)b(w)-R(w-z)b(w)a(z)=f(z-w)C(w)+g(w-z)D(w).
\end{align}
Then $C(z),D(z)\in \<U\>$
 and
\begin{align}\label{YE-ab}
&Y_{\E}(a(x),z)Y_{\E}(b(x),w)-R(w-z)Y_{\E}(b(x),w)Y_{\E}(a(x),z)\nonumber\\
&\quad = f(z-w)Y_{\E}(C(x),w)+g(w-z)Y_{\E}(D(x),w).
\end{align}
\end{lem}

\begin{lem}\label{lem-fab=gba}
Let $U$ be an $\hbar$-adically compatible subset of $\E_{\hbar}(W)$ with $a(z),b(z)\in U$. Assume
\begin{eqnarray}\label{aba=gba}
p(z-w,\hbar)a(z)b(w)=q(w-z,\hbar)b(w)a(z),
\end{eqnarray}
where $p(x,z),q(x,z)\in \C[x,z]$ with $p(x,0)\ne 0$.
Then
\begin{align}\label{fY-gY}
p(z-w,\hbar)Y_{\E}(a(x),z)Y_{\E}(b(x),w)=q(w-z,\hbar)Y_{\E}(b(x),w)Y_{\E}(a(x),z).
\end{align}
\end{lem}

\section{Vertex-operator Serre relation and  iterate formula}
In this section, we study a Serre-like relation, give two characterizations, and
establish a vertex-operator iterate formula.

\begin{lem}\label{sing-simple-fact}
Let $U$ be a $\C[[\hbar]]$-module and let $A(z)=\sum_{n\in \Z}u(n)z^{-n-1}\in U[[z,z^{-1}]]$, $\mu\in \C$.
Then $(z-\mu \hbar)A(z)\in U[[z]]$ if and only if ${\rm Sing}_zA(z)=u(0)(z-\mu\hbar)^{-1}$.
\end{lem}

\begin{proof}
Note that $(z-\mu\hbar)A(z)\in U[[z]]$ if and only if $(z-\mu\hbar){\rm Sing}_zA(z)\in U[[z]]$. As
$$(z-\mu \hbar){\rm Sing}_zA(z)-u(0)=z {\rm Sing}_zA(z)-u(0)-\mu\hbar {\rm Sing}_zA(z) \in z^{-1}U[[z^{-1}]],$$
we see that $(z-\mu\hbar){\rm Sing}_zA(z)\in U[[z]]$ implies $(z-\mu\hbar){\rm Sing}_zA(z)=u(0)$, which is equivalent to
 ${\rm Sing}_zA(z)=u(0)(z-\mu\hbar)^{-1}$. Conversely, if ${\rm Sing}_zA(z)=u(0)(z-\mu\hbar)^{-1}$, then
 $(z-\mu\hbar){\rm Sing}_zA(z)=u(0)$, which implies $(z-\mu\hbar)A(z)\in U[[z]]$.
\end{proof}

The following is a straightforward observation (see \cite{JKLT-22}):

\begin{lem}\label{lem-sing-fact}
Let $b\in \C,\ F(x,\hbar)\in \C[[x,\hbar]]$. Then
\begin{eqnarray}
&&{\rm Sing}_x (x-b\hbar)^{-1}F(x,\hbar)=(x-b\hbar)^{-1}F(b\hbar, \hbar),\label{singular-fact}\\
&&\Res_x(x-b \hbar)^{-1}F(x,\hbar)
=F(b \hbar,\hbar).
\end{eqnarray}
\end{lem}

From now on, we assume that $W$ is a topologically free $\C[[\hbar]]$-module.

\begin{lem}\label{a-b-abstract}
Let $a(x),b(x)\in \E_\hbar(W)$. Assume
 \begin{eqnarray}\label{ab-weak-locality}
(x_1-x_2+\mu \hbar)a(x_1)b(x_2)=(x_1-x_2+\nu \hbar)b(x_2)a(x_1),
\end{eqnarray}
where  $\mu,\nu\in \C$. Then $(a(x),b(x))$ is $\hbar$-adically compatible and
\begin{eqnarray}
&&{\rm Sing}_zY_{\E}(a(x),z)b(x)=a(x)_0b(x)(z+\mu\hbar)^{-1},\label{singular-a(x)-b(x)}\\
&&x_1\inv\delta\!\(\frac{x-\mu \hbar}{x_1}\)\!a(x)_0b(x)=a(x_1)b(x)- \!\(\frac{x-x_1-\nu \hbar}{x-x_1-\mu \hbar} \)\!b(x)a(x_1).
\label{a(x)-0-b(x)}
\end{eqnarray}
\end{lem}

\begin{proof} With the assumption (\ref{ab-weak-locality}), by Lemma \ref{lem-fab=gba} we have
$$a(x_1)b(x_2)\sim \(\frac{x_2-x_1-\nu \hbar}{x_2-x_1-\mu \hbar}\)b(x_2)a(x_1),$$
which implies that $(a(x),b(x))$ is $\hbar$-adically compatible, and
\begin{align}\label{Y-E-a(x)-b(x)}
\Res_{z}x_1^{-1}\delta\!\left(\frac{x+z}{x_1}\right)\!Y_{\E}(a(x),z)b(x)
=a(x_1)b(x)- \(\frac{x-x_1-\nu \hbar}{x-x_1-\mu \hbar}\)b(x)a(x_1)
\end{align}
by Lemma \ref{iterate-module}. Then for any nonnegative integer $n$,
\begin{eqnarray*}
&&\quad \Res_z z^n(z+\mu\hbar)Y_{\E}(a(x),z)b(x)\\
&&= \Res_z \Res_{x_1}(x_1-x)^n(x_1-x+\mu\hbar)x_1^{-1}\delta\!\left(\frac{x+z}{x_1}\right)\!Y_{\E}(a(x),z)b(x)\\
&&=\Res_{x_1}(x_1-x)^n\( (x_1-x+\mu \hbar)a(x_1)b(x)-(x_1-x+\nu \hbar)b(x)a(x_1)\)\\
&&=0.
\end{eqnarray*}
This implies that $(z+\mu\hbar)Y_\E(a(x),z)b(x)$ is regular in $z$.
Then (\ref{singular-a(x)-b(x)}) holds by Lemma \ref{sing-simple-fact}.
Furthermore, using Lemma \ref{lem-sing-fact} we have
\begin{align*}
\Res_{z}x_1^{-1}\delta\!\left(\frac{x+z}{x_1}\right)\!Y_{\E}(a(x),z)b(x)
=\ &\Res_{z}x_1^{-1}\delta\!\left(\frac{x+z}{x_1}\right)\!{\rm Sing}_z Y_{\E}(a(x),z)b(x)\nonumber\\
=\ &\Res_{z}x_1^{-1}\delta\!\left(\frac{x+z}{x_1}\right)(z+\mu \hbar)^{-1}a(x)_0b(x)\nonumber\\
=\ &x_1^{-1}\delta\!\left(\frac{x-\mu\hbar}{x_1}\right)a(x)_0b(x).
\end{align*}
Now, combining this with (\ref{Y-E-a(x)-b(x)}) we obtain (\ref{a(x)-0-b(x)}).
\end{proof}

The following is the Serre relation that we need:

\begin{de}\label{def-Serre-relation}
Let $a(x),b(x)\in \E_{\hbar}(W)$, and $m\in\Z_+$.  The relation
\begin{align}
\sum_{\sigma\in S_m}\sum_{r=0}^m\binom{m}{r}(-1)^r a(z_{\sigma(1)})a(z_{\sigma(2)})\cdots a(z_{\sigma(r)})b(w) a(z_{\sigma(r+1)})
\cdots a(z_{\sigma(m)})=0
\end{align}
is referred to as the {\em Serre-like relation for $(a(z), b(z),m)$}.
\end{de}

We use a similar but simpler method to the proof of \cite[Theorem 5.16]{CJKT}
to give another characterization of the Serre-like relation.
First, we present the following two technical results:

\begin{lem}\label{first}
 Let $a(x)\in \E_{\hbar}(W),\ \mu\in \C$ such that
\begin{align}\label{az-aw-comm}
(z-w-\mu \hbar)a(z)a(w)=(z-w+\mu \hbar)a(w)a(z).
\end{align}
Then
\begin{align}
& (z-w-\mu \hbar)a(z)a(w)\in\E_\hbar^{(2)}(W):=\Hom (W,W_{\hbar}((z,w))),\\
&  z\inv\delta\(\frac wz\) \left((z-w-\mu\hbar)a(z)a(w)\right)=0.
\end{align}
Furthermore, we have
\begin{align}
&(z-w)^{-1}(z-w-\mu \hbar)a(z)a(w)=(w-z)^{-1}(w-z-\mu\hbar)a(w)a(z),\label{a-a-(-1)}\\
&  (z-w)\inv (z-w-\mu \hbar)a(z)a(w)\in\E_\hbar^{(2)}(W).\label{a-a-strong}
\end{align}
\end{lem}

\begin{proof} Set
\begin{align*}
 A(z,w)=(z-w-\mu \hbar)a(z)a(w).
\end{align*}
From (\ref{az-aw-comm}), we have $ A(z,w)\in \E_\hbar^{(2)}(W)$, so that $A(z,z)$ exists in $\E_{\hbar}(W)$.
Using (\ref{az-aw-comm}) again
we get
\begin{align*}
 A(z,w)=(z-w+\mu\hbar)a(w)a(z) =-(w-z-\mu\hbar)a(w)a(z)=-A(w,z),
\end{align*}
which implies $A(z,z)=0$. Then
$$ z\inv\delta\!\(\frac wz\)\! \left((z-w-\mu\hbar)a(z)a(w)\right)= z\inv\delta\!\(\frac wz\) A(z,w) = z\inv\delta\!\(\frac wz\) A(z,z) =0,$$
which amounts to
\begin{align}
(z-w)\inv \left((z-w-\mu \hbar)a(z)a(w)\right)=(-w+z)^{-1}((z-w-\mu \hbar)a(z)a(w)).
\end{align}
This together with (\ref{az-aw-comm}) gives  (\ref{a-a-(-1)}), which implies (\ref{a-a-strong}).
\end{proof}

\begin{lem}\label{second}
Let $a(x), b(x)\in \E_\hbar(W)$ and let $\mu\in \C^\times$, $\nu\in \C$ such that
\begin{align}
  &(z-w-2\mu\hbar)a(z)a(w)=(z-w+2\mu\hbar)a(w)a(z),\label{eq:a-a-relation}\\
  &(z-w+\mu\nu\hbar)a(z)b(w)=(z-w-\mu\nu\hbar)b(w)a(z).\label{eq:a-b-relation}
\end{align}
For $r\in\N$, we set
\begin{align*}
  &\nob{a(z_1)\cdots a(z_r)b(w)}=\prod_{1\le i<j\le r}\frac{z_i-z_j-2\mu\hbar}{z_i-z_j} \prod_{1\le i\le r}(z_i-w+\mu\nu\hbar)
  a(z_1)\cdots a(z_r)b(w).
\end{align*}
Then
\begin{align}
  &\nob{a(z_1)\cdots a(z_r)b(w)}\in\E_\hbar^{(r+1)}(W):=\Hom(W,W_\hbar((z_1,\dots, z_r,w))),\label{eq:no-a-a-b}\\
  &a(w)_0^rb(w)\label{eq:no-a-a-b-0}\\
  &\quad=\nob{a(w+(2r-2-\nu)\mu\hbar)a(w+(2r-4-\nu)\mu\hbar)\cdots a(w-\nu\mu\hbar)b(w)}, \nonumber\\
  &a(z)a(w)_0^rb(w)=\frac{w-z+\mu\nu\hbar}{w-z-\mu\nu\hbar}\prod_{s=1}^{r}\lim_{z_s\to w+(2r-2s-\mu\nu)\hbar}
  \frac{z_s-z-2\mu\hbar}{z_s-z+2\mu\hbar} a(w)_0^rb(w)a(z)\label{eq:a-a-b-relation}\\
  &\quad +z\inv\delta\left(\frac{w-(\nu-2r)\mu\hbar}{z}\right)a(w)_0^{r+1}b(w).\nonumber
\end{align}
\end{lem}

\begin{proof}
It is no harm to assume $\mu=1$
Applying Lemma \ref{first} to equation \eqref{eq:a-a-relation}, we get
\begin{align}
  &\frac{z-w-2\hbar}{z-w}a(z)a(w)=\frac{w-z-2\hbar}{w-z}a(w)a(z).\label{eq:a-a-relation-1}
\end{align}
Then the first relation follows immediate from \eqref{eq:a-a-relation-1} and \eqref{eq:a-b-relation}.
We next prove the relation \eqref{eq:no-a-a-b-0} by using induction on $r$.
If $r=0$, there is nothing need to prove.
Suppose it holds for $r$.
Then
\begin{align*}
  & (z-w+\nu\hbar)\prod_{s=1}^r\lim_{z_s\to w+(2r-2s-\nu)\hbar} \frac{z-z_s-2\hbar}{z-z_s} a(z)a(w)_0^rb(w)
   \\
  =&\nob{a(z)a(w+(2r-2-\nu)\hbar)a(w+(2r-4-\nu)\hbar)\cdots a(w-\nu\hbar)b(w)},\nonumber\\
  &(-w+z-\nu\hbar)\prod_{s=1}^{r}\lim_{z_s\to w+(2r-2s-\nu)\hbar}\frac{z_s-z-2\hbar}{z_s-z} a(w)_0^rb(w)a(z)
    \\
  =&\nob{a(z)a(w+(2r-2-\nu)\hbar)a(w+(2r-4-\nu)\hbar)\cdots a(w-\nu\hbar)b(w)}.\nonumber
\end{align*}
Notice that
\begin{align*}
  (z-w+\nu\hbar)\prod_{s=1}^r\lim_{z_s\to w+(2r-2s-\nu)\hbar} \frac{z-z_s-2\hbar}{z-z_s}
  =(z-w+(\nu-2r)\hbar).
\end{align*}
By Lemma \ref{a-b-abstract}, we get that
\begin{align}
  &z\inv\delta\left(\frac{w-(\nu-2r)\hbar}{z}\right)a(w)_0^{r+1}b(w)\nonumber\\
  =&a(z)a(w)_0^rb(w)-\frac{w-z+\nu\hbar}{w-z-\nu\hbar}\prod_{s=1}^{r}\lim_{z_s\to w+(2r-2s-\nu)\hbar}\frac{z_s-z-2\hbar}{z_s-z+2\hbar} a(w)_0^rb(w)a(z)
  \label{eq:a-a-b-relation-temp}\\
  =&z\inv\delta\left(\frac{w-(\nu-2r)\hbar}{z}\right)\nonumber\\
  &\times  \nob{a(w+(2r-\nu)\hbar)a(w+(2r-2-\nu)\hbar)\cdots a(w-\nu\hbar)b(w)}.\nonumber
\end{align}
Taking $\Res_z$ on both hand sides, we get that
\begin{align*}
  &a(w)_0^{r+1}b(w)=\nob{a(w+(2r-\nu)\hbar)a(w+(2r-2-\nu)\hbar)\cdots a(w-\nu\hbar)b(w)},
\end{align*}
as desired.
The last relation follows immediate from the relation \eqref{eq:a-a-b-relation-temp}.
\end{proof}

We introduce some notations and conventions.
For a set $A$, we denote by $|A|$ the cardinal of $A$.
For two integers $m\le n$, we denote by $[m,n]$ the set $\set{k \in \Z}{m\le k\le n}$.
Given a natural number $r\le n-m$ and a subset $I\subset [m,n]$ of cardinal $r$, we write
\begin{align*}
  &I={i_1<i_2<\cdots< i_r},\quad [0,n]\setminus I= {i'_1<i'_2<\cdots < i'_{n-r}}.
\end{align*}

\begin{lem}\label{third}
Let $a(x), b(x)\in \E_\hbar(W)$ and let $\mu\in\C^\times$, $\nu\in \C$ such that \eqref{eq:a-a-relation} and \eqref{eq:a-b-relation} hold.
Then for any $r\in\N$, we have that
\begin{align*}
  &a(z_1)a(z_2)\cdots a(z_r)b(w)
  =\sum_{I\subset [1,r]}\prod_{s=1}^{r-|I|}\bigg(
  \frac{w-z_{i'_s}+\mu\nu\hbar}{w-z_{i'_s}-\mu\nu\hbar}\prod_{\substack{1\le t\le |I|\\ i'_s< i_t}}
  \frac{z_{i_t}-z_{i'_s}-2\mu\hbar}{z_{i_t}-z_{i'_s}+2\mu\hbar}\bigg)
    \\
  &\times a(w)_0^{|I|}b(w)a(z_{i'_1})a(z_{i'_2})\cdots a(z_{i'_{r-|I|}})
  \prod_{t=1}^{|I|} z_{i_t}\inv\delta\left(\frac{w-(\nu-2(|I|-t))\mu\hbar}{z_{i_t}}\right).
\end{align*}
\end{lem}

\begin{proof}
It is no harm to assume $\mu=1$.
If $r=0$, then there is nothing need to prove.
We suppose this lemma holds true for $r$.
From Lemma \ref{second}
\begin{align*}
  &a(z_0)a(z_1)a(z_2)\cdots a(z_r)b(w)\\
  =&\sum_{I\subset [1,r]}\prod_{s=1}^{r-|I|}
    \bigg(\frac{w-z_{i'_s}+\nu\hbar}{w-z_{i'_s}-\nu\hbar}\prod_{\substack{1\le t\le |I|\\ i'_s< i_t}}\frac{z_{i_t}-z_{i'_s}-2\hbar}{z_{i_t}-z_{i'_s}+2\hbar}\bigg)\\
  &\times a(z_0)a(w)_0^{|I|}b(w)a(z_{i'_1})a(z_{i'_2})\cdots a(z_{i'_{r-|I|}})
  \prod_{t=1}^{|I|} z_{i_t}\inv\delta\left(\frac{w-(\nu-2(|I|-t))\hbar}{z_{i_t}}\right)\\
  =&\sum_{I\subset [1,r]}\prod_{s=1}^{r-|I|}
    \bigg(\frac{w-z_{i'_s}+\nu\hbar}{w-z_{i'_s}-\nu\hbar}\prod_{\substack{1\le t\le |I|\\ i'_s< i_t}}\frac{z_{i_t}-z_{i'_s}-2\hbar}{z_{i_t}-z_{i'_s}+2\hbar}\bigg)\\
  &\times a(w)_0^{|I|+1}b(w)a(z_{i'_1})a(z_{i'_2})\cdots a(z_{i'_{r-|I|}})
  \prod_{t=1}^{|I|} z_{i_t}\inv\delta\left(\frac{w-(\nu-2(|I|-t))\hbar}{z_{i_t}}\right)\\
  &\times z_0\inv \delta\left(\frac{w-2|I|\hbar}{z_0}\right)\\
  +&\sum_{I\subset [1,r]}\prod_{s=1}^{r-|I|}
    \bigg(\frac{w-z_{i'_s}+\nu\hbar}{w-z_{i'_s}-\nu\hbar}\prod_{\substack{1\le t\le |I|\\ i'_s< i_t}}\frac{z_{i_t}-z_{i'_s}-2\hbar}{z_{i_t}-z_{i'_s}+2\hbar}\bigg)
  \times \frac{w-z_0+\nu\hbar}{w-z_0-\nu\hbar} \prod_{t=1}^{|I|}\frac{z_{i_t}-z_0-2\hbar}{z_{i_t}-z_0+2\hbar} \\
  &\times a(w)_0^{|I|}b(w)a(z_0)a(z_{i'_1})a(z_{i'_2})\cdots a(z_{i'_{r-|I|}})
  \prod_{t=1}^{|I|} z_{i_t}\inv\delta\left(\frac{w-(\nu-2(|I|-t))\hbar}{z_{i_t}}\right)\\
  =&\sum_{I\subset [0,r]}\prod_{s=1}^{r+1-|I|}
    \bigg(\frac{w-z_{i'_s}+\nu\hbar}{w-z_{i'_s}-\nu\hbar}\prod_{\substack{1\le t\le |I|\\ i'_s< i_t}}\frac{z_{i_t}-z_{i'_s}-2\hbar}{z_{i_t}-z_{i'_s}+2\hbar}\bigg)\\
  &\times a(w)_0^{|I|}b(w)a(z_{i'_1})a(z_{i'_2})\cdots a(z_{i'_{r-|I|}})
  \prod_{t=1}^{|I|} z_{i_t}\inv\delta\left(\frac{w-(\nu-2(|I|-t))\hbar}{z_{i_t}}\right),
\end{align*}
as desired.
\end{proof}

For a positive integer $m$, we denote by $S_m$ the symmetric group on $[1,m]$.
Given an integer $p\in [0,m]$, we view $S_p$ as a subgroup of $S_m$ in the obvious way and introduce the following subgroups of $S_m$:
\begin{align*}
  &S_m^p:=\set{\sigma\in S_m}{\sigma([p+1,m])\subset [p+1,m],\,\, \te{and}\,\,\sigma(r)=r,\,\,\te{for }r\le p},\\
  &S_{m,p}:=\set{\sigma\in S_m}{\sigma(s)<\sigma(t),\,\,\te{for }s<t\le p\,\,\te{or }p<s<t}.
\end{align*}
Then we have that
\begin{lem}\label{forth}
Let $a(x), b(x)\in \E_\hbar(W)$ and let $\mu\in\C^\times$, $\nu\in\N$ such that \eqref{eq:a-a-relation} and \eqref{eq:a-b-relation} hold.
Set $m=\nu+1$.
Then
\begin{align*}
  &\sum_{\tau\in S_m}\sum_{r=0}^m\binom{m}{r}(-1)^r a(z_{\tau(1)})\cdots a(z_{\tau(r)})b(w)a(z_{\tau(r+1)})\cdots a(z_{\tau(m)})\\
  =&\sum_{\tau\in S_m}\sum_{p=0}^m
  a(w)_0^pb(w)a(z_{\tau(p+1)})a(z_{\tau(p+2)})\cdots a(z_{\tau(m)})
  \prod_{s=1}^{p} z_{\tau(s)}\inv\delta\left(\frac{w-(\nu-2(p-s))\mu\hbar}{z_{\tau(s)}}\right)\\
  &\times \sum_{\sigma\in S_{m,p}}\sum_{r=\sigma(p)}^m
  \binom{m}{r}(-1)^r
  \prod_{s=p+1}^{r}\frac{w-z_{\tau(s)}+\mu\nu\hbar}{w-z_{\tau(s)}-\mu\nu\hbar}
  \prod_{\substack{1\le s\le p<t\le r\\ \sigma(s)> \sigma(t)}}\frac{z_{\tau(s)}-z_{\tau(t)}-2\mu\hbar}{z_{\tau(s)}-z_{\tau(t)}+2\mu\hbar}.
\end{align*}
\end{lem}

\begin{proof}
It is no harm to assume that $\mu=1$.
For a subset $I\subset [1,m]$, we set $m(I)=\max I$ if $I\ne\emptyset$ or $0$ if $I=\emptyset$.
From Lemma \ref{third}, we have that
\begin{align*}
  &\sum_{r=0}^m\binom{m}{r}(-1)^r a(z_1)\cdots a(z_r)b(w)a(z_{r+1})\cdots a(z_{m})\\
  =&\sum_{r=0}^m\binom{m}{r}(-1)^r\sum_{I\subset[1,r]}
  \prod_{s=1}^{r-|I|}\bigg(\frac{w-z_{i'_s}+\nu\hbar}{w-z_{i'_s}-\nu\hbar}
  \prod_{\substack{1\le t\le |I|\\ i'_s< i_t}}\frac{z_{i_t}-z_{i'_s}-2\hbar}{z_{i_t}-z_{i'_s}+2\hbar}\bigg)\\
  &\times  a(w)_0^{|I|}b(w)a(z_{i'_1})a(z_{i'_2})\cdots a(z_{i'_{r-|I|}})a(z_{r+1})\cdots a(z_m)\\
  &\times \prod_{s=1}^{|I|} z_{i_s}\inv\delta\left(\frac{w-(\nu-2(|I|-s))\hbar}{z_{i_s}}\right)\\
  =&\sum_{I\subset[1,m]} \sum_{r=m(I)}^m
  \binom{m}{r}(-1)^r a(w)_0^{|I|}b(w)a(z_{i'_1})a(z_{i'_2})\cdots a(z_{i'_{m-|I|}})\\
  &\times \prod_{s=1}^{r-|I|}\bigg(\frac{w-z_{i'_s}+\nu\hbar}{w-z_{i'_s}-\nu\hbar}
  \prod_{\substack{1\le t\le |I|\\ i'_s< i_t}}\frac{z_{i_t}-z_{i'_s}-2\hbar}{z_{i_t}-z_{i'_s}+2\hbar}\bigg)
  \prod_{s=1}^{|I|} z_{i_s}\inv\delta\left(\frac{w-(\nu-2(|I|-s))\hbar}{z_{i_s}}\right)\\
  =&\sum_{p=0}^m\sum_{\sigma\in S_{m,p}}\sum_{r=\sigma(p)}^m
  \binom{m}{r}(-1)^r a(w)_0^pb(w)a(z_{\sigma(p+1)})a(z_{\sigma(p+2)})\cdots a(z_{\sigma(m)})\\
  &\times \prod_{s=p+1}^{r}\frac{w-z_{\sigma(s)}+\nu\hbar}{w-z_{\sigma(s)}-\nu\hbar}
  \prod_{\substack{1\le s\le p<t\le r\\ \sigma(s)> \sigma(t)}}\frac{z_{\sigma(s)}-z_{\sigma(t)}-2\hbar}{z_{\sigma(s)}-z_{\sigma(t)}+2\hbar}\\
  &\times\prod_{s=1}^{p} z_{\sigma(s)}\inv\delta\left(\frac{w-(\nu-2(p-s))\hbar}{z_{\sigma(s)}}\right),
\end{align*}
Then
\begin{align*}
  &\sum_{\tau\in S_m}\sum_{r=0}^m\binom{m}{r}(-1)^r a(z_{\tau(1)})\cdots a(z_{\tau(r)})b(w)a(z_{\tau(r+1)})\cdots a(z_{\tau(m)})\\
  =&\sum_{p=0}^m\sum_{\tau\in S_m}\sum_{\sigma\in S_{m,p}}\sum_{r=\sigma(p)}^m
  \binom{m}{r}(-1)^r  a(w)_0^pb(w)a(z_{\tau\sigma(p+1)})a(z_{\tau\sigma(p+2)})\cdots a(z_{\tau\sigma(m)})\\
  &\times\prod_{s=p+1}^{r}\frac{w-z_{\tau\sigma(s)}+\nu\hbar}{w-z_{\tau\sigma(s)}-\nu\hbar}
  \prod_{\substack{1\le t\le p<s\le r\\ \sigma(s)< \sigma(t)}}
    \frac{z_{\tau\sigma(t)}-z_{\tau\sigma(s)}-2\hbar}{z_{\tau\sigma(t)}-z_{\tau\sigma(s)}+2\hbar}\\
  &\times
  \prod_{s=1}^{p} z_{\tau\sigma(s)}\inv\delta\left(\frac{w-(\nu-2(p-s))\hbar}{z_{\tau\sigma(s)}}\right)\\
  =&\sum_{\tau\in S_m}\sum_{p=0}^m
  a(w)_0^pb(w)a(z_{\tau(p+1)})a(z_{\tau(p+2)})\cdots a(z_{\tau(m)})
  \prod_{s=1}^{p} z_{\tau(s)}\inv\delta\left(\frac{w-(\nu-2(p-s))\hbar}{z_{\tau(s)}}\right)\\
  &\times \sum_{\sigma\in S_{m,p}}\sum_{r=\sigma(p)}^m
  \binom{m}{r}(-1)^r \prod_{s=p+1}^{r}\frac{w-z_{\tau(s)}+\nu\hbar}{w-z_{\tau(s)}-\nu\hbar}
  \prod_{\substack{1\le s\le p<t\le r\\ \sigma(s)> \sigma(t)}}\frac{z_{\tau(s)}-z_{\tau(t)}-2\hbar}{z_{\tau(s)}-z_{\tau(t)}+2\hbar},
\end{align*}
as desired.
\end{proof}

\begin{lem}\label{fifth}
For integers $0\le p<r\le m$, $\mu\in\C^\times$, $\nu\in\C$ and $\sigma\in S_{m,p}$, we have that
\begin{align}
  &\prod_{\substack{1\le t\le p<s\le r\\ \sigma(s)< \sigma(t)}}\frac{z_t-z_s-2\mu\hbar}{z_t-z_s+2\mu\hbar}
  =\prod_{1\le s<t\le m}\frac{z_{\sigma\inv(s)}-z_{\sigma\inv(t)}+2\mu\hbar}{z_s-z_t+2\mu\hbar},\label{eq:fifth-1}\\
  &\prod_{s=p+1}^{r}\frac{w-z_s+\mu\nu\hbar}{w-z_s-\mu\nu\hbar}
  =\prod_{1\le s\le p}\frac{1}{w-z_s+\mu\nu\hbar}\prod_{p<s\le m}\frac{1}{w-z_s-\mu\nu\hbar}\label{eq:fifth-2}\\
  &\quad\times
  \prod_{1\le s\le r}(w-z_{\sigma\inv(s)}+\mu\nu\hbar)\prod_{r<s\le m}(w-z_{\sigma\inv(s)}-\mu\nu\hbar).\nonumber
\end{align}
\end{lem}

\begin{proof}
It is no harm to assume that $\mu=1$.
Notice that
\begin{align}\label{eq:shuffle}
  &\sigma(s)=s\quad\te{for }s>\sigma(p).
\end{align}
Then for any $s>\sigma(p)$, we have that
\begin{align*}
  &\sigma(s)=s>\sigma(p)\ge\sigma(t)\quad\te{for }t\le p.
\end{align*}
Since $\sigma\in S_{m,p}$, we also have that
\begin{align*}
  &\set{(t,s)}{1\le t\le p<s\le \nu+1,\,\,\sigma(s)<\sigma(t)}\\
  =&\set{(t,s)}{1\le t<s\le \nu+1,\,\,\sigma(s)<\sigma(t)}.
\end{align*}
Then
\begin{align}
  &\prod_{\substack{1\le t\le p<s\le r\\ \sigma(s)< \sigma(t)}}\frac{z_t-z_s-2\hbar}{z_t-z_s+2\hbar}
  =\prod_{\substack{1\le t\le p<s\le m\\ \sigma(s)< \sigma(t)}}\frac{z_t-z_s-2\hbar}{z_t-z_s+2\hbar}
  =\prod_{\substack{1\le t<s\le m\\ \sigma(s)< \sigma(t)}}\frac{z_t-z_s-2\hbar}{z_t-z_s+2\hbar}
  .\label{eq:fifth-1-1}
\end{align}
On the other hand,
\begin{align}
  &\prod_{1\le s<t\le m}\frac{z_{\sigma\inv(s)}-z_{\sigma\inv(t)}+2\hbar }{z_s-z_t+2\hbar}\nonumber\\
  =&\prod_{\substack{1\le s<t \le m\\ \sigma\inv(s)<\sigma\inv(t)}} \frac{z_{\sigma\inv(s)} -z_{\sigma\inv(t)}+2\hbar }{ z_s-z_t+2\hbar }
  \prod_{\substack{1\le s<t \le m\\ \sigma\inv(s)>\sigma\inv(t)}} \frac{z_{\sigma\inv(s)} -z_{\sigma\inv(t)}+2\hbar }{ z_s-z_t+2\hbar }
  \nonumber\\
  =&(-1)^{|\sigma|}\prod_{1\le s<t\le m}
  \frac{1}{z_s-z_t+2\hbar}
  \prod_{\substack{ s<t \\ \sigma\inv(s)<\sigma\inv(t)}} (z_{\sigma\inv(s)} -z_{\sigma\inv(t)}+2\hbar )\nonumber\\
  &\times\prod_{\substack{ t<s \\ \sigma\inv(s)<\sigma\inv(t)}} (z_{\sigma\inv(s)} -z_{\sigma\inv(t)}-2\hbar )\nonumber\\
  =&(-1)^{|\sigma|}\prod_{1\le s<t\le m}
  \frac{1}{z_s-z_t+2\hbar}
  \prod_{\substack{ s<t \\ \sigma(s)<\sigma(t)}} (z_s -z_t+2\hbar )
  \prod_{\substack{ s<t \\ \sigma(t)<\sigma(s)}} (z_s -z_t-2\hbar )\nonumber\\
  =&(-1)^{|\sigma|}\prod_{\substack{1\le s<t \le m\\ \sigma(t)<\sigma(s)}}
  \frac{z_s-z_t-2\hbar}{z_s-z_t+2\hbar}.\label{eq:fifth-1-2}
\end{align}
Combining \eqref{eq:fifth-1-1} and \eqref{eq:fifth-1-2}, we complete the proof of the relation \eqref{eq:fifth-1}.

Next, we start to prove the relation \eqref{eq:fifth-2}.
From \eqref{eq:shuffle}, we have that
\begin{align*}
  &\sigma(s)=s\quad\te{for }s>r>p=\sigma(r),\quad \sigma([1,r])=[1,r].
\end{align*}
Then
\begin{align*}
  &\prod_{s=p+1}^r\frac{w-z_s+\nu\hbar}{w-z_s-\nu\hbar}\\
  =&\prod_{1\le s\le p}\frac{1}{w-z_s+\nu\hbar}\prod_{p<s\le\nu+1}\frac{1}{w-z_s-\nu\hbar}
  \prod_{1\le s\le r}(w-z_s+\nu\hbar)\prod_{r<s\le\nu+1}(w-z_s-\nu\hbar)\\
  =&\prod_{1\le s\le p}\frac{1}{w-z_s+\nu\hbar}\prod_{p<s\le\nu+1}\frac{1}{w-z_s-\nu\hbar}
  \prod_{1\le s\le r}(w-z_{\sigma\inv(s)}+\nu\hbar)\\
  &\times\prod_{r<s\le\nu+1}(w-z_{\sigma\inv(s)}-\nu\hbar),
\end{align*}
as desired.
\end{proof}

Let $\mu\in\C^\times$, $\nu\in\N$ and let $m=\nu+1$.
For any $p\in [0,m]$, we set
\begin{align}
  &Dr_p(z_1,\dots,z_m,w)\nonumber\\
  =&a(w)_0^pb(w)a(z_{p+1})a(z_{p+2})\cdots a(z_m)
  \prod_{s=1}^{p} z_s\inv\delta\left(\frac{w-(\nu-2(p-s))\mu\hbar}{z_s}\right)\label{eq:def-Drp}\\
  &\times \sum_{\sigma\in S_{m,p}}\sum_{r=\sigma(p)}^m
  \binom{m}{r}(-1)^r
  \prod_{s=p+1}^{r}\frac{w-z_s+\mu\nu\hbar}{w-z_s-\mu\nu\hbar}
  \prod_{\substack{1\le s\le p<t\le r\\ \sigma(s)> \sigma(t)}}\frac{z_s-z_t-2\mu\hbar}{z_s-z_t+2\mu\hbar}.\nonumber
\end{align}
And for a subset $H\subset S_m$, we define
\begin{align*}
  &P_m(H,z_1,\dots, z_m,w)=\sum_{\sigma\in H}\sum_{r=0}^m\binom{m}{r}(-1)^{r+|\sigma|}
  \prod_{1\le s<t\le m}(z_{\sigma\inv (s)}-z_{\sigma\inv(t)}+2\mu\hbar)\\
  &\quad\times\prod_{s=1}^{r}(w-z_{\sigma\inv(s)}+\nu\hbar)\prod_{s=r+1}^m(w-z_{\sigma\inv(s)}-\mu\nu\hbar).
\end{align*}
\begin{lem}\label{sixth}
Let $a(x), b(x)\in \E_\hbar(W)$ and let $\mu\in\C^\times$, $\nu\in\N$, $m=\nu+1$, $p\in [0,m-1]$ such that \eqref{eq:a-a-relation} and \eqref{eq:a-b-relation} hold.
Then
\begin{align*}
  &\prod_{s=1}^{p} z_s\inv\delta\left(\frac{w-(\nu-2(p-s))\mu\hbar}{z_s}\right)
  \prod_{1\le t<s\le m}(z_t-z_s+2\mu\hbar)\\
  &\times\prod_{1\le s\le p}(w-z_s+\mu\nu\hbar)\prod_{p<s\le m}(w-z_s-\mu\nu\hbar)
  \ne 0.
\end{align*}
Moreover,
\begin{align*}
  &Dr_p(z_1,\dots,z_m,w)\\
  =&a(w)_0^pb(w)a(z_{p+1})a(z_{p+2})\cdots a(z_m)
  \prod_{s=1}^{p} z_s\inv\delta\left(\frac{w-(\nu-2(p-s))\mu\hbar}{z_s}\right)\\
  &\times\prod_{1\le s\le p}\frac{1}{w-z_s+\mu\nu\hbar}\prod_{p<s\le m}\frac{1}{w-z_s-\mu\nu\hbar}
  \prod_{1\le s<t\le m}\frac{1}{z_s-z_t+2\mu\hbar}\\
  &\times P_m(S_{m,p}S_p,z_1,\dots, z_m,w).
\end{align*}
\end{lem}

\begin{proof}
The first assertion is clear.
Applying Lemma \ref{fifth} to \eqref{eq:def-Drp}, we get that
\begin{align*}
  &Dr_p(z_1,\dots,z_m,w)\\
  =&
  a(w)_0^pb(w)a(z_{p+1})a(z_{p+2})\cdots a(z_{m})
  \prod_{s=1}^{p} z_s\inv\delta\left(\frac{w-(\nu-2(p-s))\mu\hbar}{z_s}\right)\\
  &\times \prod_{1\le s\le p}\frac{1}{w-z_s+\mu\nu\hbar}\prod_{p<s\le m}\frac{1}{w-z_s-\mu\nu\hbar}
  \prod_{1\le s<t\le m}\frac{1}{z_s-z_t+2\mu\hbar}\\
  &\times \sum_{\sigma\in S_{m,p}}\sum_{r=\sigma(p)}^m
  \binom{m}{r}(-1)^{r+|\sigma|}
  \prod_{1\le s<t\le m}(z_{\sigma\inv(s)}-z_{\sigma\inv(t)}+2\mu\hbar)\\
  &\times
  \prod_{1\le s\le r}(w-z_{\sigma\inv(s)}+\mu\nu\hbar)
  \prod_{r<s\le m}(w-z_{\sigma\inv(s)}-\mu\nu\hbar).
\end{align*}
If $r<\sigma(p)$, then there exists $s>r$ such that $p=\sigma\inv(s)$. It follows that
\begin{align*}
  &z_p\inv\delta\left(\frac{w-\mu\nu\hbar}{z_p}\right)\prod_{1\le s\le r}(w-z_{\sigma\inv(s)}+\mu\nu\hbar)
  \prod_{r<s\le m}(w-z_{\sigma\inv(s)}-\mu\nu\hbar)=0.
\end{align*}
By using this, we get that
\begin{align*}
  &Dr_p(z_1,\dots,z_m,w)\\
  =& a(w)_0^pb(w)a(z_{p+1})a(z_{p+2})\cdots a(z_m)
  \prod_{s=1}^{p} z_s\inv\delta\left(\frac{w-(\nu-2(p-s))\mu\hbar}{z_s}\right)\\
  &\times \prod_{1\le s\le p}\frac{1}{w-z_s+\mu\nu\hbar}\prod_{p<s\le m}\frac{1}{w-z_s-\mu\nu\hbar}
  \prod_{1\le s<t\le m}\frac{1}{z_s-z_t+2\mu\hbar}\\
  &\times \sum_{\sigma\in S_{m,p}}\sum_{r=0}^m
  \binom{m}{r}(-1)^{r+|\sigma|}
  \prod_{1\le s<t\le m}(z_{\sigma\inv(s)}-z_{\sigma\inv(t)}+2\mu\hbar)\\
  &\times
  \prod_{1\le s\le r}(w-z_{\sigma\inv(s)}+\mu\nu\hbar)
  \prod_{r<s\le m}(w-z_{\sigma\inv(s)}-\mu\nu\hbar).
\end{align*}
For any $1\ne g\in S_p$, we choose $1\le t< p$ such that $1\le g(t+1)<g(t)\le p$.
Then $\sigma g(t+1)<\sigma g(t)$, as $\sigma\in S_{m,p}$.
It implies
\begin{align*}
  &(z_{g(t+1)}-z_{g(t)}+2\mu\hbar)\quad\te{divides}\quad\prod_{1\le t<s\le m}(z_{\sigma\inv(t)} -z_{\sigma\inv(s)}+2\mu\hbar).
\end{align*}
So
\begin{align*}
  &(z_{t+1}-z_{t}+2\mu\hbar)\quad\te{divides}\quad\prod_{1\le t<s\le m}(z_{g\inv\sigma\inv(t)} -z_{g\inv\sigma\inv(s)}+2\mu\hbar).
\end{align*}
It implies that
\begin{align*}
  &\prod_{s=1}^{p} z_{\tau(s)}\inv\delta\left(\frac{w-(\nu-2(p-s))\mu\hbar}{z_{\tau(s)}}\right)
  \prod_{1\le s<t\le m}(z_{g\inv\sigma\inv(s)} -z_{g\inv\sigma\inv(t)}+2\mu\hbar)=0.
\end{align*}
Therefore,
\begin{align*}
  &Dr_p(z_1,\dots,z_m,w)\\
  =& a(w)_0^pb(w)a(z_{p+1})a(z_{p+2})\cdots a(z_m)
  \prod_{s=1}^{p} z_s\inv\delta\left(\frac{w-(\nu-2(p-s))\mu\hbar}{z_s}\right)\\
  &\times \prod_{1\le s\le p}\frac{1}{w-z_s+\mu\nu\hbar}\prod_{p<s\le m}\frac{1}{w-z_s-\mu\nu\hbar}
  \prod_{1\le s<t\le m}\frac{1}{z_s-z_t+2\mu\hbar}\\
  &\times \sum_{\sigma\in S_{m,p}S_p}\sum_{r=0}^m
  \binom{m}{r}(-1)^{r+|\sigma|}
  \prod_{1\le s<t\le m}(z_{\sigma\inv(s)}-z_{\sigma\inv(t)}+2\mu\hbar)\\
  &\times
  \prod_{1\le s\le r}(w-z_{\sigma\inv(s)}+\mu\nu\hbar)
  \prod_{r<s\le m}(w-z_{\sigma\inv(s)}-\mu\nu\hbar),
\end{align*}
as desired.
\end{proof}

\begin{lem}\label{seventh}
Let $\mu\in\C^\times$, $\nu\in\N$ and let $m=\nu+1$.
Then $P_m(S_m,z_1,\dots, z_m,w)=0$.
\end{lem}

\begin{proof}
It is no harm to assume $\mu=1$.
We prove this lemma by using induction on $m$.
If $m=1$, it is easy to check that $P_1(S_1,z_1,w)=0$.
Suppose that $$P_{m-1}(S_{m-1},z_1,\dots,z_{m-1},w)=0.$$
Notice that
\begin{align*}
  &S_m=S_{m-1}S_{m,m-1}=S_m^{m-1}S_{m,m-1}.
\end{align*}
Then we have that
\begin{align*}
  &P_m(S_m,z_1,\dots,z_m,w)=
  \sum_{\sigma\in S_m}\sum_{r=0}^m\binom{m}{r}(-1)^{r+|\sigma|}
  \prod_{1\le s<t\le m}(z_{\sigma\inv (s)}-z_{\sigma\inv(t)}+2\hbar)\\
  &\quad\times\prod_{s=1}^{r}(w-z_{\sigma\inv(s)}+\nu\hbar)\prod_{s=r+1}^m(w-z_{\sigma\inv(s)}-\nu\hbar)\\
  =&\sum_{\sigma\in S_{m,m-1}}\sum_{\tau\in S_{m-1}}\sum_{r=0}^{m-1}\binom{m-1}{r}(-1)^{r+|\tau\sigma|}
  \prod_{1\le s<t\le m}(z_{\sigma\inv\tau\inv (s)}-z_{\sigma\inv\tau\inv(t)}+2\hbar)\\
  &\quad\times\prod_{s=1}^{r}(w-z_{\sigma\inv\tau\inv(s)}+\nu\hbar)\prod_{s=r+1}^m(w-z_{\sigma\inv\tau\inv(s)}-\nu\hbar)\\
  &-\sum_{\sigma\in S_{m,m-1}}\sum_{\tau\in S_m^{m-1}}\sum_{r=0}^{m-1}\binom{m-1}{r}(-1)^{r+|\tau\sigma|}
  \prod_{1\le s<t\le m}(z_{\sigma\inv\tau\inv (s)}-z_{\sigma\inv\tau\inv(t)}+2\hbar)\\
  &\quad\times\prod_{s=1}^{r+1}(w-z_{\sigma\inv\tau\inv(s)}+\nu\hbar)\prod_{s=r+2}^m(w-z_{\sigma\inv\tau\inv(s)}-\nu\hbar)\\
  =&\sum_{\sigma\in S_{m,m-1}}\sum_{\tau\in S_{m-1}}\sum_{r=0}^{m-1}\binom{m-1}{r}(-1)^{r+|\tau\sigma|}
  \prod_{1\le s<t\le m-1}(z_{\sigma\inv\tau\inv (s)}-z_{\sigma\inv\tau\inv(t)}+2\hbar)\\
  &\quad\times\prod_{s=1}^{r}(w-z_{\sigma\inv\tau\inv(s)}+\nu\hbar)\prod_{s=r+1}^{m-1}(w-z_{\sigma\inv\tau\inv(s)}-\nu\hbar)\\
  &\quad\times (w-z_{\sigma\inv(m)}-\nu\hbar)
  \prod_{s=1}^{m-1}(z_{\sigma\inv(s)}-z_{\sigma\inv(m)}+2\hbar)\\
  &-\sum_{\sigma\in S_{m,m-1}}\sum_{\tau\in S_m^{m-1}}\sum_{r=0}^{m-1}\binom{m-1}{r}(-1)^{r+|\tau\sigma|}
  \prod_{2\le s<t\le m}(z_{\sigma\inv\tau\inv (s)}-z_{\sigma\inv\tau\inv(t)}+2\hbar)\\
  &\quad\times\prod_{s=2}^{r+1}(w-z_{\sigma\inv\tau\inv(s)}+\nu\hbar)\prod_{s=r+2}^m(w-z_{\sigma\inv\tau\inv(s)}-\nu\hbar)\\
  &\quad\times (w-z_{\sigma\inv(1)}+\nu\hbar)
  \prod_{s=2}^m(z_{\sigma\inv(1)}-z_{\sigma\inv(s)}+2\hbar)\\
  =&\sum_{\sigma\in S_{m,m-1}}(-1)^{|\sigma|}P_{m-1}(z_{\sigma\inv(1)},\dots,z_{\sigma\inv(m-1)},w)(w-z_{\sigma\inv(m)}-\nu\hbar)\\
  &\quad\times\prod_{s=1}^{m-1}(z_{\sigma\inv(s)}-z_{\sigma\inv(m)}+2\hbar)\\
  &-\sum_{\sigma\in S_{m,m-1}}(-1)^{|\sigma|}P_{m-1}(z_{\sigma\inv(2)},\dots,z_{\sigma\inv(m)},w)
  (w-z_{\sigma\inv(1)}+\nu\hbar)\\
  &\quad\times
  \prod_{s=2}^m(z_{\sigma\inv(1)}-z_{\sigma\inv(s)}+2\hbar)\\
  =&0.
\end{align*}
Therefore, we complete the proof.
\end{proof}

\begin{lem}\label{eighth}
Let $a(x), b(x)\in \E_\hbar(W)$ and let $\mu\in\C^\times$, $\nu\in\N$, $m=\nu+1$, $p\in [0,m-1]$ such that \eqref{eq:a-a-relation} and \eqref{eq:a-b-relation} hold.
Then
\begin{align*}
  \sum_{\sigma\in S_m}Dr_p(z_{\sigma(1)},\dots,z_{\sigma(m)},w)=0.
\end{align*}
\end{lem}

\begin{proof}
It is no harm to assume $\mu=1$.
Noticing that $S_m=S_{m,p}S_p\times S_m^p$ and
\begin{align*}
  P_m(S_m,z_1,\dots,z_m,w)=0\quad\te{(see Lemma \ref{seventh})},
\end{align*}
we have
\begin{align*}
  &P_m(S_{m,p}S_p,z_1,\dots,z_m,w)=\sum_{1\ne \sigma_1\in S_m^p}(-1)^{1+|\sigma_1|}P_m(S_{m,p}S_p,z_{\sigma_1\inv(1)},\dots,z_{\sigma_1\inv(m)},w).
\end{align*}
By using Lemma \ref{sixth}, we get that
\begin{align*}
  &\sum_{\sigma\in S_m}Dr_p(z_{\sigma(1)},\dots,z_{\sigma(m)},w)\\
  =&\sum_{\sigma\in S_m}\sum_{1\ne\sigma_1\in S_m^p}a(w)_0^p b(w)a(z_{\sigma(p+1)})\cdots a(z_{\sigma(m)})
  \prod_{s=1}^pz_{\sigma(s)}\inv\delta\left(\frac{w-(\nu-2p+2s)\hbar}{z_{\sigma(s)}}\right)\\
  &\times \prod_{s=1}^p\frac{1}{w-z_{\sigma(s)}+\nu\hbar} \prod_{s=p+1}^m\frac{1}{w-z_{\sigma(s)}-\nu\hbar}
  \prod_{1\le s<t\le m}\frac{1}{z_{\sigma(s)}-z_{\sigma(t)}+2\hbar}\\
  &\times (-1)^{1+|\sigma_1|} P_m(S_{m,p}S_p,z_{\sigma\sigma_1\inv(1)},\dots,z_{\sigma\sigma_1\inv(m)},w)\\
  =&\sum_{\substack{\sigma\in S_m\\ 1\ne\sigma_1\in S_m^p}}
  a(w)_0^p b(w)a(z_{\sigma\sigma_1(p+1)})\cdots a(z_{\sigma\sigma_1(m)})
  \prod_{s=1}^pz_{\sigma(s)}\inv\delta\left(\frac{w-(\nu-2p+2s)\hbar}{z_{\sigma(s)}}\right)\\
  &\times \prod_{p=1}^s\frac{1}{w-z_{\sigma(s)}+\nu\hbar} \prod_{s=p+1}^m\frac{1}{w-z_{\sigma(s)}-\nu\hbar}
  \prod_{1\le s<t\le m}\frac{1}{z_{\sigma\sigma_1(s)}-z_{\sigma\sigma_1(t)}+2\hbar}\\
  &\times P_m(S_{m,p}S_p,z_{\sigma(1)},\dots,z_{\sigma(m)},w).
\end{align*}
There is no harm to assume that $p<m-1$.
For $k=p+1,p+2,\dots, m-1$, define an equivalent relation in $S_m$ by $\sigma\sim_k\sigma'$ if and only if either $\sigma=\sigma'$ or
$\sigma(k)=\sigma'(k+1)$, $\sigma(k+1)=\sigma'(k)$ and $\sigma(k'')=\sigma'(k'')$ for all $k''\ne k,k+1$.
Notice that for each $k$, there exact two element in each equivalent class.
So for each $\sigma\in S_m$, we denote by $\bar\sigma$ the element in $S_m^p$ such that $\sigma\sim_k \bar\sigma$ and $\sigma\ne \bar\sigma$.
Notice that for two distinct $\sigma,\sigma'$, such that $\sigma\sim_k \sigma'$, one has that
\begin{align*}
  \te{either}\quad\sigma(k)>\sigma(k+1)\quad\te{or}\quad\sigma'(k)>\sigma'(k+1).
\end{align*}
Let $S_m^p(k)$ be a complete set of equivalence class representatives,
such that for any $\sigma\in S_m^p(k)$, one has that $\sigma(k)>\sigma(k+1)$.
Then
\begin{align}
  &\sum_{\sigma\in S_m}Dr_p(z_{\sigma(1)},\dots,z_{\sigma(m)},w)\nonumber\\
  =&\sum_{\substack{\sigma\in S_m\\ 1\ne\sigma_1\in S_m^p(k)}}
  \prod_{s=1}^p\frac{1}{w-z_{\sigma(s)}+\nu\hbar} \prod_{s=p+1}^m\frac{1}{w-z_{\sigma(s)}-\nu\hbar}\label{eq:symm-Dr=0-temp}\\
  &\times a(w)_0^p b(w)a(z_{\sigma\sigma_1(p+1)})\cdots a(z_{\sigma\sigma_1(k-1)})
  \prod_{s=1}^pz_{\sigma(s)}\inv\delta\left(\frac{w-(\nu-2p+2s)\hbar}{z_{\sigma(s)}}\right)\nonumber\\
  &\times \bigg(
    a(z_{\sigma\sigma_1(k)})a(z_{\sigma\sigma_1(k+1)})
    \prod_{1\le s<t\le m}\frac{1}{z_{\sigma\sigma_1(s)}-z_{\sigma\sigma_1(t)}+2\hbar}\nonumber\\
  &\quad +a(z_{\sigma\sigma_1(k+1)})a(z_{\sigma\sigma_1(k)})
  \prod_{1\le s<t\le m}\frac{1}{z_{\sigma\bar\sigma_1(s)}-z_{\sigma\bar\sigma_1(t)}+2\hbar}\bigg)\nonumber\\
  &\times  P_m(S_{m,p}S_p,z_{\sigma(1)},\dots,z_{\sigma(m)},w)
  a(z_{\sigma\sigma_1(k+2)})\cdots a(z_{\sigma\sigma_1(m)}).\nonumber
\end{align}
Note that for each $\sigma_1\in S_m^p(k)$, one has that
\begin{align*}
  &(z_{\sigma\sigma_1(k)}-z_{\sigma\sigma_1(k+1)}-2\hbar)\,\,\te{divides}\,\,
  \prod_{1\le s<t\le m}(z_{\sigma\tau\inv(s)}-z_{\sigma\tau\inv(t)}+2\hbar)\,\,\te{for any }\tau\in S_{m,p}S_p.
\end{align*}
Then
\begin{align*}
  &(z_{\sigma\sigma_1(k)}-z_{\sigma\sigma_1(k+1)}-2\hbar)\quad\te{divides}\quad P_m(S_{m,p}S_p,z_{\sigma(1)},\dots,z_{\sigma(m)},w).
\end{align*}
From \eqref{eq:a-a-relation}, we have that
\begin{align*}
  &a(z_{\sigma\sigma_1(k)})a(z_{\sigma\sigma_1(k+1)})
    \prod_{1\le s<t\le m}\frac{1}{z_{\sigma\sigma_1(s)}-z_{\sigma\sigma_1(t)}+2\hbar}\\
  &\quad\times  P_m(S_{m,p}S_p,z_{\sigma(1)},\dots,z_{\sigma(m)},w)\\
  & +a(z_{\sigma\sigma_1(k+1)})a(z_{\sigma\sigma_1(k)})
  \prod_{1\le s<t\le m}\frac{1}{z_{\sigma\bar\sigma_1(s)}-z_{\sigma\bar\sigma_1(t)}+2\hbar}\\
  &\quad\times  P_m(S_{m,p}S_p,z_{\sigma(1)},\dots,z_{\sigma(m)},w)\\
  =&-a(z_{\sigma\sigma_1(k+1)})a(z_{\sigma\sigma_1(k)})
  \frac{z_{\sigma\sigma_1(k)}-z_{\sigma\sigma_1(k+1)}+2\hbar}{z_{\sigma\sigma_1(k+1)}-z_{\sigma\sigma_1(k)}+2\hbar}
    \prod_{1\le s<t\le m}\frac{1}{z_{\sigma\sigma_1(s)}-z_{\sigma\sigma_1(t)}+2\hbar}\\
  &\quad\times  P_m(S_{m,p}S_p,z_{\sigma(1)},\dots,z_{\sigma(m)},w)\\
  & +a(z_{\sigma\sigma_1(k+1)})a(z_{\sigma\sigma_1(k)})
  \prod_{1\le s<t\le m}\frac{1}{z_{\sigma\bar\sigma_1(s)}-z_{\sigma\bar\sigma_1(t)}+2\hbar}\\
  &\quad\times  P_m(S_{m,p}S_p,z_{\sigma(1)},\dots,z_{\sigma(m)},w)\\
  =&-a(z_{\sigma\sigma_1(k+1)})a(z_{\sigma\sigma_1(k)})
  \prod_{1\le s<t\le m}\frac{1}{z_{\sigma\bar\sigma_1(s)}-z_{\sigma\bar\sigma_1(t)}+2\hbar}\\
  &\quad\times  P_m(S_{m,p}S_p,z_{\sigma(1)},\dots,z_{\sigma(m)},w)\\
  & +a(z_{\sigma\sigma_1(k+1)})a(z_{\sigma\sigma_1(k)})
  \prod_{1\le s<t\le m}\frac{1}{z_{\sigma\bar\sigma_1(s)}-z_{\sigma\bar\sigma_1(t)}+2\hbar}\\
  &\quad\times  P_m(S_{m,p}S_p,z_{\sigma(1)},\dots,z_{\sigma(m)},w)\\
  =&0.
\end{align*}
Combining this with \eqref{eq:symm-Dr=0-temp}, we complete the proof.
\end{proof}

\begin{thm}\label{prop:Serre}
Let $W$ be a topologically free $\C[[\hbar]]$-module and  let $a(z),b(z)\in \E_{\hbar}(W)$, $\mu\in\C^\times$, $\nu\in \N$
such that \eqref{eq:a-a-relation} and \eqref{eq:a-b-relation} hold.
Let $m=\nu+1$.
Then the Serre-like relation for $(a(z),b(w),m)$ holds if and only if
\begin{align}\label{eq:Serre-alt}
  a(w)_0^m b(w)=0,
\end{align}
if and only if
\begin{align}\label{eq:Serre-alt2}
  &\nob{a(w+\mu\nu\hbar)a(w+(\nu-2)\mu\hbar)\cdots a(w-\mu\nu\hbar)b(w)}=0.
\end{align}
\end{thm}

\begin{proof}
It is no harm to assume $\mu=1$.
Applying Lemma \ref{eighth} to Lemma \ref{forth}, we get
\begin{align}
  &\sum_{\sigma\in S_m}\sum_{r=0}^m\binom{m}{r}(-1)^r a(z_{\sigma(1)})a(z_{\sigma(2)})\cdots a(z_{\sigma(r)})b(w) a(z_{\sigma(r+1)})
\cdots a(z_{\sigma(m)})\nonumber\\
&\quad=\sum_{\sigma\in S_m}a(w)_0^mb(w)\prod_{s=1}^m z_{\sigma(s)}\inv\delta\left(\frac{w-(\nu-2m+2s)\hbar}{z_{\sigma(s)}}\right).
\label{eq:Serre-temp}
\end{align}
Then \eqref{eq:Serre-alt} implies the Serre-like relation for $(a(z),b(z),m)$.
On the other hand, assume that the Serre-like relation holds for $(a(z),b(z),m)$.
By multiplying $\prod_{s=1}^m\prod_{t\ne s}(z_s-w+(\nu-2m+2t)\hbar)$
and taking $\prod_{s=1}^m\Res_{z_s}$
on both hand sizes of \eqref{eq:Serre-temp},
we get that
\begin{align*}
  &\bigg(\prod_{s=1}^m\prod_{t\ne s}2(t-s)\hbar\bigg) a(w)_0^mb(w)=0.
\end{align*}
Therefore, $a(w)_0^mb(w)=0$.
The equivalence between \eqref{eq:Serre-alt} and \eqref{eq:Serre-alt2} follows immediate from \eqref{eq:no-a-a-b-0}.
\end{proof}

Next, we consider the situation with $\hbar$-adic nonlocal vertex algebras and modules.

\begin{thm}\label{Serre-relation-equivalence-va}
Let $V$ be an $\hbar$-adic nonlocal vertex algebra with $a,b\in V$ and $\mu\in\C^\times$, $\nu\in\N$ such that
\begin{align}
&  (z-w-2\mu\hbar)Y(a,z)Y(a,w)=(z-w+2\mu\hbar)Y(a,w)Y(a,z),\label{thm-a-a-relation-prop}\\
&  (z-w+\mu\nu\hbar)Y(a,z)Y(b,w)=(z-w-\mu\nu\hbar)Y(b,w)Y(a,z).\label{thm-a-b-relation-prop}
\end{align}
Then the Serre-like relation  holds for $(Y(a,z), Y(b,z),\nu+1)$ if and only if
\begin{align}\label{singular-Y-a-b}
a_0^{\nu+1}b=0.
\end{align}
\end{thm}

\begin{proof} From Theorem \ref{prop:Serre}, we have that the Serre-like relation  holds for $$(Y(a,z), Y(b,z),\nu+1)$$ if and only if
\begin{align*}
  0=Y(a,z)_0^{\nu+1}Y(b,z)=Y(a_0^{\nu+1}b,z).
\end{align*}
Since $Y(\cdot, z)$ is injective, we complete the proof.
\end{proof}

\begin{prop}\label{Serre-relation-abstract}
Let $V$ be an $\hbar$-adic nonlocal vertex algebra, let $a,b\in V$ and let $\mu\in\C^\times$, $\nu\in \N$. If
\begin{align}
&  (z-w-2\mu\hbar)Y(a,z)Y(a,w)=(z-w+2\mu\hbar)Y(a,w)Y(a,z),\label{thm-a-a-relation}\\
&  (z-w+\mu\nu\hbar)Y(a,z)Y(b,w)=(z-w-\mu\nu\hbar)Y(b,w)Y(a,z),\label{thm-a-b-relation}
\end{align}
and if the Serre-like relation  holds for $(Y(a,z), Y(b,z),\nu+1)$,
then for any $V$-module $(W,Y_W)$,
\begin{align}
&  (z-w-2\mu\hbar)Y_W(a,z)Y_W(a,w)=(z-w+2\mu\hbar)Y_W(a,w)Y_W(a,z),\label{thm-a-a-relation-module}\\
&  (z-w+\mu\nu\hbar)Y_W(a,z)Y_W(b,w)=(z-w-\mu\nu\hbar)Y_W(b,w)Y_W(a,z),\label{thm-a-b-relation-module}
\end{align}
and the Serre-like relation  holds for $(Y_W(a,z), Y_W(b,z),\nu+1)$.
The converse is also true for any  faithful $V$-module $(W,Y_W)$.
\end{prop}

\begin{proof} In view of Theorem \ref{prop:Serre},
under the assumption (\ref{thm-a-a-relation}) and (\ref{thm-a-b-relation}),
the Serre-like relation  for $(Y(a,z), Y(b,z),\nu+1)$ holds if and only if
\begin{align}
 a_0^{\nu+1}b=0.
\end{align}
For the same reason, under the assumption that (\ref{thm-a-a-relation-module}) and (\ref{thm-a-b-relation-module}) hold,
the Serre-like relation for $(Y_W(a,z), Y_W(b,z),\nu+1)$  holds if and only if
\begin{align}
 0=Y_W(a,z)_0^{\nu+1}Y_W(b,z)=Y_W(a_0^{\nu+1}b,z),
\end{align}
where the last equation follows from Lemma \ref{module-rep}.
Now it follows.
\end{proof}

As an immediate consequence of Proposition \ref{Serre-relation-abstract} we have:

\begin{coro}\label{Sere-relation}
Let $W$ be a topologically free $\C[[\hbar]]$-module and let $a(x),b(x)\in U$,
where $U$ is an $\hbar$-adically compatible subset of $\E_{\hbar}(W)$, such that
\begin{eqnarray}
&&(x_1-x_2-2\mu\hbar)a(x_1)a(x_2)=(x_1-x_2+2\mu\hbar)a(x_2)a(x_1),\\
&&(x_1-x_2+\mu \nu\hbar)a(x_1)b(x_2)=(x_1-x_2-\mu\nu\hbar)b(x_2)a(x_1)
\end{eqnarray}
for some $\mu\in\C^\times$, $\nu\in \N$.
Assume that the Serre-like relation for $(a(z),b(z),\nu+1)$ holds.
Then
\begin{eqnarray*}
&&(z_1-z_2-2\mu\hbar)Y_{\E}(a(x),z_1)Y_{\E}(a(x),z_2)=(z_1-z_2+2\mu\hbar)Y_{\E}(a(x),z_2)Y_{\E}(a(x),z_1),\\
&&(z_1-z_2+\mu\nu\hbar)Y_{\E}(a(x),z_1)Y_{\E}(b(x),z_2)=(z_1-z_2-\mu\nu\hbar)Y_{\E}(b(x),z_2)Y_{\E}(a(x),z_1)
\end{eqnarray*}
on $\<U\>$, and the Serre-like relation holds for $(Y_{\E}(a(x),z), Y_{\E}(b(x),z),\nu+1)$.
\end{coro}

Recall the following standard formal series convention (see \cite{FLM}):
\begin{align}
&\log (1+f(x))=\sum_{n\ge 1}(-1)^{n-1}\frac{1}{n}f(x)^n\in xK[[x]],\\
&e^{f(x)}=\sum_{n\ge 0}\frac{1}{n!}f(x)^n\in K[[x]]
\end{align}
for $f(x)\in xK[[x]]$, where $K$ is any associative algebra with identity over $\C$.

 For $f(x,z)\in \C(x,z)$, as a convention we write
\begin{align}
f(x,\hbar):=\(\iota_{x,z}f(x,z)\)|_{z=\hbar}\in \C((x))((\hbar)).
\end{align}
Let $p(x,z),q(x,z)\in \C[x,z]$ with $p(x,0)=q(x,0)\ne 0$.
Then  $\frac{p(x,\hbar)}{q(x,\hbar)}=1+hf(x,\hbar)$ for some $f(x,\hbar)\in \C((x))[[\hbar]]$,
and hence
\begin{align}
\log \left(\frac{p(x,\hbar)}{q(x,\hbar)}\right)\in \hbar \C((x))[[\hbar]].
\end{align}
In particular, for $\mu,\nu\in \hbar \C[[\hbar]],\ a\in K$, we have
\begin{align}
\log \(\frac{x+\mu}{x+\nu}\)&=\log (1+\mu x^{-1})-\log (1+\nu x^{-1}) 
=(e^{\mu\partial_x}-e^{\nu\partial_x})\log x,
\end{align}
\begin{align}
\(\frac{x+\mu}{x+\nu}\)^{a}=\exp \!\(\!a \log \(\frac{x+\mu}{x+\nu}\)\! \)\!=\exp \(a (e^{\mu \partial_x}-e^{\nu \partial_x}) \log x\)\!.
\end{align}

Next, we present a vertex-operator iterate formula. First, we establish the following technical result (cf. \cite{JKLT-22}):

\begin{lem}\label{lem:exp-cal}
Let $W$ be a topologically free $\C[[\hbar]]$-module and let
\begin{align*}
  \al(x)\in \Hom (W,W\wh\ot \C[x,x\inv][[\hbar]]),\ \
  \beta(x)\in(\End W)[[x]],\ \  \gamma(x)\in \C((x))[[\hbar]],
\end{align*}
satisfying the condition
\begin{align}
  [\al(x_1),\al(x_2)]=0=[\beta(x_1),\beta(x_2)],\quad
  [\al(x_1),\beta(x_2)]=\gamma(x_1-x_2).\label{eq:exp-cal-cond2}
\end{align}
Then $U:=\{ \al(x), \beta(x), 1_W\}$ is an $\SY$-local subset of $\E_{\hbar}(W)$.
Furthermore, assume $\al(x),\be(x)\in \hbar (\End W)[[x,x^{-1}]]$ and
set $E_{\gamma}=\Res_zz\inv \gamma(z)\in \hbar^2\C[[\hbar]]$.
Then
\begin{align}
 \exp\(\(\al(x)+\beta(x)\)_{-1}\)1_W=\exp\(\frac{1}{2}E_\gamma\)\exp \al(x)\exp \be(x).
\end{align}
\end{lem}

\begin{proof} From \eqref{eq:exp-cal-cond2} it can be readily seen that $U$ is $\SY$-local.
Then we have an $\hbar$-adic weak quantum vertex algebra $\<U\>$.
With \eqref{eq:exp-cal-cond2}, by Lemma \ref{lem-a-b-C-D}  we get
\begin{align*}
  [Y_\E(\al(x),x_1),Y_\E(\beta(x),x_2)]=\gamma(x_1-x_2),
\end{align*}
which in particular implies
\begin{eqnarray*}
  [\al(x)_{-1},\beta(x)_{-1}]=E_{\gamma}.
\end{eqnarray*}
Then by the Baker-Campbell-Hausdorff formula  we have
\begin{eqnarray}\label{BCH-exp}
  \exp\((\al(x)+\beta(x))_{-1}\)
  =\exp\(\frac{1}{2} E_{\gamma}\)\exp\(\beta(x)_{-1}\)\exp\(\al(x)_{-1}\)\!.
  \end{eqnarray}
On the other hand, notice that
\begin{eqnarray*}
  \al(x_1)\al(x_2)^n,\,\beta(x_1)u(x_2)\in \E_\hbar^{(2)}(W)\quad\te{for }n\in\N,\,\,u(x)\in \<U\>.
\end{eqnarray*}
Then from the definition of $Y_\E$ we get
\begin{eqnarray*}
  &&\al(x)_{-1}\al(x)^n=\al(x)^{n+1}\quad \te{for }n\in\N,\label{eq:exp-cal-temp2}\\
  &&\beta(x)_{-1}u(x)=\beta(x)u(x)\quad\te{for }u(x)\in \<U\>.\label{eq:exp-cal-temp3}
\end{eqnarray*}
Using (\ref{BCH-exp}) and these relations we obtain
\begin{align*}
  &\exp\((\al(x)+\beta(x))_{-1}\)1_W\\
  =&\exp\(\half E_\gamma\)\exp\(\beta(x)_{-1}\)\exp\(\al(x)_{-1}\)1_W\\
  =&\exp\(\half E_\gamma\)\exp\(\beta(x)_{-1}\)\exp\(\al(x)\)\\
  =&\exp\(\half E_\gamma\)\exp\(\beta(x)\)\exp\(\al(x)\)\!,
\end{align*}
as desired.
\end{proof}

Introduce a formal power series
\begin{align}\label{L-def}
L(x)=\sum_{n\ge 1}\frac{1}{n!}x^{n-1}=\frac{e^{x}-1}{x}\in \C[[x]].
\end{align}
For $v\in V$ (an $\hbar$-adic nonlocal vertex algebra),
denote by $Y^{+}(v,x)$ and $Y^{-}(v,x)$ the regular and singular parts of $Y(v,x)$, respectively.
The following is the vertex-operator iterate formula that we need:

\begin{prop}\label{Y-W-E(a,z)}
Let $V$ be an $\hbar$-adic nonlocal vertex algebra, $(W,Y_W)$ a $V$-module.
Assume $a\in V$ such that
\begin{eqnarray}\label{a-conditions}
[Y^{\pm}(a,x_1),Y^{\pm}(a,x_2)]=0, \quad
[Y^{-}(a,x_1),Y^{+}(a,x_2)]=\gamma(x_1-x_2),
\end{eqnarray}
where $\gamma(x)\in x^{-1} \C[x^{-1}][[\hbar]]$.  Let $z\in \hbar \C[[\hbar]]$ and set
$$E^{\pm }(a,z)=\exp \(\sum_{n\in \pm \Z_{+}}\frac{1}{n}a_{n}z^{-n} \)$$
on $V$ and $W$, where $Y(a,x)=\sum_{n\in \Z}a_nx^{-n-1}$. Then
\begin{align}
&Y_W(E^{-}(-a,z){\bf 1},x)=\exp \(zL(z\partial_x)Y_W^{+}(a,x)\)\exp \(zL(z\partial_x)Y_W^{-}(a,x)\)\\
&\quad = \(1+\frac{z}{x}\)^{a_0}E^{-}(-a,x+z)E^{-}(a,x)E^{+}(-a,x+z)E^{+}(a,x). \nonumber
\end{align}
\end{prop}

\begin{proof} Set $v=\sum_{n\ge 1}\frac{1}{n}z^na_{-n}{\bf 1}\in \hbar V$. We have
\begin{align}
v=\sum_{n\ge 1}\frac{1}{n!}z^{n}\D^{n-1}a\quad \text{ and }\quad v_{-1}=\sum_{n\ge 1}\frac{1}{n}z^na_{-n}.
\end{align}
Then $E^{-}(-a,z)=\exp v_{-1}$. On the other hand, as
\begin{eqnarray*}
Y_W(v,x)=\sum_{n\ge 1}z^{n}\frac{1}{n!}Y_W(\D^{n-1}a,x)=\sum_{n\ge 1}z^{n}\frac{1}{n!}\(\frac{d}{dx}\)^{n-1}Y_W(a,x),
\end{eqnarray*}
we get
\begin{eqnarray}\label{YW-v-pm}
Y_W^{\pm}(v,x)=\sum_{n\ge 1}z^{n}\frac{1}{n!}\(\frac{d}{dx}\)^{n-1}Y_W^{\pm}(a,x)
=zL(z\partial_x)Y_W^{\pm}(a,x).
\end{eqnarray}

Note that the commutation relations in (\ref{a-conditions}) amount to
$$[Y(a,x_1),Y(a,x_2)]=\gamma(x_1-x_2)-\gamma(x_2-x_1).$$
In view of Lemma \ref{Y-V-W-a-b-relation}, we have
$$[Y_W(a,x_1),Y_W(a,x_2)]=\gamma(x_1-x_2)-\gamma(x_2-x_1),$$
which amounts to
\begin{eqnarray}
[Y_W^{\pm}(a,x_1),Y^{\pm}_W(a,x_2)]=0,\quad
[Y_W^{-}(a,x_1),Y_W^{+}(a,x_2)]=\gamma(x_1-x_2)
\end{eqnarray}
as $\gamma(x)\in x^{-1} \C[x^{-1}][[\hbar]]$.
Then using (\ref{YW-v-pm}) we get
\begin{align}
[Y_W^{-}(v,x_1),Y_W^{+}(v,x_2)]
&\ = z^2 L(z\partial_{x_1})L(z\partial_{x_2})\gamma(x_1-x_2)\\
&\ =z^2 L(z\partial_{x_1})L(-z\partial_{x_1})\gamma(x_1-x_2)\nonumber\\
&\ =\tilde{\gamma}(x_1-x_2),\nonumber
\end{align}
where $\tilde{\gamma}(x)=z^2L(z\partial_x)L(-z\partial_x)\gamma(x)\in \hbar^2x^{-1} \C[x^{-1}][[\hbar]].$
Then using Lemma \ref{lem:exp-cal} with $\alpha(x)=Y_W^{-}(v,x)$, $\beta(x)=Y_W^{+}(v,x)$, and using the fact that
 $Y_W(\cdot,x)$ is a nonlocal vertex algebra homomorphism, we obtain
\begin{eqnarray*}
&&Y_W(E^{-}(-a,z){\bf 1},x)=Y_{W}( (\exp v_{-1}){\bf 1},x)
=\exp \(Y_W(v,x)_{-1}\)1_W\nonumber\\
&&=\exp (Y_W^{+}(v,x))\exp (Y_W^{-}(v,x))=\exp \(zL(z\partial_x)Y_W^{+}(a,x)\)\exp \(zL(z\partial_x)Y_W^{-}(a,x)\)\!,
\end{eqnarray*}
noticing that $E_{\tilde{\gamma}}:=z^2\Res_xx^{-1}L(z\partial_x)L(-z\partial_x)\gamma(x)=0$.
Furthermore, we have
\begin{align}
zL(z\partial_x)Y_W^{+}(a,x)=(e^{z\partial_x}-1)\sum_{n\ge 1}\frac{1}{n}a_{-n}x^n=\sum_{n\ge 1}\frac{1}{n}a_{-n} \((x+z)^{n}-x^{n}\)\!,
\end{align}
\begin{align}
zL(z\partial_x)Y_W^{-}(a,x)
&=(e^{z\partial_x}-1)\sum_{n\ge 1}(-1)\frac{1}{n}a_n x^{-n}+a_0(e^{z\partial_x}-1)\log x\nonumber\\
&=-\sum_{n\ge 1}\frac{1}{n}a_n \((x+z)^{-n}-x^{-n}\)+a_0\log \(1+\frac{z}{x}\)\!.\nonumber
\end{align}
Then it follows immediately.
\end{proof}

\section{Double Yangians and $\hbar$-adic quantum vertex algebras}\label{sec:DY}

In this section, for any symmetrizable generalized Cartan matrix (GCM) $A$, we introduce an algebra
$\wh{\mathcal{DY}}(A)$, generalizing centrally extended double Yangians. As the main results,
we give a new field presentation of $\wh{\mathcal{DY}}(A)$ and a (tautological) construction of
a  universal vacuum $\wh{\mathcal{DY}}(A)$-module $\mathcal{V}_A(\ell)$ of any level $\ell\in \C$, and we prove that
$\mathcal{V}_A(\ell)$ has a canonical $\hbar$-adic weak quantum vertex algebra structure
with any (other) restricted $\wh{\mathcal{DY}}(A)$-modules of level $\ell$ as modules.

\subsection{Centrally extended double Yangian}
Let $A=[a_{i,j}]_{i,j\in I}$ be a symmetrizable GCM, which is fixed throughout this section.
Then there are unique relatively prime positive integers $r_i$ ($i\in I$) such that $DA$ is symmetric with $D=\te{diag}\{r_i\}_{i\in I}$.

Now, we introduce the main object for this paper.

\begin{de}
Define $\wh{\mathcal{DY}}(A)$ to be
the topological associative unital algebra over $\C[[\hbar]]$, generated by
\begin{align}\label{eq:DY-gen-set}
\set{H_{i,n},\, X_{i,n}^\pm}{i\in I,\,n\in\Z}\cup \{\kappa\},
\end{align}
subject to a set of relations written in terms of generating functions:
\begin{align}
  &H_i^{+}(z)=1+2\hbar \sum_{n\ge 0}H_{i,n}z^{-n-1},\quad
   H_i^{-}(z)= 1+2\hbar \sum_{n\ge 1}H_{i,-n}z^{n-1},\\
 &\hspace{3.5cm} X_i^\pm(z)=\sum_{n\in\Z}X_{i,n}^\pm z^{-n-1}.
\end{align}
In addition to that $\kappa$ and $H_{i,0}$ for $i\in I$ are central,
the following are the defining relations with $i,j\in I$:
\begin{align*}
    &\te{(DY1)}\quad [H_i^\pm(z),H_j^\pm(w)]=0,\\
  &\te{(DY2)}\quad H_i^+(z)H_j^-(w)=H_j^-(w)H_i^+(z)
  \frac{(z-w-r_ia_{i,j}\hbar-\kappa\hbar)(z-w+r_ia_{i,j}\hbar+\kappa\hbar)}
    {(z-w+r_ia_{i,j}\hbar-\kappa\hbar)(z-w-r_ia_{i,j}\hbar+\kappa\hbar)},\\
  &\te{(DY3)}\quad H_i^+(z)X_j^\pm(w)=X_j^\pm(w)H_i^+(z)
    \(\frac{z-w+r_ia_{i,j}\hbar\pm\half \kappa\hbar}{z-w-r_ia_{i,j}\hbar\pm\half \kappa\hbar}\)^{\pm 1},\\
  &\te{(DY4)}\quad H_i^-(z)X_j^\pm(w)=X_j^\pm(w)H_i^-(z)
    \(\frac{w-z-r_ia_{j,i}\hbar\pm\half k\hbar}{w-z+r_ia_{j,i}\hbar\pm\half \kappa\hbar}\)^{\pm 1},\\
  &\te{(DY5)}\quad [X_i^+(z),X_j^-(w)]
    =\delta_{i,j}\frac{1}{2r_i\hbar}
    \Bigg(\!
        H_i^+(w+\half \kappa\hbar)z\inv\delta\!\(\frac{w+\kappa\hbar}{z}\)\\
  &\qquad\qquad\qquad\qquad
    -   H_i^-(w-\half \kappa\hbar)z\inv\delta\!\(\frac{w-\kappa\hbar}{z}\)\!
    \Bigg),\\
  &\te{(DY6)}\quad (z-w\mp r_ia_{i,j}\hbar)X_i^\pm(z)X_j^\pm(w)
    =(z-w\pm r_ia_{i,j}\hbar)X_j^\pm(w)X_i^\pm(z),\\
  &\te{(DY7)}\quad
 \sum_{\sigma\in S_{1-a_{i,j}}} \sum_{s=0}^{1-a_{i,j}}\binom{1-a_{i,j}}{s}(-1)^s
    X_i^\pm(z_{\sigma(1)})X_i^\pm(z_{\sigma(2)})\cdots X_i^\pm(z_{\sigma(s)}) X_j^\pm(w)\\
    &\qquad\qquad\times X_i^\pm(z_{\sigma(s+1)})\cdots X_i^\pm(z_{\sigma(1-a_{i,j})})=0
    \quad \te{if }a_{i,j}\le 0.
\end{align*}
\end{de}

Note that relation (DY1) is equivalent to that
\begin{align}
 [H_{i,m},H_{j,n}]=0\quad \te{ for }m,n\in \Z\te{ with either }m,n\ge 0\te{ or }m,n< 0.
\end{align}

\begin{rem}
{\em Note that if $A$ is of simply-laced finite type with $\g=\g(A)$,
this definition except (DY7) for $a_{i,j}=0$ is the same as in \cite[page 37]{DK} (after a parameter-change $\hbar\rightarrow 2i\hbar$).
On the other hand, if $\g=\mathfrak{sl}_{n+1}$,  $\wh{\mathcal{DY}}(A)$ is isomorphic to
the algebra $\mathcal {DY}_\hbar(\g)$ defined in \cite[Corollary 3.4]{Io}.}
\end{rem}

\begin{de}\label{restricted-module}
A $\wh{\mathcal{DY}}(A)$-module $W$ is called a {\em restricted module} if
$W$ is a topologically free $\C[[\hbar]]$-module such that for every $i\in I,\ w\in W$,
\begin{align}
\lim_{n\rightarrow \infty}H_{i,n}w=0=\lim_{n\rightarrow \infty}X_{i,n}^{\pm}w.
\end{align}
\end{de}

Note that $H_i^\pm(z),X_i^\pm(z)\in \E_\hbar(W)$  $(i\in I)$ for any restricted $\wh{\mathcal{DY}}(A)$-module $W$.

\begin{de}\label{def-A}
Denote by $\widetilde{\mathcal{A}}$ the associative algebra over $\C$, generated by elements
$$\tilde{\kappa},\ \tilde{H}_{i,n},\ \tilde{X}_{i,n}^{\pm}\ \  (\text{where }i\in I,\ n\in \Z),$$
subject to the relation (condition) that $\tilde{\kappa}$ and $\tilde{H}_{i,0}$ for $i\in I$ are central.
\end{de}

Consider the $\C[[\hbar]]$-algebra $\widetilde{\mathcal{A}}[[\hbar]]$ and
formulate generating functions $\tilde{H}_i^{\pm}(z),\ \tilde{X}_i^{\pm}(z)$ for $i\in I$ in the same way
as for $H_i^{\pm}(z),\ X_i^{\pm}(z)$.
Note that $\tilde{H}_i^{+}(z)$ and $\tilde{H}_i^{-}(z)$ are invertible elements of
$\widetilde{\mathcal{A}}[[z^{-1}]][[\hbar]]\ (=\widetilde{\mathcal{A}}[[\hbar]][[z^{-1}]])$
and $\widetilde{\mathcal{A}}[[z]][[\hbar]]\ (=\widetilde{\mathcal{A}}[[\hbar]][[z]])$, respectively.

A {\em restricted} $\widetilde{\mathcal{A}}[[\hbar]]$-module is defined to be an  $\widetilde{\mathcal{A}}[[\hbar]]$-module $W$
which is a topologically free  $\C[[\hbar]]$-module such that
\begin{align}
\lim_{n\rightarrow \infty}\tilde{H}_{i,n}w=0=\lim_{n\rightarrow \infty}\tilde{X}_{i,n}^{\pm}w
\quad \text{ for }i\in I,\ w\in W.
\end{align}

Set $q=e^{\hbar}\in \C[[\hbar]]$. Introduce the following power series in $x$ and $\hbar$:
\begin{align}
&G(x)=\frac{q^{x}-q^{-x}}{x}:=2\hbar\(1+\sum_{n\ge 1}\frac{\hbar^{2n}}{(2n+1)!}x^{2n}\)\in \hbar\C[[\hbar,x]],\\
&F(x)= \frac{1}{2\hbar}\(1+\sum_{n\ge 1}\frac{\hbar^{2n}}{(2n+1)!}x^{2n}\)^{-1}\in \hbar^{-1}\C[[x,\hbar]].
\end{align}
Then $G(-x)=G(x)$, $F(-x)=F(x)$,
\begin{align}
G(\partial_z)=\frac{q^{\partial_z}-q^{-\partial_z}}{ \partial_z},\ \ \
F(\partial_z)=\frac{\partial_{z}}{q^{\partial_z}-q^{-\partial_z}},\ \ \text{and }\  F(\partial_z)G(\partial_z)=1.
\end{align}
 For $m\in \Z$, in terms of $G(x)$ we have
\begin{align}\label{log-two-terms}
\log \left(\frac{x+m\hbar}{x-m\hbar}\right)=\(\frac{q^{m\partial_x}-q^{-m\partial_x}}{\partial_x}\)x^{-1}
= G(\partial_x) [m]_{q^{\partial_x}}x^{-1}.
\end{align}
Furthermore, we have
\begin{align}\label{log-four-terms}
&\log  \frac{(z-w-m\hbar-\kappa\hbar)(z-w+m\hbar+\kappa\hbar)}
    {(z-w+m\hbar-\kappa\hbar)(z-w-m\hbar+\kappa\hbar)}\\
=\ & -G(\partial_z) [m]_{q^{\partial_z}} (z-w-\kappa\hbar)^{-1}+G(\partial_z)[m]_{q^{\partial_z}} (z-w+\kappa\hbar)^{-1}
\nonumber\\
=\ & G(\partial_z)[m]_{q^{\partial_z}}\( (z-w+\kappa\hbar)^{-1}-(z-w-\kappa\hbar)^{-1}\)\nonumber\\
=\ &  -G(\partial_z)G(\partial_z)[m]_{q^{\partial_z}}[\kappa]_{q^{\partial_z}}(z-w)^{-2}.\nonumber
\end{align}

\begin{de}\label{new-operators}
Let $W$ be a restricted $\widetilde{\mathcal{A}}[[\hbar]]$-module. For $i\in I$, set
\begin{align}
  h_{i,\Y}^\pm(z) =\pm F(\partial_z) \log \tilde{H}_i^\pm\big(z\pm  \frac{1}{2}\tilde{\kappa}\hbar \big)\in \widetilde{\mathcal{A}}[[\hbar]][[z^{\mp 1}]].
    \label{eq:def-h-x-0}
  \end{align}
 Furthermore, set
  \begin{align}
  &h_{i,\Y}(z)=h_{i,\Y}^+(z)+h_{i,\Y}^-(z),\\
  &x_{i,\Y}^+(z)=\tilde{X}_i^+(z),\label{eq:def-h-x-1}\\
  &x_{i,\Y}^-(z)=\tilde{X}_i^-(z-\tilde{\kappa}\hbar)\tilde{H}_i^+\big(z-\frac{1}{2} \tilde{\kappa}\hbar\big)\inv.\label{eq:def-h-x-2}
\end{align}
\end{de}

Note that $h_{i,\Y}^+(z)$ and $h_{i,\Y}^-(z)$ coincide with the {\em singular and regular} parts of  $h_{i,\Y}(z)$
with respect to variable $z$, respectively. For any  restricted $\widetilde{\mathcal{A}}[[\hbar]]$-module $W$, we have
$$h_{i,\Y}(z),\  x_{i,\Y}^\pm(z)\in \E_\hbar(W)\quad \te{for }i\in I.$$

As the first main result of this section, we have:

\begin{thm}\label{main-DY}
Let $W$ be a restricted $\widetilde{\mathcal{A}}[[\hbar]]$-module.
For $i\in I$, set
\begin{align}
C_i(x)&=\tilde{H}^{-}_i(x-\frac{3}{2}\tilde{\kappa}\hbar)\tilde{H}_i^+(x-\frac{1}{2}\tilde{\kappa}\hbar)^{-1}\\
&=\exp\!\(-G(\partial_x)q^{-\tilde{\kappa}\partial_x}h_{i,\Y}^{-}(x)\)
\exp\!\(-G(\partial_x)q^{-\tilde{\kappa}\partial_x}h_{i,\Y}^{+}(x)\)\!,\nonumber
\end{align}
an element of $\E_{\hbar}(W)$. Then $W$ is a restricted $\wh{\mathcal{DY}}(A)$-module with
$$\kappa=\tilde{\kappa}, \ H_{i,n}=\tilde{H}_{i,n},\ X_{i,n}^{\pm}=\tilde{X}_{i,n}^{\pm} \quad \text{ for }i\in I$$
if and only if the following relations hold for $i,j\in I$:
\begin{align*}
&[h_{i,\Y}(z),h_{j,\Y}(w)]
    =[r_ia_{i,j}]_{q^{\partial_{w}}}
        [\kappa]_{q^{\partial_{w}}}\( (z-w+\hbar \kappa)^{-2}-(w-z+\hbar \kappa)^{-2}\)\!,\\
&[h_{i,\Y}(z),x_{j,\Y}^\pm(w)]
=\pm x_{j,\Y}^\pm(w)[r_ia_{i,j}]_{q^{\partial_{w}}}\((z-w+\hbar \kappa)^{-1}+(w-z+\hbar \kappa)^{-1}\)\!,\\
&(z-w-r_ia_{i,j}\hbar)x_{i,\Y}^\pm(z)x_{j,\Y}^\pm(w)=(z-w+r_ia_{i,j}\hbar)x_{j,\Y}^\pm(w)x_{i,\Y}^\pm(z),\\
&x_{i,\Y}^+(z)x_{j,\Y}^-(w)  - \( \frac{w-z+r_ia_{i,j}\hbar}{w-z-r_ia_{i,j}\hbar}\)x_{j,\Y}^-(w)x_{i,\Y}^+(z)\nonumber\\
&\quad  =\delta_{i,j}\frac{1}{2r_i\hbar}
\(\!z\inv\delta\!\(\frac{w}{z}\)-C_i(w)  z\inv\delta\!\(\frac{w-2\kappa\hbar}{z}\)\)\!,
\end{align*}
and the Serre-like relations for $(x_{i,\Y}^\pm(z), x_{j,\Y}^\pm(z),1-a_{i,j})$ hold if $a_{i,j}\le 0$.
\end{thm}

\begin{proof} Note that the relation for $[h_{i,\Y}(z),h_{j,\Y}(w)]$ amounts to
\begin{align*}
&[h_{i,\Y}^+(z),h_{j,\Y}(w)]
    =[r_ia_{i,j}]_{q^{\partial_{w}}}  [\kappa]_{q^{\partial_{w}}}(z-w+\hbar \kappa)^{-2},\\
&[h_{i,\Y}^-(z),h_{j,\Y}(w)]
    =-[r_ia_{i,j}]_{q^{\partial_{w}}}  [\kappa]_{q^{\partial_{w}}}(w-z+\hbar \kappa)^{-2}.
\end{align*}
These relations are furthermore equivalent to
\begin{align}
&[h_{i,\Y}^{\pm}(z),h_{j,\Y}^{\pm}(w)]=0,\label{h-i-y-pm}\\
&[h_{i,\Y}^+(z),h_{j,\Y}^{-}(w)]
    =[r_ia_{i,j}]_{q^{\partial_{w}}}  [\kappa]_{q^{\partial_{w}}}(z-w+\hbar \kappa)^{-2}\label{6.32}
\end{align}
for all $i,j\in I$ (recall that $r_ia_{i,j}=r_ja_{j,i}$).

On the other hand, for $i,j\in I$, by (\ref{log-four-terms}) we have
\begin{align*}
&\log  \frac{(z-w-r_ia_{i,j}\hbar-\kappa\hbar)(z-w+r_ia_{i,j}\hbar+\kappa\hbar)}
    {(z-w+r_ia_{i,j}\hbar-\kappa\hbar)(z-w-r_ia_{i,j}\hbar+\kappa\hbar)}\nonumber\\
=&\  -G(\partial_z)G(\partial_w)[r_ia_{i,j}]_{q^{\partial_z}}[\kappa]_{q^{\partial_z}}(z-w)^{-2}.
\end{align*}
Then (DY2) amounts to
\begin{align*}
[\log H_i^{+}(z),\log H_j^{-}(w)]
=-G(\partial_z)G(\partial_w) [r_ia_{i,j}]_{q^{\partial_z}}[\kappa]_{q^{\partial_z}}(z-w)^{-2},
\end{align*}
 which is equivalent to
\begin{align*}
&[\log H_i^{+}(z+1/2\kappa\hbar),\log H_j^{-}(w-1/2\kappa\hbar)]\\
=\ &- G(\partial_z)G(\partial_w)[r_ia_{i,j}]_{q^{\partial_z}}[\kappa]_{q^{\partial_z}} (z-w+\kappa\hbar)^{-2}.\nonumber
\end{align*}
This is furthermore equivalent to
\begin{align*}
&-[F(\partial_z)\log H_i^{+}(z+1/2\kappa\hbar),F(\partial_w)\log H_j^{-}(w-1/2\kappa\hbar)]\nonumber\\
=& \  [r_ia_{i,j}]_{q^{\partial_w}}[\kappa]_{q^{\partial_w}}(z-w+\kappa\hbar)^{-2}.
\end{align*}
Thus, (\ref{6.32}) is equivalent to (DY2). It is clear that (\ref{h-i-y-pm}) is equivalent to (DY1).

 Recall
$$x_{i,\Y}^+(x)=X_i^+(x),\quad x_{i,\Y}^-(x)=X_i^-(x-\kappa\hbar)H_i^+(x-\frac{1}{2} \kappa\hbar)\inv.    $$
Rewrite (DY2) as
\begin{align*}
H_i^-(z)H_j^+(w)^{-1}=H_j^+(w)^{-1}H_i^-(z)
\frac{(w-z-r_ja_{j,i}\hbar-\kappa\hbar)(w-z+r_ja_{j,i}\hbar+\kappa\hbar)}
{(w-z+r_ja_{j,i}\hbar-\kappa\hbar)(w-z-r_ja_{j,i}\hbar+\kappa\hbar)},
\end{align*}
which is equivalent to
\begin{align*}
&[\log H_i^-(z),H_j^+(w)^{-1}]\\
=\ & H_j^+(w)^{-1}\cdot G(\partial_w)[r_ja_{j,i}]_{q^{\partial_w}}
\((w-z-\kappa\hbar)^{-1}-(w-z+\kappa\hbar)^{-1}\)\!.
\end{align*}
Then
\begin{align*}
&[h_{i,\Y}^-(z),H_j^+(w-\frac{1}{2}\kappa\hbar)^{-1}]\nonumber\\
=\ &-F(\partial_z)[\log H_i^-(z-\frac{1}{2}\kappa\hbar),H_j^+(w-\frac{1}{2}\kappa\hbar)^{-1}]
\nonumber\\
=\ &H_j^+(w-\frac{1}{2}\kappa\hbar)^{-1}[r_ia_{i,j}]_{q^{\partial_z}}
((w-z-\kappa\hbar)^{-1}-(w-z+\kappa\hbar)^{-1}).
\end{align*}
Thus
\begin{align}\label{hiY-Hj+}
&[h_{i,\Y}(z),H_j^+(w-\frac{1}{2}\kappa\hbar)^{-1}]\\
=\ &H_j^+(w-\frac{1}{2}\kappa\hbar)^{-1}[r_ia_{i,j}]_{q^{\partial_z}}
((w-z-\kappa\hbar)^{-1}-(w-z+\kappa\hbar)^{-1}).\nonumber
\end{align}

Note that relation (DY3) is equivalent to
\begin{align}
[\log H_i^{+}(z),X_j^{\pm}(w)]
=\pm X_j^{\pm}(w)
\log\( \frac{z-w+r_ia_{i,j}\hbar\pm \frac{1}{2}\kappa\hbar}{z-w-r_ia_{i,j}\hbar\pm \frac{1}{2}\kappa\hbar}\)\!.
\end{align}
By (\ref{log-two-terms}) this amounts to
\begin{align*}
[h_{i,\Y}^{+}(z),X_j^{+}(w)]
=\ &X_j^{+}(w)F(\partial_z)\log\( \frac{z-w+r_ia_{i,j}\hbar+\kappa\hbar}{z-w-r_ia_{i,j}+\kappa\hbar}\)\\
=\ &X_j^{+}(w)[r_ia_{i,j}]_{q^{\partial_w}}(z-w+\kappa\hbar)^{-1},
\end{align*}
\begin{align*}
[h_{i,\Y}^{+}(z),X_j^{-}(w)]
=\ &-X_j^{-}(w)F(\partial_z)
\log \(\frac{z-w+r_ia_{i,j}\hbar}{z-w-r_ia_{i,j}\hbar}\)\nonumber\\
=\ &-X_j^{-}(w)[r_ia_{i,j}]_{q^{\partial_w}}(z-w)^{-1}.
\end{align*}

Similarly, relation (DY4) is equivalent to
\begin{align}
[\log H_i^{-}(z),X_j^{\pm}(w)]
=\pm X_j^{\pm}(w)
\log \(\frac{w-z-r_ia_{i,j}\hbar\pm \frac{1}{2}\kappa\hbar}{w-z+r_ia_{i,j}\hbar\pm \frac{1}{2}\kappa\hbar}\)\!.
\end{align}
This is equivalent to
\begin{align*}
[h_{i,\Y}^{-}(z),X_j^{+}(w)]
=\ &- X_j^{+}(w)F(\partial_z)
\log \(\frac{w-z-r_ia_{i,j}\hbar+\kappa\hbar}{w-z+r_ia_{i,j}\hbar+\kappa\hbar}\)\nonumber\\
=\ &X_j^{+}(w) [r_ia_{i,j}]_{q^{\partial_w}} (w-z+\kappa\hbar)^{-1}
\end{align*}
and
\begin{align*}
[h_{i,\Y}^{-}(z),X_j^{-}(w)]
=\ &X_j^{-}(w)F(\partial_z)
\log \(\frac{w-z-r_ia_{i,j}\hbar}{w-z+r_ia_{i,j}\hbar}\)\nonumber\\
=\ &-X_j^{-}(w) [r_ia_{i,j}]_{q^{\partial_w}} (w-z)^{-1}.
\end{align*}
Therefore,  (DY3) together with (DY4) is equivalent to
\begin{align}
[h_{i,\Y}(z),x_{j,\Y}^{+}(w)]&\ =[h_{i,\Y}^{+}(z)+h_{i,\Y}^{-}(z),X_{j}^{+}(w)]\\
&\ =X_{j}^+(w)[r_ia_{i,j}]_{q^{\partial_{w}}}((z-w+\hbar \kappa)^{-1}+(w-z+\hbar \kappa)^{-1}) \nonumber\\
&\ =x_{j,\Y}^+(w)[r_ia_{i,j}]_{q^{\partial_{w}}}((z-w+\hbar \kappa)^{-1}+(w-z+\hbar \kappa)^{-1}), \nonumber
\end{align}
\begin{align}
[h_{i,\Y}(z),X_{j}^{-}(w)]=-X_{j}^-(w)[r_ia_{i,j}]_{q^{\partial_{w}}}((z-w)^{-1}+(w-z)^{-1}).
\end{align}
Note that the very last relation amounts to
\begin{align}
&[h_{i,\Y}(z),X_{j}^{-}(w-\kappa\hbar)]\\
=\ &-X_{j}^-(w-\kappa\hbar)[r_ia_{i,j}]_{q^{\partial_{w}}}((z-w+\kappa\hbar)^{-1}+(w-z-\kappa\hbar)^{-1}),\nonumber
\end{align}
which in the presence of (\ref{hiY-Hj+}) is equivalent to
\begin{align}
[h_{i,\Y}(z),x_{j,\Y}^{-}(w)]
=-x_{j,\Y}^-(w)[r_ia_{i,j}]_{q^{\partial_{w}}}((z-w+\kappa\hbar)^{-1}+(w-z+\kappa\hbar)^{-1}).
\end{align}

For $s\in I$, setting
$$\bar{X}_s(z)=X_s^-(z-\kappa\hbar),\ \ K_s(z)=H_s^+(z-\kappa\hbar/2)\inv,$$
we have $x_{s,\Y}^-(z)=\bar{X}_s(z)K_s(z).$
It is clear that (DY6) is equivalent to
\begin{align}
(z-w+r_sa_{s,t}\hbar)\bar{X}_s(z)\bar{X}_t(w)=(z-w-r_sa_{s,t}\hbar) \bar{X}_t(w)\bar{X}_s(z)\label{bar-X-bar-X}
\end{align}
for $s,t\in I$. On the other hand, in the presence of (DY1) and (DY3) we have
\begin{align}
&K_s(z)K_t(w)=K_t(w)K_s(z), \\
&K_s(z)\bar{X}_t(w)=\left(\frac{z-w+r_sa_{s,t}\hbar}
{z-w-r_sa_{s,t}\hbar}\right)\bar{X}_t(w)K_s(z).\label{Kr-bar-Xs}
\end{align}
Furthermore, we have
\begin{align}
&(z-w-r_sa_{s,t}\hbar)x_{s,\Y}^{-}(z)x_{t,\Y}^{-}(w)
=(z-w+r_sa_{s,t}\hbar)\bar{X}_s(z)\bar{X}_t(w)K_s(z)K_t(w),\label{KX-bar}\\
&(z-w+r_sa_{s,t}\hbar)x_{t,\Y}^{-}(w)x_{s,\Y}^{-}(z)
=(z-w-r_sa_{s,t}\hbar)\bar{X}_t(w)\bar{X}_s(z) K_t(w)K_s(z),
\end{align}
noticing that $r_sa_{s,t}=r_ta_{t,s}$. Then it follows that (DY6) is equivalent to
\begin{align}\label{xrY-xsY}
(z-w-r_sa_{s,t}\hbar)x_{s,\Y}^{-}(z)x_{t,\Y}^{-}(w)
=(z-w+r_sa_{s,t}\hbar) x_{t,\Y}^{-}(w)x_{s,\Y}^{-}(z).
\end{align}


Next, we consider the Serre relations. As $x_{s,\Y}^{+}(z)=X_s^{+}(z)$ for $s\in I$,
we only need to consider the Serre relation for $(x_{i,\Y}^{-}(z), x_{j,\Y}^{-}(w),1-a_{i,j})$.
With relation (\ref{bar-X-bar-X}), by Lemma \ref{first} we have
$$(z-w)^{-1}(z-w+2r_i\hbar)\bar{X}_i(z)\bar{X}_i(w)
=(w-z)^{-1}(w-z+2r_i\hbar)\bar{X}_i(w)\bar{X}_i(z).$$
It follows that
$$\nob{\bar{X}_i(z_1)\bar{X}_i(z_2)\cdots \bar X_i(z_{m_{i,j}})\bar{X}_j(w)}= \nob{\bar{X}_i(z_{m_{i,j}})\cdots\bar{X}_i(z_1)\bar{X}_j(w)}.$$
Then using (\ref{xrY-xsY}) and (\ref{KX-bar}) we have
\begin{align*}
 &\prod_{s=1}^{1-a_{i,j}}z_s\inv\delta\left(\frac{w-r_i(a_{i,j}+2s-2)\hbar}{z_s}\right)
 \nob{x_{i,\Y}^{-}(z_1)x_{i,\Y}^{-}(z_2)\cdots x_{i,\Y}^-(z_{1-a_{i,j}})x_{j,\Y}^{-}(w)}\\
 =\ &\prod_{s=1}^{1-a_{i,j}}z_s\inv\delta\left(\frac{w-r_i(a_{i,j}+2s-2)\hbar}{z_s}\right)
 \nob{\bar{X}_i(z_1)\cdots \bar{X}_i(z_{1-a_{i,j}})\bar{X}_j(w)}\\
 &\quad\times K_i(z_1)\cdots K_i(z_{1-a_{i,j}})K_j(w)\\
 =\ &\prod_{s=1}^{1-a_{i,j}}z_s\inv\delta\left(\frac{w-r_i(a_{i,j}+2s-2)\hbar}{z_s}\right)
  \nob{\bar{X}_i(z_{1-a_{i,j}})\cdots\bar{X}_i(z_1)\bar{X}_j(w)}\\
 &\quad\times K_i(z_1)\cdots K_i(z_{1-a_{i,j}})K_j(w).
\end{align*}
Applying $\Res_{z_1}\cdots\Res_{z_{1-a_{i,j}}}$ we get
\begin{eqnarray*}
&& \nob{x_i^-(w-r_ia_{i,j}\hbar) x_i^-(w-r_i(a_{i,j}+2)\hbar)\cdots x_i^-(w+r_ia_{i,j}\hbar)x_j^-(w) }\\
&=&  \nob{\bar X_i(w+r_ia_{i,j}\hbar) \bar X_i(w+r_i(a_{i,j}+2)\hbar)\cdots \bar X_i(w-r_ia_{i,j}\hbar)\bar X_j(w) }\\
&&\quad\times K_i(w-r_ia_{i,j}\hbar)K_i(w-r_i(a_{i,j}+2)\hbar) K_i(w+r_ia_{i,j}\hbar)K_j(w).
  \end{eqnarray*}
From this we see that $$ \nob{x_i^-(w-r_ia_{i,j}\hbar) x_i^-(w-r_i(a_{i,j}+2)\hbar)\cdots x_i^-(w+r_ia_{i,j}\hbar)x_j^-(w) }=0$$
 if and only if
\begin{align*}
  \nob{\bar X_i(w+r_ia_{i,j}\hbar) \bar X_i(w+r_i(a_{i,j}+2)\hbar)\cdots \bar X_i(w-r_ia_{i,j}\hbar)\bar X_j(w) }=0.
\end{align*}
Then by Proposition \ref{prop:Serre}, the Serre-like relation for $(\bar{X}_i(z),\bar{X}_j(z),1-a_{i,j})$ holds if and only if
 the Serre-like relation for $(x_{i,\Y}^{-}(z), x_{j,\Y}^{-}(z),1-a_{i,j})$ holds.

Finally, we consider relation (DY5). A variation of this is
\begin{align*}
 [X_i^+(z),\bar{X}_j(w)]
    =\delta_{i,j}\frac{1}{2r_i\hbar}
    \(\!
        K_i(w)^{-1}z\inv\delta\!\(\frac{w}{z}\)
    - H_i^-(w-3/2 \kappa\hbar)z\inv\delta\!\(\frac{w-2\kappa\hbar}{z}\) \)\!.
    \end{align*}
 From (DY3) we have
     \begin{align*}
  X_i^+(z)K_j(w)=K_j(w)X_i^+(z) \(\frac{w-z+r_ja_{j,i}\hbar}{w-z-r_ja_{j,i}\hbar}\)\!.
  \end{align*}
Then
\begin{align*}
&x_{i,\Y}^+(z)x_{j,\Y}^-(w)  - \( \frac{w-z+r_ja_{j,i}\hbar}{w-z-r_ja_{j,i}\hbar}\)
x_{j,\Y}^-(w)x_{i,\Y}^+(z)\nonumber\\
=\ & X_{i}^+(z)\bar{X}_j(w)K_j(w)
- \( \frac{w-z+r_ja_{j,i}\hbar}{w-z-r_ja_{j,i}\hbar}\)\bar{X}_j(w)K_j(w) X_{i}^+(z)
\nonumber\\
=\ & X_{i}^+(z)\bar{X}_j(w)K_j(w)-\bar{X}_j(w)X_{i}^+(z)K_j(w) \nonumber\\
=\ & [X_{i}^+(z),\bar{X}_j(w)]K_j(w). \nonumber
\end{align*}
On the other hand, we have
  \begin{align*}
 &\delta_{i,j}\frac{1}{2r_i\hbar}
\(\!z\inv\delta\!\(\frac{w}{z}\)-C_i(w)  z\inv\delta\!\(\frac{w-2\kappa\hbar}{z}\)\!\)\!\\
=\ &\delta_{i,j}\frac{1}{2r_i\hbar}
    \(\! z^{-1}\delta\(\frac{w}{z}\)
    - H_i^-(w-\frac{3}{2} \kappa\hbar)H_i^+(w-\frac{1}{2} \kappa\hbar)\inv z\inv\delta\!\(\frac{w-2\kappa\hbar}{z}\)\! \)\nonumber\\
=\ &\delta_{i,j}\frac{1}{2r_i\hbar}
    \(\!
        K_i(w)^{-1}z\inv\delta\!\(\frac{w}{z}\)
    - H_i^-(w-\frac{3}{2} \kappa\hbar)z\inv\delta\!\(\frac{w-2\kappa\hbar}{z}\)\! \)\!K_j(w).
\end{align*}
This proves the equivalence regarding (DY5).
 Now, the proof is complete.
\end{proof}

With Theorem \ref{main-DY},
from now on we shall freely use $x_{i,\Y}^\pm(z), \ h_{i,\Y}(z)$  $(i\in I)$ as generating functions for $\wh{\mathcal{DY}}(A)$.
For $i\in I$, write
\begin{align}
x_{i,\Y}^{\pm}(z)=\sum_{n\in \Z}x_{i,\Y}^{\pm}(n)z^{-n-1},\quad h_{i,\Y}(z)=\sum_{n\in \Z}h_{i,\Y}(n)z^{-n-1}.
\end{align}

\begin{de}\label{def-vacuum-module}
Let $\ell\in \C$. {\em A vacuum $\wh{\mathcal{DY}}(A)$-module} of level $\ell$ is a pair $(W,w_0)$, where $W$ is a
restricted $\wh{\mathcal{DY}}(A)$-module of level $\ell$ and  $w_0$ is a vector in $W$ such that
\begin{align}
&W=(\wh{\mathcal{DY}}(A)w_0)[[\hbar]]',\\
&x_{i,\Y}^{\pm}(n)w_0=0=h_{i,\Y}(n)w_0\quad \text{ for }i\in I,\ n\ge 0.
\end{align}
\end{de}

Let $W$ be any restricted $\wh{\mathcal{DY}}(A)$-module.
Identify $\kappa$ with its corresponding element of $\End W$ and set
\begin{align}
U_W=\set{h_{i,\Y}(z),x_{i,\Y}^\pm(z)}{i\in I}\cup\{1_W, \kappa\}.
\end{align}
It follows from Theorem \ref{main-DY} and Lemmas \ref{lem-a-b-C-D} and \ref{lem-fab=gba}
that $U_W$ is an $\mathcal{S}$-local subset of $\E_{\hbar}(W)$, which is $\hbar$-adically compatible
by Lemma \ref{W-hwqva}. Then  we have an $\hbar$-adic nonlocal vertex algebra $\<U_W\>$ with $W$ as a faithful module.
Furthermore, we have:

\begin{prop}\label{Y-E-mod-struct}
Let $W$ be any restricted $\wh{\mathcal{DY}}(A)$-module.
Then the $\hbar$-adic nonlocal vertex algebra $\<U_W\>$ is also a restricted $\wh{\mathcal{DY}}(A)$-module
with $\kappa$ acting as its associated left multiplication on $\<U_W\>$ and with
\begin{align*}
  h_{i,\Y}(z_0)=Y_\E(h_{i,\Y}(z),z_0),\quad
  x_{i,\Y}^{\pm}(z_0)=Y_\E(x_{i,\Y}^{\pm}(z),z_0)\quad \text{ for } i\in I.
\end{align*}
Furthermore, if $W$ is of level $\ell\in \C$, then $\<U_W\>$ is a $\wh{\mathcal{DY}}(A)$-module of level $\ell$.
\end{prop}

\begin{proof} First, identify $\kappa$ with its associated left multiplication, an operator on $\E_{\hbar}(W)$.
Then $Y_{\E}(\kappa,z)=\kappa$.
As  $\kappa$ is a central element of $\wh{\mathcal{DY}}(A)$, it follows that
\begin{align*}
[\kappa,Y_{\E}(h_{i,\Y}(z),z_2)]=0=[\kappa,Y_{\E}(x_{i,\Y}^\pm(z),z_2)]\quad \text{ for }i\in I.
\end{align*}
With the relations in Theorem \ref{main-DY},
by Lemmas \ref{pq-va-module-relation} and \ref{Y-V-W-a-b-relation} we have
\begin{align*}
&[Y_{\E}(h_{i,\Y}(z),z_1),Y_{\E}(h_{j,\Y}(z),z_2)]\nonumber\\
    =\ &  [r_ia_{i,j}]_{q^{\pd{z_2}}}
        [\kappa]_{q^{\pd{z_2}}}\( (z_1-z_2+ \kappa\hbar)^{-2}-(z_2-z_1+\kappa\hbar)^{-2}\)\!,\\
&[Y_{\E}(h_{i,\Y}(z),z_1),Y_{\E}(x_{j,\Y}^\pm(z),z_2]\nonumber\\
=\ &  \pm Y_{\E}(x_{j,\Y}^\pm(z),z_2)[r_ia_{i,j}]_{q^{\pd{z_2}}}\((z_1-z_2+\kappa\hbar)^{-1}-(-z_2+z_1- \kappa\hbar)^{-1}\)\!,\\
&(z_1-z_2-r_ia_{i,j}\hbar)Y_{\E}(x_{i,\Y}^\pm(z),z_1)Y_{\E}(x_{j,\Y}^\pm(z),z_2)\nonumber\\
=\ &(z_1-z_2+r_ia_{i,j}\hbar)Y_{\E}(x_{j,\Y}^\pm(z),z_2)Y_{\E}(x_{i,\Y}^\pm(z),z_1)
\end{align*}
for $i,j\in I$.
 In case $a_{i,j}\le 0$,  by  Corollary \ref{Sere-relation} the Serre-like relations hold for
 the triples $(Y_{\E}(x_{i,\Y}^\pm(z),z_1),Y_{\E}(x_{j,\Y}^\pm(z),z_1),1-a_{i,j})$.

Now, we consider the remaining commutator relation.
For $i,j\in I$, if $i\ne j$, by Lemma \ref{lem-fab=gba} we have
\begin{align*}
Y_{\E}(x_{i,\Y}^+(z),z_1)Y_{\E}(x_{j,\Y}^-(z),z_2)-\!\( \frac{z_2-z_1+r_ia_{i,j}\hbar}{z_2-z_1-r_ia_{i,j}\hbar}\)\!
Y_{\E}(x_{j,\Y}^-(z),z_2)Y_{\E}(x_{i,\Y}^+(z),z_1)=0.
\end{align*}
Assume $i=j$. Recall
$$C_i(x)
=\exp\!\(-G(\partial_x)q^{-\kappa\partial_x}h_{i,\Y}^{-}(x)\)
\exp\!\(- G(\partial_x)q^{-\kappa\partial_x}h_{i,\Y}^{+}(x)\)\!.$$
Notice that
\begin{align*}
  z\inv\delta\!\(\frac{w-2\kappa\hbar}{z}\)\!C_i(w)
=  \sum_{n\ge 0}(-2\kappa\hbar)^n
    C_i(w) \frac{1}{n!} \left(\frac{\partial}{\partial w}\right)^n  z\inv\delta\!\(\frac{w}{z}\)\!.
\end{align*}
By Proposition \ref{prop-2.25} we have
$$C_i(z)=1_W-2\hbar x_{i,\Y}^+(z)_0x_{i,\Y}^{-}(z)\in \<U_W\>$$
and
\begin{align}\label{Y-E-xi-xj}
&Y_{\E}(x_{i,\Y}^+(z),z_1)Y_{\E}(x_{i,\Y}^-(z),z_2)-\!\( \frac{z_2-z_1+2r_i\hbar}{z_2-z_1-2r_i\hbar}\)\!
Y_{\E}(x_{i,\Y}^-(z),z_2)Y_{\E}(x_{i,\Y}^+(z),z_1)\\
=\ &\frac{1}{2r_i\hbar}\!\(\! z_1\inv\delta\!\(\frac{z_2}{z_1}\)- Y_{\E}\(C_i(z),z_2\)  z_1\inv\delta\!\(\frac{z_2-2\kappa\hbar}{z_1}\)\!\)\!.\nonumber
\end{align}
Recall that $h_{i,\Y}(x)\in \<U_W\>$. In the following,  we use Proposition \ref{Y-W-E(a,z)} to show
\begin{align}\label{Ci(x)-exp}
C_i(x)=q^{(1-\kappa)\D}E^{-}(-h_{i,\Y}(x),-2\hbar)1_W,
\end{align}
where  $\D$ denotes the $D$-operator of $\<U_W\>$.

Note that
$$[Y_{\E}^{\pm }(h_{i,\Y}(x),z),Y_{\E}^{\pm}(h_{i,\Y}(x),w)]=0,$$
$$[Y_{\E}^{-}(h_{i,\Y}(x),z),Y_{\E}^{+}(h_{i,\Y}(x),w)]
=[2r_i]_{q^{\partial_w}}[\kappa]_{q^{\partial_w}}(z-w+\kappa\hbar)^{-2}.$$
With $W$ naturally a $ \<U_W\>$-module, by Proposition \ref{Y-W-E(a,z)} we have
  \begin{align}\label{first-version}
 &Y_W(E^{-}(-h_{i,\Y}(z),\zeta){\bf 1},x)\\
 =&\exp \(\zeta L(\zeta\partial_x)Y_W^{+}(h_{i,\Y}(z),x)\) \exp (\zeta L(\zeta\partial_x)Y_W^{-}(h_{i,\Y}(z),x))\nonumber
 \end{align}
for $\zeta\in \hbar\C[[\hbar]]$, where $L(x)=\frac{e^x-1}{x}\in \C[[x]]$. Note that
 \begin{align}\label{G(x)-L(x)-relation}
 -G(x)q^{-x}=\frac{e^{-2\hbar x}-1}{x}=-2\hbar L(-2\hbar x).
 \end{align}
 With $Y_W^{\pm}(h_{i,\Y}(z),x)=h_{i,\Y}^{\mp}(x)$, using (\ref{first-version}) with $\zeta=-2\hbar$ we get
\begin{align*}
&Y_W\!\(q^{(1-\kappa)\D}E^{-}(-h_{i,\Y}(z),-2\hbar)1_W,x\)\!\\
=\ &q^{(1-\kappa)\partial_x}Y_W\!\(E^{-}(-h_{i,\Y}(z),-2\hbar)1_W,x\)\nonumber\\
=\ &q^{(1-\kappa)\partial_x}\exp\!\(-2\hbar L(-2\hbar \partial_x)Y_W^{+}(h_{i,\Y}(z),x)\)\!
\exp\!\(-2\hbar L(-2\hbar \partial_x)Y_W^{-}(h_{i,\Y}(z),x)\)\nonumber\\
=\ &\exp\!\(-G(\partial_x)q^{-\kappa\partial_x}Y_W^{+}(h_{i,\Y}(z),x)\)\!
\exp\!\(- G(\partial_x)q^{-\kappa\partial_x}Y_W^{-}(h_{i,\Y}(z),x)\)\nonumber\\
=\ &\exp\!\(-G(\partial_x)q^{-\tilde{\kappa}\partial_x}h_{i,\Y}^{-}(x)\)
\exp\!\(-G(\partial_x)q^{-\tilde{\kappa}\partial_x}h_{i,\Y}^{+}(x)\)\!\nonumber\\
=\ &C_i(x)\nonumber\\
=\ & Y_W(C_i(z),x).\nonumber
\end{align*}
Then (\ref{Ci(x)-exp}) follows as $W$ is a faithful module.

Now, using (\ref{Ci(x)-exp}) and using Proposition \ref{Y-W-E(a,z)} for the adjoint $\<U_W\>$-module we get
 \begin{align*}
 &Y_{\E}(C_i(z),x)\\
 =\ &Y_{\E}\!\(q^{(1-\kappa)\D}E^{-}(-h_{i,\Y}(z),-2\hbar)1_W,x\)\nonumber\\
 =\ &q^{(1-\kappa)\partial_x}Y_{\E}(E^{-}(-h_{i,\Y}(z),-2\hbar)1_W,x)\nonumber\\
 =\ &q^{(1-\kappa)\partial_x}\exp\!\(-2\hbar L(-2\hbar \partial_x)Y_{\E}^{+}(h_{i,\Y}(z),x)\)\!
\exp\!\(-2\hbar L(-2\hbar \partial_x)Y_{\E}^{-}(h_{i,\Y}(z),x)\)\nonumber\\
=\ & \exp\!\(-G(\partial_x)q^{-\kappa\partial_x}Y_{\E}^{+}(h_{i,\Y}(z),x)\)\!
\exp\!\(-G(\partial_x)q^{-\kappa\partial_x}Y_{\E}^{-}(h_{i,\Y}(z),x)\)\!.\nonumber
\end{align*}
Finally, by invoking Theorem \ref{main-DY} we conclude that $\<U_W\>$ is a restricted $\wh{\mathcal{DY}}(A)$-module with
$h_{i,\Y}(z_0)=Y_\E(h_{i,\Y}(z),z_0)$,
  $x_{i,\Y}^{\pm}(z_0)=Y_\E(x_{i,\Y}^{\pm}(z),z_0)$ for $i\in I$.

The second assertion is clear.
\end{proof}

\subsection{Universal vacuum $\wh{\mathcal{DY}}(A)$-module ${\mathcal{V}}_A(\ell)$}
 Let $\ell\in \C$. We here construct a universal vacuum $\wh{\mathcal{DY}}(A)$-module ${\mathcal{V}}_A(\ell)$ of level $\ell$
 and present the second main result of this paper.

 First, let $\widetilde{\mathcal{U}}$ denote the associative algebra with identity over $\C$, generated by set
$$\{\tilde{\kappa}\}\cup \{\tilde{e}_i(m), \tilde{f}_i(m), \tilde{h}_i(m)\ |\ i\in I,\ m\in \Z\}.$$
For $X\in \{ \tilde{e}_i,\tilde{f}_i, \tilde{h}_i\ |\ i\in I\}$, form a generating function
\begin{align}
X(z)=\sum_{m\in \Z}X(m)z^{-m-1}.
\end{align}

Second, consider the $\C[[\hbar]]$-algebra $\widetilde{\mathcal{U}}[[\hbar]]$, and
define $\widetilde{J}_{A,\ell}$ to be the two-sided ideal,
generated by element $\tilde{\kappa}-\ell$ and by the following relations for $i,j\in I$:
\begin{align}
&\tilde{h}_i(z)\tilde{h}_j(w)-\tilde{h}_j(w)\tilde{h}_i(z)\\
&\quad\nonumber
=[r_ia_{i,j}]_{q^{\partial_w}}[\ell]_{q^{\partial_w}}((z-w+\ell\hbar)^{-2}-(w-z+\ell\hbar)^{-2}),\\
&\tilde{h}_i(z)\tilde{e}_j(w)-\tilde{e}_j(w)\tilde{h}_i(z)\\
&\quad\nonumber
=\tilde{e}_j(w)[r_ia_{i,j}]_{q^{\partial_w}}((z-w+\ell\hbar)^{-1}-(w-z+\ell\hbar)^{-1}),\\
&\tilde{h}_i(z)\tilde{f}_j(w)-\tilde{f}_j(w)\tilde{h}_i(z)\\
&\quad\nonumber
=-\tilde{f}_j(w)[r_ia_{i,j}]_{q^{\partial_w}}((z-w+\ell\hbar)^{-1}-(w-z+\ell\hbar)^{-1}),\\
&(z-w-r_ia_{i,j}\hbar)\tilde{e}_i(z)\tilde{e}_j(w)
=(z-w+r_ia_{i,j}\hbar)\tilde{e}_j(w)\tilde{e}_i(z),\label{e-i-e-j}\\
&(z-w-r_ia_{i,j}\hbar)\tilde{f}_i(z)\tilde{f}_j(w)
=(z-w+r_ia_{i,j}\hbar)\tilde{f}_j(w)\tilde{f}_i(z),\label{f-i-f-j}\\
&\label{i-not-j}
(z-w+r_ia_{i,j}\hbar) \tilde{e}_i(z)\tilde{f}_j(w)=(z-w-r_ia_{i,j}\hbar) \tilde{f}_j(w)\tilde{e}_i(z)\quad \text{ if }i\ne j,\\
&\label{i=j}
(z-w)(z-w+2\ell\hbar)\!\big((z-w+2r_i\hbar) \tilde{e}_i(z)\tilde{f}_i(w)\\
&\qquad\nonumber-(z-w-2r_i\hbar) \tilde{f}_i(w)\tilde{e}_i(z)\big)\!=0,
\end{align}
and the Serre-like relations for $(\tilde{e}_i(z),  \tilde{e}_j(z),1-a_{i,j})$ and $(\tilde{f}_i(z),  \tilde{f}_j(z),1-a_{i,j})$
for $i,j\in I$ with $a_{i,j}\le 0$.

Note that for any $s,t\in \Z$, the coefficients of monomial $z^sw^t$ in all the relations above indeed
lie in $\widetilde{\mathcal{U}}[[\hbar]]$.

In view of Lemma \ref{basic-facts-top-algebra}, $\overline{[\widetilde{J}_{A,\ell}]}$ is a two-sided ideal of $\widetilde{\mathcal{U}}[[\hbar]]$
and a topologically free $\C[[\hbar]]$-submodule.  Then set
\begin{align}
\mathcal{U}_{A,\ell}=\widetilde{\mathcal{U}}[[\hbar]] / \overline{[\widetilde{J}_{A,\ell}]},
\end{align}
which is an associative $\C[[\hbar]]$-algebra and
a topologically free $\C[[\hbar]]$-module by Lemma \ref{strong-submodule}. For $i\in I,\ m\in \Z$, set
\begin{align}
\bar{e}_i(m)=\tilde{e}_i(m)+\overline{[\widetilde{J}_{A,\ell}]},\  \bar{f}_i(m)=\tilde{f}_i(m)+\overline{[\widetilde{J}_{A,\ell}]},\
\bar{h}_i(m)=\tilde{h}_i(m)+\overline{[\widetilde{J}_{A,\ell}]}\in \mathcal{U}_{A,\ell}.
\end{align}
Then define generating functions $\bar{e}_i(z), \bar{f}_i(z), \bar{h}_i(z)$ for $i\in I$ correspondingly.

\begin{rem}\label{rem-DY-A}
{\em Note that from Theorem \ref{main-DY},
every restricted $\wh{\mathcal{DY}}(A)$-module of level $\ell$ is naturally a restricted
$\mathcal{U}_{A,\ell}$-module. }
\end{rem}

Third, let $\mathcal{J}_{A,\ell}$ be the left ideal of $\mathcal{U}_{A,\ell}$, generated by the subset
$$\{ \bar{e}_i(n),\bar{f}_i(n), \bar{h}_i(n)\ |\ i\in I,\ n\ge 0\}.$$
Then $\overline{[\mathcal{J}_{A,\ell}]}$ is a topologically free $\C[[\hbar]]$-submodule and a left ideal of $\mathcal{U}_{A,\ell}$.
Define
\begin{align}
\mathcal{V}_{A,\ell}=\mathcal{U}_{A,\ell}/\overline{[\mathcal{J}_{A,\ell}]},
\end{align}
which is a (left) $\mathcal{U}_{A,\ell}$-module and a  topologically free $\C[[\hbar]]$-module.  Set
$${\bf 1}=1+\overline{[\mathcal{J}_{A,\ell}]}\in \mathcal{V}_{A,\ell},$$
and furthermore set
\begin{align}
\hat{e}_i=\bar{e}_i(-1){\bf 1},\ \  \hat{f}_i=\bar{f}_i(-1){\bf 1},\ \ \hat{h}_i=\bar{h}_i(-1){\bf 1}\in \mathcal{V}_{A,\ell}
\quad \text{for }i\in I.
\end{align}

As a key result of this section, we have:

\begin{prop}\label{pre-universal-qva}
There exists an $\hbar$-adic weak quantum vertex algebra structure on $\mathcal{V}_{A,\ell}$, which is uniquely determined by
the condition that ${\bf 1}$ is the vacuum vector and
\begin{align}
Y(\hat{e}_i,z)=\bar{e}_i(z),\ \ Y(\hat{f}_i,z)=\bar{f}_i(z),\ \ Y(\hat{h}_i,z)=\bar{h}_i(z)\quad \text{for }i\in I.
\end{align}
\end{prop}

\begin{proof} First of all, it is clear that $\mathcal{V}_{A,\ell}$ is topologically spanned over $\C[[\hbar]]$ by vectors
$X_1\cdots X_r {\bf 1}$ for $r\ge 0,\ X_j\in\{ \bar{e}_i(m),\bar{f} _i(m),\bar{h}_i(m)\ |\ i\in I,\ m\in \Z\}.$
From the construction of $\mathcal{V}_{A,\ell}$ as a $\mathcal{U}_{A,\ell}$-module, we have
$$\bar{e}_i(x){\bf 1},\ \bar{f}_i(x){\bf 1},\ \bar{h}_i(x){\bf 1}\in \mathcal{V}_{A,\ell}[[x]]\subset (\mathcal{V}_{A,\ell})_{\hbar}((x))
\ \ \text{ for }i\in I.$$
Here, we need to show that $\mathcal{V}_{A,\ell}$ is a restricted $\mathcal{U}_{A,\ell}$-module.
Indeed, it follows from induction and technical Lemmas \ref{restricted-V-g-1} and \ref{restricted-V-g-2}
that for every $w\in \mathcal{V}_{A,\ell}$,
$$\bar{e}_i(x)w,\ \bar{f}_i(x)w,\ \bar{h}_i(x)w\in (\mathcal{V}_{A,\ell})_{\hbar}((x))
\ \ \text{ for }i\in I.$$
That is,
\begin{align}
\bar{e}_i(x),\ \bar{f}_i(x), \ \bar{h}_i(x)\in \E_{\hbar}(\mathcal{V}_{A,\ell})\ \ \text{ for }i\in I.
\end{align}
Furthermore, by Lemma \ref{new-added} $U:=\{\bar{e}_i(x),\bar{f}_i(x),\bar{h}_i(x)\ |\  i\in I\}$
is an $\SY$-local subset of $\E_{\hbar}(\mathcal{V}_{A,\ell})$.
Then $U$ generates an $\hbar$-adic nonlocal vertex algebra $\<U\>$ inside $\E_{\hbar}(\mathcal{V}_{A,\ell})$
with $\mathcal{V}_{A,\ell}$ as a faithful module.
It follows from Lemmas  \ref{lem-a-b-C-D} and \ref{lem-fab=gba} that $\<U\>$ is a $\mathcal{U}_{A,\ell}$-module
with $\bar{e}_i(m), \bar{f}_i(m), \bar{h}_i(m)$ for $i\in I,\ m\in \Z$ acting as
$\bar{e}_i(x)_m, \bar{f}_i(x)_m, \bar{h}_i(x)_m$, respectively, and
it is clear that $\<U\>$ is a restricted $\mathcal{U}_{A,\ell}$-module.
Note that
$$u(x)_n1_{\mathcal{V}_{A,\ell}}=0\ \ \text{ for all }u(x)\in \<U\>,\ n\in \N.$$
It follows from the construction of $\mathcal{V}_{A,\ell}$ and Lemma \ref{simple-fact-3}
that there is a $\mathcal{U}_{A,\ell}$-module morphism
$\phi: \mathcal{V}_{A,\ell}\rightarrow \<U\>$ such that $\phi({\bf 1})=1_{\mathcal{V}_{A,\ell}}$.
Now, by invoking Theorem \ref{hwqva-construction} we conclude that there exists
an $\hbar$-adic weak quantum vertex algebra structure on $\mathcal{V}_{A,\ell}$ with the required condition.
\end{proof}

Furthermore, using Theorem \ref{hwqva-module} we immediately have:

\begin{prop}\label{hqva-module-A-ell}
For any restricted $\mathcal{U}_{A,\ell}$-module $W$, there exists a
$\mathcal{V}_{A,\ell}$-module structure $Y_W(\cdot,z)$ which is uniquely determined by
 \begin{align}
 Y_W(\hat{e}_{i},z)=\bar{e}_{i}(z),\ \ Y_W(\hat{f}_{i},z)=\bar{f}_{i}(z),\ \ Y_W(\hat{h}_{i},z)=\bar{h}_{i}(z)\ \ \text{ for }i\in I.
 \end{align}
On the other hand, for any $\mathcal{V}_{A,\ell}$-module $(W,Y_W)$,
there exists a restricted
$\mathcal{U}_{A,\ell}$-module structure on $W$ such that
 \begin{align}
 \bar{e}_{i}(z)=Y_W(\hat{e}_{i},z),\ \ \bar{f}_{i}(z)=Y_W(\hat{f}_{i},z),\ \ \bar{h}_{i}(z)=Y_W(\hat{h}_{i},z)\ \ \text{ for }i\in I.
 \end{align}
\end{prop}

We continue to define a quotient of the $\hbar$-adic weak quantum vertex algebra $\mathcal{V}_{A,\ell}$,
which will be a vacuum $\wh{\mathcal{DY}}(A)$-module of level $\ell$.
For $i\in I$, set
\begin{align}
C_i=e^{(1-\ell)\hbar \D}E^{-}(-\hat{h}_i,-2\hbar){\bf 1}\in \mathcal{V}_{A,\ell},
\end{align}
where $\D$ is the canonical $D$-operator of $\mathcal{V}_{A,\ell}$ and
$$ E^{-}(-\hat{h}_i,-2\hbar)=\exp\!\(\sum_{n\in \Z_{+}}\frac{1}{n}\hat{h}_i(-n)(-2\hbar)^n\right)\!.$$
Note that $C_i={\bf 1}-2\hbar \hat{h}_i+O(\hbar^2)$ (as $\D {\bf 1}=0$).

Denote by $K$ the subset of $\mathcal{V}_{A,\ell}$, consisting of vectors
 \begin{align}
 (\hat{e}_i)_0(\hat{f}_j)-\frac{\delta_{i,j}}{2r_i\hbar}({\bf 1}-C_i),\ \
 (\hat{e}_i)_n(\hat{f}_j)-\frac{\delta_{i,j}\ell}{r_i} (-2\ell\hbar)^{n-1}C_i\ \ \text{ for }i,j\in I,\ n\ge 1.
 \end{align}

 \begin{de}
 Let $(K)$ denote the strong ideal of $\mathcal{V}_{A,\ell}$, generated by $K$.
 Set
 \begin{align}
 {\mathcal{V}}_{A}(\ell)=\mathcal{V}_{A,\ell}/(K),
 \end{align}
which is a topologically free $\C[[\hbar]]$-module and an $\hbar$-adic weak quantum vertex algebra.
\end{de}

For $i\in I$, set
\begin{align}
\check{e}_i=\hat{e}_i+(K),\ \check{f}_i=\hat{f}_i+(K),\ \check{h}_i=\hat{h}_i+(K)\in {\mathcal{V}}_A(\ell),
\end{align}
 and furthermore set
\begin{align}
\mathfrak{c}_i=e^{(1-\ell)\hbar \D}E^{-}(-\check{h}_i,-2\hbar){\bf 1}\in {\mathcal{V}}_A(\ell),
\end{align}
where $\D$ is the $D$-operator of ${\mathcal{V}}_A(\ell)$.
Then we have:

\begin{thm}\label{YD-vacuum-module}
 Let $\ell\in \C$. Then there exists a restricted $\wh{\mathcal{DY}}(A)$-module structure
 of level $\ell$ on $\mathcal{V}_A(\ell)$, which is uniquely determined by
 \begin{align}
 x_{i,\mathcal{Y}}^{+}(z)=Y(\check{e}_i,z),\ \ x_{i,\Y}^{-}(z)=Y(\check{f}_i,z),\ \ h_{i,\Y}(z)=Y(\check{h}_i,z)\ \ \text{ for }i\in I.
 \end{align}
 Furthermore, $(\mathcal{V}_A(\ell), {\bf 1})$ is a universal vacuum $\wh{\mathcal{DY}}(A)$-module
 of level $\ell$.
\end{thm}

\begin{proof} We first prove the existence.
By Lemma \ref{h-adic-compatibility}, relations (\ref{i-not-j}) and (\ref{i=j}) imply
 $$Y(\check{e}_i,z)Y(\check{f}_j,w) \sim \(\frac{w-z+r_ia_{i,j}\hbar}{w-z-r_ia_{i,j}\hbar}\)\!Y(\check{f}_j,w)Y(\check{e}_i,z).$$
 Then the $\SY$-commutator formula gives
 \begin{align}\label{e-f-S-commutator}
 &Y(\check{e}_i,z)Y(\check{f}_j,w) -\! \(\frac{w-z+r_ia_{i,j}\hbar}{w-z-r_ia_{i,j}\hbar}\)\!Y(\check{f}_j,w)Y(\check{e}_i,z)\\
 =\ &\sum_{n\ge 0}Y\((\check{e}_i)_n(\check{f}_j),w\)\frac{1}{n!}\(\frac{\partial}{\partial w}\)^nz^{-1}\delta\!\left(\frac{w}{z}\right)\!.\nonumber
 \end{align}
 From the construction of ${\mathcal{V}}_A(\ell)$, we have
\begin{align}\label{e-f-singular}
(\check{e}_i)_0(\check{f}_j)=\frac{\delta_{i,j}}{2r_i\hbar}({\bf 1}-\mathfrak{c}_i),\ \
 (\check{e}_i)_n(\check{f}_j)=\frac{\delta_{i,j}\ell }{r_i} (-2\ell\hbar)^{n-1}\mathfrak{c}_i
 \end{align}
 for $i,j\in I,\ n\ge 1$.
 Combining this with (\ref{e-f-S-commutator}) we get
   \begin{align}\label{key-relation}
 &Y(\check{e}_i,z)Y(\check{f}_j,w) -\! \(\frac{w-z+r_ia_{i,j}\hbar}{w-z-r_ia_{i,j}\hbar}\)\!Y(\check{f}_j,w)Y(\check{e}_i,z)\\
 =\ &\delta_{i,j}\frac{1}{2r_i\hbar}\!\left(\!z^{-1}\delta\!\left(\frac{w}{z}\right)-Y(\mathfrak{c}_i,w)z^{-1}\delta\!\left(\frac{w-2\ell\hbar}{z}\right)\!\)\!.\nonumber
  \end{align}
  Note that
  \begin{align*}
G(x)q^{-\ell x}=2\hbar L(-2\hbar x)q^{(1-\ell)x}.
\end{align*}
Using Proposition \ref{Y-W-E(a,z)} we obtain
 \begin{align}\label{explicit-expression}
 &Y(\mathfrak{c}_i,x) \\
 =\ &e^{(1-\ell)\hbar \partial_x}Y\!\(E^{-}(-\check{h}_i,-2\hbar){\bf 1},x\)\nonumber\\
= \ &e^{(1-\ell)\hbar \partial_x}
\exp\! \(-2\hbar L(-2\hbar\partial_x)Y^{+}(\check{h}_i,x)\)\exp\! \(-2\hbar L(-2\hbar\partial_x)Y^{-}(\check{h}_i,x)\)\nonumber\\
 =\ &\exp\!\(-G(\partial_x)q^{-\ell\partial_x}Y^{+}(\check{h}_i,x)\)\exp\!\(-G(\partial_x)q^{-\ell\partial_x}Y^{-}(\check{h}_i,x)\)\!.\nonumber
  \end{align}
%
Then it follows from Theorem \ref{main-DY} that
 ${\mathcal{V}}_A(\ell)$ is a $\wh{\mathcal{DY}}(A)$-module of level $\ell$ with
 $$x_{i,\Y}^{+}(z)=Y(\check{e}_i,z),\ \ x_{i,\Y}^{-}(z)=Y(\check{f}_i,z),\ \ h_{i,\Y}(z)=Y(\check{h}_i,z)\ \ \text{ for }i\in I.$$
 The uniqueness and vacuum module assertions are clear.

 Next, we prove that $({\mathcal{V}}_A(\ell), {\bf 1})$ is universal.
 Let $W$ be any restricted $\wh{\mathcal{DY}}(A)$-module of level $\ell$ with a vacuum vector $w_0$. Set
 $$U_W=\{ x_{i,\Y}^{+}(z),\ x_{i,\Y}^{-}(z),\ h_{i,\Y}(z)\ |\ i\in I\}\cup \{ 1_W\}.$$
 Then $U_W$ is an $\SY$-local subset of $\E_{\hbar}(W)$, which generates
 an $\hbar$-adic nonlocal vertex algebra $\<U_W\>$ with $W$ as a module.
By Lemma \ref{U-vacuum-like}, there is a $\<U_W\>$-module morphism $\psi_0: \<U_W\>\rightarrow W$ with
 $\psi_0(1_W)=w_0$.
  Note that $\<U_W\>$ is a restricted $\wh{\mathcal{DY}}(A)$-module of level $\ell$ by Proposition \ref{Y-E-mod-struct}.
  It is easy to see that $\psi_0$ is also a $\wh{\mathcal{DY}}(A)$-module morphism.

 On the other hand,  as $\<U_W\>$ is a $\mathcal{U}_{A,\ell}$-module by Remark \ref{rem-DY-A},
  there exists a $\mathcal{U}_{A,\ell}$-module
 morphism $\hat{\psi}_z: \mathcal{U}_{A,\ell}\rightarrow \<U_W\>$ such that $\hat{\psi}_z(1)=1_W$.
 With $\<U_W\>$ torsion free and separated, by Lemma \ref{simple-fact-3} $\hat{\psi}_z$ reduces to
 a $\mathcal{U}_{A,\ell}$-module morphism $\tilde{\psi}_z: \mathcal{V}_{A,\ell}\rightarrow \<U_W\>$ with $\tilde{\psi}_z({\bf 1})=1_W$.
 Since $\mathcal{V}_{A,\ell}$ as a $\mathcal{U}_{A,\ell}$-module is generated from ${\bf 1}$,
it follows that $\tilde{\psi}_z$ is a morphism of $\hbar$-adic nonlocal vertex algebras.
Hence, $W$ is a $\mathcal{V}_{A,\ell}$-module via $\tilde{\psi}_z$.
By reversing the arguments in the existence part above, we see that
  $K\subset \ker Y_W(\cdot,z)=\ker (\tilde{\psi}_z)$.
  Then $\tilde{\psi}_z$ reduces to an $\hbar$-adic nonlocal vertex algebra morphism
  $\psi_1: {\mathcal{V}}_A(\ell)\rightarrow \<U_W\>$.
  Consequently, $\psi:=\psi_0\psi_1$ is a $\wh{\mathcal{DY}}(A)$-module morphism from
  ${\mathcal{V}}_A(\ell)$ to $W$ with $\psi({\bf 1})=w_0$.
\end{proof}

In summary, for any $\ell\in \C$, ${\mathcal{V}}_A(\ell)$ is an $\hbar$-adic weak quantum vertex algebra
and a restricted $\wh{\mathcal{DY}}(A)$-module of level $\ell$ with
 \begin{align}
Y(\check{e}_{i},z)=x_{i,\Y}^{+}(z),\ \  Y(\check{f}_{i},z)=x_{i,\Y}^{-}(z),\ \ Y(\check{h}_{i},z)=h_{i,\Y}(z)\ \ \text{ for }i\in I.
 \end{align}
With the notations
 \begin{eqnarray*}
 &&Y(u,z)=\sum_{m\in \Z}u_mz^{-m-1}\quad \text{for }u\in \{ \check{e}_{i},\  \check{f}_{i},\  \check{h}_{i}\ |\ i\in I\}, \\
 &&X(z)=\sum_{m\in \Z}X(m)z^{-m-1}\quad \text{for }X\in \{ x_{i,\Y}^{\pm},\ h_{i,\Y}\ |\ i\in I\},
 \end{eqnarray*}
we have
 \begin{eqnarray}
 (\check{e}_i)_m=x_{i,\Y}^{+}(m),\ \  (\check{f}_i)_m=x_{i,\Y}^{-}(m),\ \  (\check{h}_i)_m=h_{i,\Y}(m)
 \end{eqnarray}
 for $i\in I,\ m\in \Z.$ For convenience, locally set $\check{e}_i^{+}=\check{e}_i$ and $\check{e}_i^{-}=\check{f}_i$.

\begin{rem}\label{DYhqva-relations}
{\em The following relations hold on ${\mathcal{V}}_A(\ell)$ for $i,j\in I$:
\begin{align*}
&[Y(\check{h}_{i},z),Y(\check{h}_{j},w)]
    =[r_ia_{i,j}]_{q^{\partial_{w}}}
        [\ell]_{q^{\partial_{w}}}\( (z-w+\ell\hbar)^{-2}-(w-z+\ell\hbar)^{-2}\)\!,\\
&[Y(\check{h}_{i},z),Y(\check{e}_{j}^\pm,w)]
=\pm Y(\check{e}_{j}^\pm,w)[r_ia_{i,j}]_{q^{\partial_{w}}}\((z-w+\ell\hbar)^{-1}+(w-z+\ell\hbar)^{-1}\)\!,\\
&(z-w-r_ia_{i,j}\hbar)Y(\check{e}_{i}^\pm,z)Y(\check{e}_{j}^\pm,w)
=(z-w+r_ia_{i,j}\hbar)Y(\check{e}_{j}^\pm,w)Y(\check{e}_{i}^\pm,z),\\
&Y(\check{e}_{i},z)Y(\check{f}_{j},w)  - \( \frac{w-z+r_ia_{i,j}\hbar}{w-z-r_ia_{i,j}\hbar}\)\!Y(\check{f}_{j},w)Y(\check{e}_{i},z)\nonumber\\
&\quad=\delta_{i,j}\frac{1}{2r_i\hbar}\!
\(\!z\inv\delta\!\(\frac{w}{z}\)-Y\!\(q^{(1-\ell)\D}E^{-}(-\check{h}_{i},-2\hbar){\bf 1},w\)
 z\inv\delta\!\(\frac{w-2\ell\hbar}{z}\)\!\)\!,
\end{align*}
where $\D$ denotes the $D$-operator of ${\mathcal{V}}_A(\ell)$.
In addition,  for $i,j\in I$ with $a_{i,j}\le 0$,
 the Serre-like relations for $(Y(\check{e}_{i}^\pm,z), Y(\check{e}_{j}^\pm,z),1-a_{i,j})$ hold.
 On the other hand, for any ${\mathcal{V}}_A(\ell)$-module $(W,Y_W)$, all of  these relations hold on $W$
 by Propositions \ref{pq-va-module-relation} and \ref{prop-2.25} and by Lemma \ref{Y-V-W-a-b-relation}.
 Furthermore, we have
\begin{align}
&Y_W\!\(e^{(1-\ell)\hbar \D}E^{-}(-\check{h}_{i},-2\hbar){\bf 1},x\)\\
=\ & \exp\!\(-G(\partial_x)q^{-\ell\partial_x}Y_W^{+}(\check{h}_{i},x)\)\!
\exp\!\(-G(\partial_x)q^{-\ell\partial_x}Y_W^{-}(\check{h}_{i},x)\)\!.\nonumber
\end{align}}
\end{rem}

\begin{thm}\label{hqva-module-main}
Let $\ell\in \C$. For any restricted $\wh{\mathcal{DY}}(A)$-module $W$ of level $\ell$,  there exists a
${\mathcal{V}}_A(\ell)$-module structure $Y_W(\cdot,z)$ on $W$, which is uniquely determined by
 \begin{align}
 Y_W(\check{e}_{i},z)=x_{i,\Y}^{+}(z),\ \   Y_W(\check{f}_i,z)=x_{i,\Y}^{-}(z),\ \ Y_W(\check{h}_{i},z)=h_{i,\Y}(z)\ \ \text{ for }i\in I.
 \end{align}
On the other hand, for any ${\mathcal{V}}_A(\ell)$-module $(W,Y_W)$,
there exists a restricted $\wh{\mathcal{DY}}(A)$-module structure of level $\ell$ on $W$ such that
 \begin{align}
x_{i,\Y}^{+}(z)=Y_W(\check{e}_i,z),\ \  x_{i,\Y}^{-}(z)=Y_W(\check{f}_i,z),\ \ h_{i,\Y}(z)=Y_W(\check{h}_{i},z)\ \ \text{ for }i\in I.
 \end{align}
\end{thm}

\begin{proof} For the first assertion, set
$$U=\{ \check{e}_i,\  \check{f}_i, \ \check{h}_i\ |\ i\in I\}\cup \{ {\bf 1}\}\subset {\mathcal{V}}_A(\ell),$$
 and define a map $Y_W^0(\cdot,z): U\rightarrow \E_{\hbar}(W)$ by $Y_W^0({\bf 1},z)=1_W$,
 $$Y_W^0(\check{e}_i,z)=x_{i,\Y}^{+}(z),\ \ Y_W^0(\check{f}_i,z)=x_{i,\Y}^{-}(z),\ \ Y_W^0(\check{h}_i,z)=h_{i,\Y}(z)\ \ \text{ for }i\in I.$$
By Proposition \ref{Y-E-mod-struct},  $\<U_W\>$ is a restricted $\wh{\mathcal{DY}}(A)$-module of level $\ell$, containing
$1_W$ as a vacuum vector. Then by Theorem \ref{YD-vacuum-module}
 there exists a $\wh{\mathcal{DY}}(A)$-module homomorphism
 $\psi_z: {\mathcal{V}}_A(\ell)\rightarrow \langle U_W\rangle$ with $\psi_z({\bf 1})=1_W$.
 For $i\in I,\ m\in \Z,\ v\in {\mathcal{V}}_A(\ell)$, we have
  \begin{align*}
&\psi_z((\check{e}_{i}^{\pm})_mv)= \psi_z(x_{i,\Y}^{\pm}(m)v)=x_{i,\Y}^{\pm}(m)\psi_z(v)=x_{i,\Y}^{\pm}(z)_{m}\psi_z(v),\\
 &\psi_z((\check{h}_{i})_mv)=h_{i,\Y}(z)_m\psi_z(v).
 \end{align*}
Then the first assertion follows from Theorem \ref{hwqva-module}.
The second assertion follows from Remark \ref{DYhqva-relations} and Theorem \ref{main-DY}.
\end{proof}

\begin{de}
Define $\widehat{\mathcal{L}}_A$ to be the Lie algebra generated by set
\begin{align}\label{eq:Lie-gen-set}
  \set{h_{i,m},\,e_{i,m}^\pm}{i\in I,\  m\in\Z}\cup \{ {\bf k}\}
\end{align}
with ${\bf k}$, subject to the following relations for $i,j\in I$:
\begin{align*}
 \te{(L1)}\quad& [h_i(z),h_j(w)]=r_ia_{i,j}\pd{w}z\inv\delta\!\(\frac{w}{z}\){\bf k},\\
  \te{(L2)}\quad& [h_i(z),e_j^\pm(w)]=\pm r_ia_{i,j}e_j^\pm (w)z\inv\delta\!\(\frac{w}{z}\)\!,\\
  \te{(L3)}\quad& [e_i^+(z),e_j^-(w)]=\frac{\delta_{i,j}}{r_i}
  \(\!h_i(w)z\inv\delta\!\(\frac{w}{z}\)
  +{\bf k}\pd{w}z\inv\delta\!\(\frac{w}{z}\)\)\!,\\
  \te{(L4)}\quad&[e_i^\pm(z),e_j^\pm(w)]=0\quad\te{if }a_{i,j}\ge 0,\\
  \te{(L5)}\quad&(z-w)[e_i^\pm(z),e_j^\pm(w)]=0\quad\te{if }a_{i,j}<0, \\
  \te{(L6)}\quad&[e_i^\pm(z_1),[e_i^\pm(z_2), \dots [e_i^\pm(z_{1-a_{i,j}}), e_j^\pm(w)]\cdots ]]=0\quad\te{ if }a_{i,j}< 0,
\end{align*}
where
\begin{align*}
  &h_i(z)=\sum_{m\in\Z}h_{i,m}z^{-m-1},\quad
  e_i^\pm(z)=\sum_{m\in\Z}e_{i,m}^\pm z^{-m-1}.
\end{align*}
\end{de}

\begin{lem}
The quotient space ${\mathcal{V}}_A(\ell)/\hbar {\mathcal{V}}_A(\ell)$ is an $\wh{\mathcal{L}}_A$-module of level $\ell$ with
$$e_i^\pm(z)=Y(\check{e}_{i}^{\pm},z),\ \ h_i(z)=Y(\check{h}_{i},z)\ \ \text{ for }i\in I.$$
 Furthermore, ${\mathcal{V}}_A(\ell)/\hbar {\mathcal{V}}_A(\ell)$ as
 an $\wh{\mathcal{L}}_A$-module is generated by vector ${\bf 1}+\hbar {\mathcal{V}}_A(\ell)$.
\end{lem}

\begin{proof} For $i\in I$, as
 \begin{align*}
  &q^{(1-\ell)\D}E^{-}(-\check{h}_i,-2\hbar){\bf 1}={\bf 1}-2\hbar \check{h}_{i}+O(\hbar^2),\\
 &z\inv\delta\!\(\frac{w-2\ell\hbar}{z}\)\!=z\inv\delta\!\(\frac{w}{z}\)-2\ell \hbar \frac{\partial}{\partial w}z\inv\delta\!\(\frac{w}{z}\)+O(\hbar^2),
 \end{align*}
 we have
  \begin{align*}
  & z\inv\delta\!\(\frac{w-2\ell\hbar}{z}\)\! q^{(1-\ell)\D}E^{-}(-\check{h}_{i},-2\hbar){\bf 1}\\
  =\ &z\inv\delta\!\(\frac{w}{z}\)({\bf 1}-2\hbar\check{h}_{i})-2\hbar \ell \frac{\partial}{\partial w}z\inv\delta\!\(\frac{w}{z}\)+O(\hbar^2).
   \end{align*}
  This implies that relation (L3) holds on ${\mathcal{V}}_A(\ell)/\hbar {\mathcal{V}}_A(\ell)$.
 Note that in the presence of relations (L4) and (L5), (L6) is equivalent to the Serre-like relation.
Then it follows from Remark \ref{DYhqva-relations}
that ${\mathcal{V}}_A(\ell)/\hbar {\mathcal{V}}_A(\ell)$ is an $\wh{\mathcal{L}}_A$-module of level $\ell$.
 It is clear that ${\mathcal{V}}_A(\ell)/\hbar {\mathcal{V}}_A(\ell)$ is generated by
 vector ${\bf 1}+\hbar {\mathcal{V}}_A(\ell)$.
 \end{proof}




{\em Assume that $A$ is of finite type.} Set $\g=\g(A)$.
Then $\widehat{\mathcal{L}}_A$ is isomorphic to $\wh{\g}$ (see \cite{Garland}).
For $\ell\in \C$, let $V_{\wh{\g}}(\ell,0)$ be the universal vacuum $\wh{\g}$-module of level $\ell$,
which is naturally a vertex algebra.
It follows that there exists a surjective $\wh{\g}$-module homomorphism
 \begin{align}
 \phi:\  V_{\wh{\g}}(\ell,0)\rightarrow {\mathcal{V}}_A(\ell)/\hbar {\mathcal{V}}_A(\ell)\ \ \text{ with }
 \phi({\bf 1})={\bf 1}+\hbar \mathcal{V}_A(\ell).
 \end{align}
 Furthermore,  $\phi$ is a vertex algebra homomorphism.
 From \cite{Li-qva2} (cf. \cite{EK-qva}), $V_{\wh{\g}}(\ell,0)$ is non-degenerate.
 On the other hand, if $\ell$ is generic, $V_{\wh{\g}}(\ell,0)$ is a simple vertex algebra.
 With all of these facts, we immediately have:

\begin{prop}
Let $A$ be of finite type and let $\ell\in \C$. Assume that the $\wh{\g}$-module homomorphism $\phi$ defined above
is also injective, which is the case  if $\ell$ is generic and if ${\mathcal{V}}_A(\ell)\ne 0$.
Then ${\mathcal{V}}_A(\ell)$ is a non-degenerate $\hbar$-adic quantum vertex algebra
 with ${\mathcal{V}}_A(\ell)/\hbar {\mathcal{V}}_A(\ell)\cong V_{\wh{\g}}(\ell,0)$ as a vertex algebra.
\end{prop}

\section*{Appendix, Two technical results}
 Let $W=W_0[[\hbar]]$ be a topologically free $\C[[\hbar]]$-module.
Note that $(\End W)[[x,x^{-1}]]$ is naturally a $\C[x,x^{-1}][[\hbar]]$-module. Recall that
$$W_{\hbar}((x))=W_0((x))[[\hbar]].$$

\begin{lem}\label{restricted-V-g-1}
Let $a(x),b(x), c(x)\in (\End W)[[x,x^{-1}]],\ \Psi_1(z),\Psi_2(z)\in \C[[z]],\ k,k'\in \Z_{+}$, such that
\begin{align}
[a(x_1),b(x_2)]=c(x_2)\(\Psi_1(\hbar\partial_{x_2})(x_1-x_2)^{-k}+\Psi_2(\hbar\partial_{x_2})(x_2-x_1)^{-k'}\)\!.
\end{align}
Assume $w\in W$ such that $a(x)w, b(x)w, c(x)w\in W_{\hbar}((x))$. Then $a(x)b_mw\in W_{\hbar}((x))$ for every $m\in \Z$,
where $b(x)=\sum_{m\in \Z}b_mx^{-m-1}$.
\end{lem}

\begin{proof} Let $m\in \Z$. We have
\begin{align*}
a(x_1)b_mw
=\ &b_ma(x_1)w+\Res_{x_2}x_2^mc(x_2)w\(\Psi_2(\hbar\partial_{x_2})(x_2-x_1)^{-k'}\)\\
&+\Res_{x_2}x_2^mc(x_2)w \(\Psi_1(\hbar\partial_{x_2})(x_1-x_2)^{-k}\)\!.
\end{align*}
Notice that the first two terms on the right-hand side lie in $W_{\hbar}((x_1))$. Now, consider the third term.
Let $n\in \Z_{+}$. We have
$$ \Psi_1(\hbar\partial_{x_2})(x_1-x_2)^{-k}\equiv \sum_{i=0}^{n-1}\mu_i (x_1-x_2)^{-k-i} \hbar^i\
\mod \hbar^n\C[x_1^{-1}][[x_2,\hbar]],$$
where $\mu_i\in \C$. As $c(x)w\in W_{\hbar}((x))$, there exists $r\in \N$ such that $c_{m+j}w \in \hbar^n W$ for all $j> r$.
Then we get
\begin{align*}
&\Res_{x_2}x_2^mc(x_2)w \(\Psi_1(\hbar\partial_{x_2})(x_1-x_2)^{-k}\)\\
\equiv \ &
 \sum_{i=0}^{n-1}\sum_{j=0}^r \mu_i  \hbar^i \binom{-k-i}{j}(-1)^jx_1^{-k-i-j}c_{m+j}w\ \ \mod \hbar^nW[[x_1,x_1^{-1}]]\\
  \in \ & x_1^{-k}W[x_1^{-1}].
 \end{align*}
 This proves $\Res_{x_2}x_2^mc(x_2)w \(\Psi_1(\hbar\partial_{x_2})(x_1-x_2)^{-k}\)\in W_{\hbar}((x_1))$.
 Therefore, we have $a(x)b_mw\in W_{\hbar}((x))$.
\end{proof}

\begin{lem}\label{restricted-V-g-2}
Let
$$a(x),b(x)\in (\End W)[[x,x^{-1}]], \ p(x,y),q(x,y)\in \C[x,y]$$
 such that $p(x,0)=x^l$ with $l\in \N$ and
\begin{align}\label{pab=qba}
p(x_1-x_2,\hbar)a(x_1)b(x_2)=q(x_2-x_1,\hbar)b(x_2)a(x_1).
\end{align}
Let $w\in W$ such that $a(x)w, b(x)w\in W_{\hbar}((x))$. Then $a(x)b_mw\in W_{\hbar}((x))$ for all $m\in \Z$.
\end{lem}

\begin{proof} We must prove that for any $m\in \Z$ and for any $n\in \Z_{+}$,
$\pi_n(a(x)b_mw)\in (W/\hbar^nW)((x))$. Let $n$ be arbitrarily fixed.
From assumption, we have
$$p(x_1-x_2,\hbar)\pi_n(a(x_1)b(x_2)w)=q(x_2-x_1,\hbar)\pi_n(b(x_2)a(x_1)w),$$
$$\pi_n(a(x)w), \pi_n(b(x)w)\in (W/\hbar^nW)((x)).$$
Notice that
\begin{eqnarray*}
&&p(x_1-x_2,\hbar)\pi_n(a(x_1)b(x_2)w)\in (W/\hbar^nW)[[x_1,x_1^{-1}]]((x_2)),\\
&&q(x_2-x_1,\hbar)\pi_n(b(x_2)a(x_1)w)\in (W/\hbar^nW)[[x_2,x_2^{-1}]]((x_1)).
\end{eqnarray*}
Consequently, we get
\begin{align}\label{p-pi-n-a-b}
p(x_1-x_2,\hbar)\pi_n(a(x_1)b(x_2)w)\in (W/\hbar^nW)((x_1,x_2)),
\end{align}
while
$$\pi_n(a(x_1)b(x_2)w)\in (W/\hbar^nW)[[x_1,x_1^{-1}]]((x_2)).$$
As $p(x,0)=x^l$ is an invertible element of $\C[x,x^{-1}]$, $p(x,\hbar)$ is an invertible element of $\C[x,x^{-1}][[\hbar]]$, so that
$p(x_1-x_2,\hbar)^{-1}$ exists in $\C[x_1,x_1^{-1}][[x_2,\hbar]]$.
Then multiplying (\ref{p-pi-n-a-b}) by $p(x_1-x_2,\hbar)^{-1}$ we obtain
$$\pi_n(a(x_1)b(x_2)w)\in (W/\hbar^nW)((x_1))((x_2)).$$
This proves that $\pi_n(a(x)b_mw) \in (W/\hbar^nW)((x))$ for every $m\in \Z$, as desired.
\end{proof}

\end{document}